\documentclass[11pt]{amsart}
\title{On higher scissors congruence}
\author{Cary Malkiewich}
\address{Department of Mathematics, Binghamton University}
\email{cmalkiew@binghamton.edu}

\usepackage{amsmath,amssymb,amsthm,array,color,multirow,pdflscape,mathrsfs,subcaption}
\usepackage[all]{xypic}
\usepackage{tikz,tikz-cd}

\usepackage[margin=1.25in]{geometry}
\makeatletter
\def\l@section{\@tocline{1}{0pt}{1pc}{}{}}
\def\l@subsection{\@tocline{2}{0pt}{1pc}{4.6em}{}}

\def\l@subsubsection{\@tocline{3}{0pt}{1pc}{7.6em}{}}
\renewcommand{\tocsection}[3]{%
	\indentlabel{\@ifnotempty{#2}{\makebox[2.3em][l]{%
				\ignorespaces#1 #2.\hfill}}}#3}
\renewcommand{\tocsubsection}[3]{%
	\indentlabel{\@ifnotempty{#2}{\hspace*{2.3em}\makebox[2.3em][l]{%
				\ignorespaces#1 #2.\hfill}}}#3}
\renewcommand{\tocsubsubsection}[3]{%
	\indentlabel{\@ifnotempty{#2}{\hspace*{4.6em}\makebox[3em][l]{%
				\ignorespaces#1 #2.\hfill}}}#3}
\makeatother
\usepackage{xcolor}
\definecolor{darkgreen}{rgb}{0,0.30,0} 
\definecolor{darkred}{rgb}{0.75,0,0}
\definecolor{darkblue}{rgb}{0,0,0.6} 
\usepackage[pdfborder=0,pagebackref,colorlinks,citecolor=darkgreen,linkcolor=darkgreen,urlcolor=darkblue]{hyperref}
\renewcommand*{\backref}[1]{}
\renewcommand*{\backrefalt}[4]{({%
    \ifcase #1 Not cited.%
          \or On p.~#2%
          \else On pp.~#2%
    \fi%
    })}

\def\makeautorefname#1#2{\expandafter\def\csname#1autorefname\endcsname{#2}}
\makeautorefname{equation}{Equation}%
\makeautorefname{footnote}{footnote}%
\makeautorefname{item}{item}%
\makeautorefname{figure}{Figure}%
\makeautorefname{table}{Table}%
\makeautorefname{part}{Part}%
\makeautorefname{appendix}{Appendix}%
\makeautorefname{chapter}{Chapter}%
\makeautorefname{section}{Section}%
\makeautorefname{subsection}{Section}%
\makeautorefname{subsubsection}{Section}%
\makeautorefname{paragraph}{Paragraph}%
\makeautorefname{subparagraph}{Paragraph}%
\makeautorefname{theorem}{Theorem}%
\makeautorefname{thm}{Theorem}%
\makeautorefname{addm}{Addendum}%
\makeautorefname{mainthm}{Main theorem}%
\makeautorefname{corollary}{Corollary}%
\makeautorefname{cor}{Corollary}%
\makeautorefname{lemma}{Lemma}%
\makeautorefname{lem}{Lemma}%
\makeautorefname{sublemma}{Sublemma}%
\makeautorefname{sublem}{Sublemma}%
\makeautorefname{subl}{Sublemma}%
\makeautorefname{prop}{Proposition}%
\makeautorefname{prob}{Problem}%
\makeautorefname{property}{Property}
\makeautorefname{pro}{Property}
\makeautorefname{sch}{Scholium}%
\makeautorefname{step}{Step}%
\makeautorefname{conject}{Conjecture}%
\makeautorefname{conj}{Conjecture}%
\makeautorefname{questn}{Question}
\makeautorefname{quest}{Question}
\makeautorefname{qn}{Question}
\makeautorefname{definition}{Definition}%
\makeautorefname{defn}{Definition}%
\makeautorefname{defi}{Definition}%
\makeautorefname{def}{Definition}%
\makeautorefname{dfn}{Definition}%
\makeautorefname{df}{Definition}%
\makeautorefname{notation}{Notation}
\makeautorefname{notn}{Notation}
\makeautorefname{rem}{Remark}%
\makeautorefname{rems}{Remarks}%
\makeautorefname{rmk}{Remark}%
\makeautorefname{rk}{Remark}%
\makeautorefname{remarks}{Remarks}%
\makeautorefname{rems}{Remarks}%
\makeautorefname{rmks}{Remarks}%
\makeautorefname{rks}{Remarks}%
\makeautorefname{example}{Example}%
\makeautorefname{examp}{Example}%
\makeautorefname{exmp}{Example}%
\makeautorefname{exam}{Example}%
\makeautorefname{exa}{Example}%
\makeautorefname{ex}{Example}%
\makeautorefname{axiom}{Axiom}%
\makeautorefname{axi}{Axiom}%
\makeautorefname{ax}{Axiom}%
\makeautorefname{case}{Case}%
\makeautorefname{claim}{Claim}%
\makeautorefname{clm}{Claim}%
\makeautorefname{assumpt}{Assumption}%
\makeautorefname{asses}{Assumptions}%
\makeautorefname{conclusion}{Conclusion}%
\makeautorefname{concl}{Conclusion}%
\makeautorefname{conc}{Conclusion}%
\makeautorefname{cond}{Condition}%
\makeautorefname{const}{Construction}%
\makeautorefname{con}{Construction}%
\makeautorefname{criterion}{Criterion}%
\makeautorefname{criter}{Criterion}%
\makeautorefname{crit}{Criterion}%
\makeautorefname{exercise}{Exercise}%
\makeautorefname{exer}{Exercise}%
\makeautorefname{exe}{Exercise}%
\makeautorefname{warn}{Warning}%
\newtheorem{thm}{Theorem}[section]
\newtheorem{cor}{Corollary}[section]
\newtheorem{lem}{Lemma}[section]
\newtheorem{prop}{Proposition}[section]
\newtheorem{prob}{Problem}[section]

\theoremstyle{definition}
\newtheorem{df}{Definition}[section]
\newtheorem{ex}{Example}[section]

\newtheorem{rmk}{Remark}[section]
\newtheorem{notn}{Notation}[section]
\numberwithin{equation}{section}
\numberwithin{figure}{section}

\makeatletter
\let\c@cor=\c@thm
\let\c@prop=\c@thm
\let\c@prob=\c@thm
\let\c@lem=\c@thm
\let\c@df=\c@thm
\let\c@ex=\c@thm
\let\c@warn=\c@thm
\let\c@rmk=\c@thm
\let\c@notn=\c@thm
\let\c@equation\c@thm
\let\c@figure\c@thm
\let\c@table\c@thm
\makeatother

\hypersetup{linktoc=all}

\newcommand{\bI}{\mathbf{I}}
\newcommand{\bJ}{\mathbf{J}}

\newcommand{\R}{\mathbb R}

\newcommand{\Z}{\mathbb Z}
\newcommand{\Q}{\mathbb Q}
\newcommand{\Sph}{\mathbb S}
\newcommand{\mc}{\mathcal}
\newcommand{\cat}[1]{\textup{\textbf{{#1}}}}

\newcommand{\Map}{\textup{Map}}

\newcommand{\colim}{\textup{colim}\,}
\newcommand{\hocolim}{\textup{hocolim}\,}
\newcommand{\uhocolim}{\textup{hocolim}^u\,}
\newcommand{\tcofib}{\textup{tcofib}\,}

\newcommand{\sma}{\wedge}

\newcommand{\ti}{\widetilde}
\newcommand{\op}{\textup{op}}

\newcommand{\sd}{\textup{sd}}

\newcommand{\Th}{\textup{Th}}

\newcommand{\po}[2]{\ar@{}@<{#2}>[rd]|({#1})*\txt{\Large $\ulcorner$}}
\newcommand{\pb}[2]{\ar@{}@<{#2}>[rd]|({#1})*\txt{\Large $\lrcorner$}}

\newcommand{\sbt}{\,\begin{picture}(-1,1)(0.5,-1)\circle*{1.8}\end{picture}\hspace{.05cm}}

\newcommand{\Fin}{\mathrm{Fin}}

\DeclareMathOperator{\PT}{PT}
\DeclareMathOperator{\ST}{ST}
\DeclareMathOperator{\CT}{CT}
\DeclareMathOperator{\Pt}{Pt}
\DeclareMathOperator{\St}{St}

\DeclareMathOperator{\Tpl}{Tpl}

\DeclareMathOperator{\spa}{span}

\DeclareMathOperator{\T}{T}
\DeclareMathOperator{\GL}{GL}
\DeclareMathOperator{\coker}{coker}

\newcommand{\Pol}[2]{\mathcal P^{#1}_{\hspace{.1em} #2}}

\usepackage{todonotes}

\begin{document}

\begin{abstract}
	We solve the higher version of Hilbert's Third Problem for one-dimensional geometries, and in higher dimensions we reduce the problem to a computation in group homology. Our central result concerns the scissors congruence $K$-theory spectrum of Zakharevich, whose homotopy groups are the correct higher version of the classical scissors congruence groups. We prove that this spectrum is a Thom spectrum, whose base space is the homotopy orbit space of a Tits complex. The relevant computations quickly follow from this more foundational result.
\end{abstract}

\maketitle

\setcounter{tocdepth}{1}
\tableofcontents

\parskip 2.5ex

\section{Introduction}

We say that two Euclidean polygons $P$ and $Q$ are {\bf scissors congruent} if $P$ can be cut into finitely many pieces and rearranged to form $Q$. This clearly implies that they have the same area. However, the converse is also true: if $P$ and $Q$ have the same area then they are scissors congruent \cite{wallace_bolyai_gerwien}.

{\bf Hilbert's Third Problem} asks the same question in dimension three: given two Euclidean polyhedra $P$ and $Q$, when are they scissors congruent? Is it enough to check that they have the same volume? Dehn famously shows that this is false: $P$ and $Q$ must also have the same Dehn invariant, in addition to having the same volume \cite{dehn}. Later Sydler shows that the volume and Dehn invariant suffice to establish scissors congruence \cite{sydler}, finishing the solution to this form of Hilbert's Third Problem.

Hilbert's Third Problem can be considered as a question not just about three-dimensional Euclidean polytopes, but about polytopes in Euclidean, hyperbolic, or spherical geometry in any fixed dimension. When are $P$ and $Q$ scissors congruent? As already mentioned, this version of the problem has been solved in Euclidean geometry in dimension up to three, and in fact, Jessen solves the problem for four-dimensional Euclidean geometry as well \cite{jessen72}. But in dimensions five and above, the problem remains open. In hyperbolic and spherical geometry, the problem is also solved up to dimension two, and though we know a great deal about dimension three, see e.g.~\cite[Thm 1.7]{dupont_book}, the problem is still open in this case as well. This is related to the open question of whether the volumes of hyperbolic 3-manifolds are all commensurable, see \cite[3.9]{neumann_notes} and \cite{thurston_problems}.

It is fruitful to stipulate that, when rearranging the pieces of $P$ to form $Q$, we only allow moves lying in some fixed subgroup $G$ of the isometry group $I(X)$ of the geometry $X$. We say in this case that $P$ and $Q$ are {\bf $G$-scissors congruent}. This form of the problem is especially interesting in the Euclidean case, where we may take $G$ to be the subgroup of translations; a complete solution is then given in all dimensions by Jessen and Thorup \cite{jessen_thorup}, see also \cite{dupont_82,dupont_book}.

Most of these results on scissors congruence, particularly the ones from the last 50 years, are driven by casting the problem in algebraic terms. The central object of study is an abelian group called the {\bf scissors congruence group} $\mc P(X,G) = K_0(\Pol{X}{G})$. It is defined as the free abelian group on the polytopes $[P]$ in $X$, modulo relations
\begin{itemize}
	\item $[P] = \sum_i [P_i]$ when the $P_i$ have disjoint interiors and cover $P$, and
	\item $[P] = [gP]$ for every $g \in G$.
\end{itemize}
Alternatively, let the polytope module $\Pt(X)^t$ be the quotient by just the first set of relations. Then $K_0$ is the zeroth homology or coinvariants group \vspace{.5em}
\begin{equation}\label{intro_homology_eq}
	K_0(\Pol{X}{G}) = H_0(G;\Pt(X)^t).
\end{equation}
\vspace{-1.5em}

The polytopes $P$ and $Q$ are $G$-scissors congruent iff $[P] = [Q]$ in this abelian group, at least when $G$ acts transitively on $X$ (\cite{zylev_65}, \cite[\S 1.3]{sah_79}). Therefore, solving the scissors congruence problem amounts to computing the abelian group $K_0(\Pol{X}{G})$. Furthermore, the expression \eqref{intro_homology_eq} allows for algebraic approaches using group homology. See for example  \cite{jessen_karpf_thorup,jessen_thorup,sah_79,dupont_82,ds1,dupont_parry_sah,ds2,dupont_book,cath1,cz}.

\begin{rmk}\label{intro_twist}
	The $t$ decoration on $\Pt(X)$ indicates that the action of $g \in G$ sends $[P]$ to $[gP]$. This is considered a ``twisted'' action; the untwisted action is the one that sends $[P]$ to $\pm [gP]$, the sign depending on whether $g$ preserves orientation. It is not hard to show that the coinvariants under the untwisted action are always zero, see \cite[Thm 2.2]{dupont_book}.
\end{rmk}

As the notation suggests, $K_0(\Pol{X}{G})$ is the first of a natural sequence of groups $K_n(\Pol{X}{G})$, the {\bf higher scissors congruence groups}. They are defined as the homotopy groups of the scissors congruence $K$-theory spectrum $K(\Pol{X}{G})$ defined by Zakharevich \cite{zak_assemblers}. These higher groups capture information, not just about polytopes up to scissors congruence, but about the \emph{scissors congruences themselves}. Hilbert's Third Problem is therefore a special case of the following.

\begin{prob}[Higher scissors congruence problem]\label{intro_problem}
	Compute or geometrically describe the higher scissors congruence groups $K_m(\Pol{X}{G})$ for all geometries $X$ and subgroups $G \leq I(X)$.
\end{prob}

This problem has been previously studied (using different language) for $K_1$ of one-dimensional geometries, where it becomes a question about the group of interval exchange transformations. See \cite{saf1,saf2,saf3,zak_assemblers,KLMMS-1}.

This is a deep and difficult problem, especially when we move above $K_0$ and $K_1$. One issue is that the tools we have for computing higher $K$-groups do not apply to the setting of polytopes, because the $K$-theory of polytopes is a form of ``cut-and-paste'' $K$-theory. This does not fit into the classical framework for $K$-theory of a category with a notion of an exact sequence, such as an exact category, a Waldhausen category, or a stable $\infty$-category.

We may also illustrate the relative importance of \autoref{intro_problem} by comparing it to the higher algebraic $K$-theory of fields, for which there is a rich literature, see e.g.~\cite{weibel} for a survey. Higher scissors congruence is to Hilbert's Third Problem what the higher algebraic $K$-groups are to linear algebra---a substantially deeper invariant with connections to several related fields. Indeed, even $K_0$ is related to homological stability for orthogonal groups \cite{cath3}, volumes of hyperbolic manifolds \cite{goncharov}, characteristic classes of flat bundles \cite{dupont_book,cz}, and conjecturally, mixed Tate motives \cite{goncharov,rudenko}. Many of these relationships extend to the higher scissors congruence groups as well.

Now that we have motivated \autoref{intro_problem}, we turn to our results. In this paper we solve the higher scissors congruence problem for each of the one-dimensional geometries, when $G$ is the full isometry group or the subgroup of orientation-preserving isometries. Our results agree on $K_1$ with those of Zakharevich in \cite[Sec 4]{zak_k1}, which come from an earlier result of Sah \cite{saf1}.
\begin{thm}\label{intro_dim_one}
	The higher scissors congruence groups for the Euclidean line $E^1$ and the circle $S^1$, when $G$ is the group of orientation-preserving isometries, are given by
	\[ K_n \cong \R^{\wedge (n+1)}. \]
	When $G$ is the full isometry group, the answer is in both cases
	\[ K_{2n} = \R^{\wedge (2n+1)}, \qquad K_{2n+1} = 0. \]
	
\end{thm}
Here $\wedge$ denotes exterior product over $\Z$ or $\Q$. Since $H^1 \cong E^1$, these cases exhaust all of the standard one-dimensional geometries.

Above dimension one, we reduce \autoref{intro_problem} to a question about group homology.

\begin{thm}\label{intro_homology}
	For any geometry $X$ and subgroup $G \leq I(X)$, the spectrum homology of scissors congruence $K$-theory is expressed as
	\[ H_m(K(\Pol{X}{G})) \cong H_m(G; \Pt(X)^t).\]
	In particular we get a rational isomorphism
	\[ K_m(\Pol{X}{G}) \otimes \Q \cong H_m(G; \Pt(X)^t) \otimes \Q. \]
\end{thm}
It turns out that $\Pt(E^n)$ is isomorphic to the Steinberg module of the affine Tits building, which we denote $\St(E^n)$, see \autoref{pt}.

Note that \autoref{intro_homology} is a direct generalization of \eqref{intro_homology_eq} to the higher scissors congruence groups. As a result, all of the classical work on scissors congruence based on group homology, from e.g.~\cite{sah_79,ds1,dupont_parry_sah,dupont_book,cath1}, can now be applied to the higher scissors congruence groups as well. The reader may also note the conceptual similarity (and technical differences) between this result and the one that relates the higher algebraic $K$-theory of fields to the homology of the general linear group:
\[ H_m(\Omega^\infty K(F)) \cong H_m(GL_\infty(F)). \]

 We also prove a rationality result for Euclidean geometry:

\begin{thm}\label{intro_translation_homology}
	In the Euclidean case $X = E^n$, if $G$ contains all translations, then $K_m(\Pol{E^n}{G})$ is rational and the above expression holds integrally,
	\[ K_m(\Pol{E^n}{G}) \cong H_m(G; \Pt(E^n)^t). \]
\end{thm}
This extends the known result that the lowest group $K_0(\Pol{E^n}{G})$ is rational when $G$ contains all translations. The rationality of the scissors congruence groups $\mc P(X) = K_0(\Pol{X}{I(X)})$ for non-Euclidean geometries is an open problem in general.

\autoref{intro_translation_homology} demonstrates that the higher scissors congruence groups contain substantial arithmetic information about the real numbers. Indeed, this formula for the higher scissors congruence groups is closely related to Quillen's filtration of $\Omega^{\infty - 1}K(R)$ from \cite{quillen_fg}, which was recently used to show that $K_8(\Z) = 0$ \cite{sevkm}.

The group homology in \autoref{intro_homology} is very difficult to compute, as evidenced by the fact that we only know $K_0$ for the geometries $H^1$, $H^2$, $S^1$, $S^2$, and $E^1$ through $E^4$. Still, the above theorems provide ample motivation to extend these calculations further. It may be that the higher $K$-groups of low-dimensional geometries are more tractable than $K_0$ of $E^5$.

We next describe the central result that makes these computations possible. Define the {\bf polytopal Tits complex} $\PT(X)$ to be the homotopy quotient of the geometry $X$ by the homotopy colimit of all its proper subspaces. This is a space whose reduced homology is concentrated in a single degree, where it gives the abelian group $\Pt(X)$. This is a variant of the Solomon-Tits theorem \cite{solomon_tits} that we prove in \autoref{our_st}.

The tangent bundle $TX$ of the geometry $X$ passes to a vector bundle on $\PT(X) \setminus \{*\}$, which is enough to define a reduced Thom spectrum
\[ \PT(X)^{-TX} = \Sigma^{-TX}\PT(X). \]
Intuitively, this is a formal de-suspension of the based space $\PT(X)$, where the suspensions are twisted according to the vector bundle $TX$. The isometry group $G$ (as a discrete group) acts on this in a natural way, and the homotopy orbits are again a reduced Thom spectrum, which we denote $\PT(X)_{hG}^{-TX}$. Our main theorem can then be stated as follows.

\begin{thm}\label{intro_main}
	If $X$ is any geometry and $G \leq I(X)$ is any subgroup, there is an equivalence of spectra
	\[ K(\Pol{X}{G}) \simeq \PT(X)_{hG}^{-TX}. \]
\end{thm}

Let $\T(X)$ denote the Tits complex, in other words the realization of the poset of proper nonempty subspaces of $X$. Then the space $\PT(X)$ has a canonical map to the unreduced suspension $\ST(X)$:
\begin{equation}\label{intro_pt_to_st}
	\PT(X) \to \ST(X).
\end{equation}
In the Euclidean and hyperbolic cases, the map \eqref{intro_pt_to_st} is an equivalence, and therefore
\[ K(\Pol{X}{G}) \simeq \ST(X)_{hG}^{-TX}. \]
In the spherical case, the map \eqref{intro_pt_to_st} is not an equivalence, but the spectrum $\ST(X)_{hG}^{-TX}$ still has an important role to play. We define the reduced spherical $K$-theory spectrum $\ti K(\Pol{S(W)}{1})$ to be $K(\Pol{S(W)}{1})$ modulo the homotopy colimit of the spectra $K(\Pol{S(V)}{1})$ for $V \subseteq W$, along maps that suspend the polytopes. The definition for larger groups $G$ is similar. (See \autoref{sec:reduced}.) At $K_0$, this recovers the reduced spherical scissors congruence groups $\ti{\mc P}(S(W),G)$ that appear in the definition of the Dehn invariant, see e.g.~\cite{sah_79,dupont_book,cath1}.
\begin{thm}\label{intro_reduced}
	For any inner product space $W \neq 0$ and any subgroup $G \leq O(W)$, there is an equivalence of spectra
	\[ \ti K(\Pol{S(W)}{G}) \simeq \ST(S(W))_{hG}^{1-W}. \]
\end{thm}
By definition, we have $\ti K(\Pol{S(0)}{G}) = K(\Pol{S(0)}{G}) \simeq \Sph$. See \autoref{sec:examples}.

The proof of the main theorem (\autoref{intro_main}) can be sketched as follows. We first consider the case of the trivial subgroup $G = 1$. We use the Barratt-Priddy-Quillen theorem and a filtration argument to establish a Solomon-Tits theorem for the $K$-theory spectrum:
\begin{thm}[M--Zakharevich]\label{intro_st}
	For any geometry $X$, the spectrum $K(\Pol{X}{1})$ is (non-equivariantly) equivalent to a wedge of sphere spectra $\Sph$.
\end{thm}
We highlight that this first step of the argument is a joint result with Zakharevich. We then use this to define a second, $G$-equivariant equivalence
\[ K(\Pol{X}{1}) \simeq \PT(X)^{-TX}. \]
This uses a combination of Pontryagin-Thom collapse maps and a homotopy-coherent refinement of the notion of an ``apartment'' from the theory of Steinberg modules, see e.g.~\cite{solomon_tits,lee_szczarba,church_putman_one,church_farb_putman_integrality}. Finally, we invoke a result from \cite{bgmmz}, that scissors congruence $K$-theory is a homotopy orbit spectrum.

\begin{thm}[Bohmann-Gerhardt-M-Merling-Zakharevich]\label{farrell_jones}
There is an equivalence of spectra
	\[ K(\Pol{X}{1})_{hG} \simeq K(\Pol{X}{G}) \]
	for every geometry $X$ and every subgroup $G \leq I(X)$.
\end{thm}

Combining these together gives \autoref{intro_main}. The reduced spherical case (\autoref{intro_reduced}) follows a similar but more elaborate argument.

We note once more a conceptual relationship to the higher $K$-theory of fields: the above two theorems imply that scissors congruence $K$-theory is the homotopy orbits of a wedge of sphere spectra. The same was recently proven about the stable rank filtration of higher $K$-theory of fields from \cite{rognes_rank}, partially resolving a longstanding conjecture by Rognes \cite{miller_patzt_wilson}.

\begin{rmk}\label{automorphisms}
	Since the first version of this paper was posted, it has become apparent that \autoref{intro_dim_one} is closely related to the homology of the group of interval exchange transformations, which was computed more recently in \cite{tanner} using \cite{li,szymik_wahl}. In particular, Robin Sroka pointed out a striking similarity between \autoref{intro_dim_one} and \cite[Lem 5.6]{tanner}. In general, the higher scissors congruence groups are closely related to the group of ``scissors automorphisms,'' that is, self-maps of a polytope that cut it into pieces and then rearrange the pieces using elements of the group $G$. See the more recent paper \cite{KLMMS-1} for details.
\end{rmk}

\begin{rmk}\label{rt_building}
	The space $\ST(X)$ agrees with the RT-building $F^X_{\sbt}$ used by Campbell and Zakharevich \cite{cz} in their refinement of Goncharov's conjecture concerning the Dehn complex and motivic cohomology \cite{goncharov,rudenko}. Campbell and Zakharevich use an isomorphism at the level of $K_0$
	\[ K_0(\Pol{X}{G}) \cong H_{n+1}(\ST(X)^{\sigma}_{hG}), \]
	where $\sigma$ is the sign representation, and we reduce $K$-theory in the spherical case. Our results imply that the higher homology of their space $\ST(X)^{\sigma}_{hG}$ also captures the higher $K$-groups, at least rationally. Therefore the construction in \cite[\S 2]{cz} can be used to extend the classical Dehn invariant to the higher scissors congruence groups. Details of this construction will appear in \cite{kkmmw1,kkmmw2}.
\end{rmk}

\subsection{Acknowledgements}
The author is pleased to acknowledge that \autoref{intro_st} is joint work with Inna Zakharevich, and to thank Inna for her support in incorporating it into the current paper. He also thanks Inna for conversations, insights, and helpful references on scissors congruence, without which none of the other results would be possible. The author would furthermore like to thank Laura Anderson and Michael Dobbins for helpful conversations about polytopes, and Alexander Kupers, Jeremy Miller, Peter Patzt, and Robin Sroka for enlightening conversations about the connection to the work of Quillen and Rognes. Finally, he thanks Andrew Blumberg, Anna Marie Bohmann, Jonathan Campbell, Teena Gerhardt, Michael Hill, Michael Mandell, Mona Merling, and Kate Ponto for conversations about trace maps for combinatorial forms of $K$-theory, insights from which inspired the current project. The author was supported by the NSF grants DMS-2005524 and DMS-2052923.

\section{Preliminaries}

\subsection{Geometries and polytopes}\label{sec:polytopes}
For each integer $n \geq 0$, let $H^n$, $E^n$, and $S^n$ be defined as subspaces of $\R^{n+1}$:
\begin{align*}
	H^n &= \left\{ \ (x_0,\ldots,x_n) \ : \ -x_0^2 + \sum_{i=1}^n x_i^2 = -1, \ x_0 > 0 \ \right\} \\
%\end{align*}
%\begin{align*}
	E^n &= \Big\{ \ (x_0,\ldots,x_n) \ : \ x_0 = 1 \ \Big\} \\
	S^n &= \left\{ \ (x_0,\ldots,x_n) \ : \ x_0^2 + \sum_{i=1}^n x_i^2 = 1 \ \right\}
\end{align*}
If $X$ is one of these geometries, the isometry group $I(X)$ is the subgroup of $\GL_{n+1}(\R)$ that preserves
\begin{itemize}
	\item The form $-x_0^2 + \sum_{i=1}^n x_i^2$ and the sign of $x_0$, when $X = H^n$, \\[-.8em]
	\item Both $x_0$ and the form $\sum_{i=1}^n x_i^2$ along the subspace $x_0 = 0$, when $X = E^n$, and \\[-.8em]
	\item The form $\sum_{i=0}^n x_i^2$, when $X = S^n$.
\end{itemize}
We denote these groups $H(n)$, $E(n)$, and $O(n+1)$, respectively. For the sphere it is often convenient to replace $\R^{n+1}$ by an arbitrary inner product space $W$, denoting the unit sphere $S(W)$ and its isometry group $O(W)$.

The canonical map $E(n) \to O(n)$ restricts to the subspace $x_0 = 0$. Its kernel is the translation subgroup $T(n) \leq E(n)$, which is isomorphic to $\R^n$ as an abelian group, and we get a split short exact sequence
\[ \xymatrix{
	0 \ar[r] & T(n) \ar[r] & E(n) \ar[r] & O(n) \ar[r] & 0
} \]
making $E(n)$ into a semidirect product $T(n) \rtimes O(n)$.

Let $X$ be any of the above geometries. A {\bf (geometric) subspace} $U \subseteq X$ is a subset obtained by intersecting with a linear subspace of $\R^{n+1}$. If different linear subspaces have the same intersection with $X$ then we consider them to be the same geometric subspace. %Of course, each such subspace can be identified (non-canonically) with a lower-dimensional copy of the same geometry.

A {\bf (geometric) simplex} $\Delta^k \subseteq X$ is obtained from any $(k+1)$ distinct points in $X$ in general position (not all lying in a $(k-1)$-dimensional subspace) by taking the convex hull in $\R^{n+1}$, then projecting from the origin back to $X$. In the spherical case it is necessary to also assume that this convex hull does not contain the origin, in other words the $(k+1)$ points all lie in an open hemisphere. Note that we can move any of these $(k+1)$ points along the ray connecting it to the origin, before taking the convex hull, and it does not change the resulting subset of $X$. %It does change the map from $\Delta^k$ into $X$.

More generally, the {\bf join} $P * Q$ of two subsets $P,Q \subseteq X$ is the quotient of $P \times Q \times I$ by the relations
\[ (p,q,0) \sim (p',q,0), \qquad (p,q,1) \sim (p,q',1) \qquad \forall p,p' \in P, q,q' \in Q. \]
We regard this as a subspace of $X$ by taking $(p,q,t)$ to the point $(1-t)p + tq \in \R^{n+1}$ and then projecting to $X$. We only consider this to be well-defined if the resulting map $P * Q \to X$ is injective, and in the spherical case if every pair $(p,q)$ lies in an open hemisphere. For example, a $k$-simplex is a join of $(k+1)$ points in general position, while if $W = V \oplus V^\perp$ then the entire sphere $S(W)$ is the join of the subspaces $S(V)$ and $S(V^\perp)$.

A {\bf polytope} $P \subseteq X$ is any subset that can be expressed as a finite union of geometric simplices of dimension $n$. So $P$ is \emph{not} required to be convex, but it does need to be top-dimensional. A 1-simplex in $E^2$ is not considered to be a polytope, but a 2-simplex is. %Note that the join of two polytopes, if defined, is always a polytope.

If $P \subseteq X$ is any finite union of simplices, not necessarily top-dimensional, its {\bf span} is the smallest geometric subspace $U$ containing $P$. The corresponding linear subspace of $\R^{n+1}$ is just the span of the points in $P$.

A {\bf weak subdivision} or {\bf almost-disjoint cover} of a polytope $P \subseteq X$ is a finite collection of polytopes $P_i \subseteq X$, whose union is $P$, and whose interiors are disjoint. We think of these as a generalization of linear subdivision, where the pieces of the subdivision do not have to actually form a simplicial or polytopal complex, i.e.~the faces do not have to line up. %We also say that the map from the formal disjoint union $\coprod_i P_i \to P$ is an {\bf almost-disjoint cover}.

\begin{lem}\label{common_refinement}
	Any two weak subdivisions have a common refinement.
\end{lem}

\begin{proof}
	Given two covers $\coprod_i P_i \to P$ and $\coprod_j Q_j \to P$, the pairwise intersections $P_i \cap Q_j$ also form a cover. Throwing out the intersections that are less than top-dimensional gives a cover by top-dimensional polytopes.
\end{proof}

A {\bf (geometric) triangulation} of $P$ is the structure of a simplicial complex on $P$, all simplices being geometric in $X$.
\begin{prop}\label{triangulation_exists}
	Each polytope $P$ has a geometric triangulation. For any finite collection of polytopes $\{ P_i \}$, the union $\cup_i P_i$ can be triangulated so that each $P_i$ is a subcomplex.
\end{prop}

\begin{proof}
	For each $i$, take a finite list of simplices whose union is $P_i$. Each face of each of these simplices determines two closed half-spaces in $X$, one pointing to the interior of the simplex and one pointing away. Take all possible intersections of the half-spaces arising in this way, including intersections of a half-space and the closure of its own complement, which gives a hyperplane. The resulting intersections give a collection of convex polyhedra $\{Q_j\}$ in $X$ of varying dimensions, preserved under taking faces.
	
	We restrict $\{Q_j\}$ to the polyhedra that are bounded, in other words the polytopes that are not necessarily top-dimensional, and make their union into a simplicial complex. For each bounded polyhedron $Q_j$, if its boundary is triangulated, we can extend the triangulation to $Q_j$ by taking the join with one fixed point in the interior. By induction, this gives a triangulation of the union $\cup_j Q_j$ over all the bounded polyhedra in the collection. By construction, it contains each $P_i$ as a subcomplex.
\end{proof}

\begin{rmk}\label{triangulation_common_refinement}
	\autoref{triangulation_exists} implies that any two geometric triangulations of a polytope have a common refinement. (Take $P_i$ to run over all the simplices in both triangulations.) This is not surprising, since in the Euclidean case the triangulations both present the same PL structure on $\R^n$, and so one expects them to be compatible under subdivision. Since the Hauptvermutung is known to be false, if we allow non-geometric triangulations, in other words arbitrary homeomorphisms to simplicial complexes, then common refinements do not exist in general. See e.g.~\cite{haupt_survey}.
\end{rmk}

If $V \subseteq W$ is an inclusion of inner product spaces with orthogonal complement $V^\perp$, and $P \subseteq S(V)$ is any polytope, the {\bf suspension} of $P$ is the join $P * S(V^\perp)$. This is always a polytope in $S(W)$.

Finally, the {\bf empty geometry} is the empty set $\emptyset$. It has a unique polytope, the empty polytope $\emptyset \subseteq \emptyset$.

\subsection{The tuple, polytope, and Steinberg complexes}\label{sec:complexes}

In this subsection we describe the base spaces for our Thom spectra in several equivalent ways.

Let $X$ be any geometry. Following \cite[Def A.1]{cz}, let $\Tpl(X)$ be the simplicial complex in which a $k$-simplex is any $(k+1)$ points in $X$, not necessarily in general position, but lying in an open hemisphere in the spherical case. Let $\ti\Tpl(X)$ be the same complex except that, in the spherical case, the points do not have to lie in an open hemisphere. In \cite{dupont_book} these correspond to the chain complexes $C_.(X)$ and $\overline C_.(X)$, respectively.

\begin{lem}\label{tup_homotopy_type}
	The canonical map $\Tpl(X) \to X$ is an equivalence, and $\ti\Tpl(X)$ is contractible.
\end{lem}

\begin{proof}
	The first claim is a consequence of simplicial approximation, see Lemma 3.6 and (3.9) in \cite{dupont_book} for the homology version. The contractibility of $\ti\Tpl(X)$ is clear because it can be coned off to any one of its points. %, cf. \cite[Lem 1.10]{cz}.
%	Since $E^n$ and $H^n$ are contractible, this shows $\Tpl(X) \to X$ is an equivalence in those cases. The remaining claim that $\Tpl(S^n) \to S^n$ is an equivalence is essentially the nerve theorem (e.g.~\cite[Cor 4G.3]{hatcher_at}) applied to the covering of $S^n$ by open hemispheres. Note that for each open or closed hemisphere, the associated subcomplex of $\Tpl(S^n)$ that is contractible. This allows us to identify $\Tpl(S^n)$ with the nerve of the covering by open hemispheres, and hence with $S^n$. Note that the homology version of this also appears as \cite[(3.9)]{dupont_book}.
\end{proof}

When $X$ is $n$-dimensional, let $\Tpl(X)^{n-1} \subseteq \Tpl(X)$ denote the subcomplex of those tuples that lie in some $(n-1)$-dimensional subspace, and similarly for $\ti\Tpl(X)^{n-1}$. This is identified with the union of the subcomplexes $\Tpl(U)$ over all proper subspaces $U \subsetneq X$. Even when $\Tpl(X)$ is contractible, this subcomplex is not, so the quotient $\Tpl(X)/\Tpl(X)^{n-1}$ has an interesting homotopy type.

We next describe $\Tpl(X)^{n-1}$ as a homotopy colimit. Recall that for any small category $\bI$ and any diagram of spaces $F\colon \bI \to \cat{Top}$, there is an unbased homotopy colimit
\[ \underset{\bI}\uhocolim F \]
defined using the Bousfield-Kan formula. (See \autoref{sec:hocolim} for more details.) We use the ``$u$'' decoration to distinguish this from the more common based homotopy colimit, defined from this unbased one by identifying the homotopy colimit of all the basepoints to a single point.

%The $(-)^u$ decoration is to distinguish this construction from the corresponding one for based spaces and spectra, which is also used frequently in the proof of the main theorem. The based variant also quotients out 
%
%If the spaces $F(i)$ are based then the homotopy colimit has a based variant, the quotient
%\[ \underset{\bI}\hocolim F = \left( \underset{\bI}\uhocolim F \right) \big/ \left( \underset{\bI}\uhocolim * \right). \]
%This has the homotopy type of the homotopy cofiber, provided the spaces $F(i)$ are well-based (the inclusion $* \to F(i)$ is a cofibration).

\begin{df}\label{tcofib}
Suppose $\bI$ has a terminal object $* \in \bI$. We define the {\bf total homotopy cofiber} of $F$ to be the quotient
\[ \tcofib(F) := \left( \underset{\bI}\uhocolim\, F \right) \Big/ \left( \underset{\bI \setminus \{*\}}\uhocolim\, F \right). \]
This has the homotopy type of a homotopy cofiber, since the map of homotopy colimits is a cofibration of unbased spaces.
\end{df}

%The proof of the main theorem will make frequent use of this standard fact:
%\begin{thm}
%	Any map of diagrams $F \to F'$ such that $F(i) \to F'(i)$ is a weak equivalence, induces a weak equivalence of homotopy colimits.
%\end{thm}
%The same is true for based homotopy colimits and total homotopy cofibers, provided the spaces $F(i)$ are all well-based (the inclusion of the basepoint is a cofibration).

\begin{ex}
	Suppose $\bI$ is an $n$-dimensional cube, i.e.~the poset of subsets of a fixed set $\{0,\ldots,n-1\}$, but we use the convention that $S \to T$ when $T \subseteq S$. Then our definition of the total homotopy cofiber matches the one from functor calculus \cite{calc2,calc3,munson_volic}. In particular, it is homeomorphic to the iterated homotopy cofibers (mapping cones) taken along each direction of the cube.
\end{ex}

%\begin{lem}
%	If $K$ is any cell complex, and $\bI$ is any poset of subcomplexes $K(i) \subseteq K$ that is closed under arbitrary nonempty intersections, then the canonical map
%	\[ \xymatrix{
%		\underset{i \in \bI}\uhocolim\, K(i) \ar[r] & \bigcup_{i \in \bI} K(i)
%	} \]
%	is an equivalence.
%\end{lem}
%
%\begin{proof}
%	The strict colimit along $\bI$ clearly surjects onto the union, and it is injective because any two points identified in the union lie in an intersection $K(i) \cap K(j)$, but this is represented in the diagram by our assumption on $\bI$, so they are identified in the colimit as well.
%	
%	To show this is a homeomorphism and that the homotopy colimit is equivalent to the strict colimit, we just need to show the diagram itself is cellular (so that its colimit is cellular). It has one cell for each cell in the union, attached at the smallest subcomplex containing that cell (such a smallest subcomplex is in the diagram by the assumption of closure under arbitrary intersections). Attaching these cells in order re-creates the diagram up to homeomorphism, hence the diagram is cellular.
%\end{proof}
%
%\begin{rmk}
%	For the application of the nerve theorem to $\Tpl(S^n)$, the diagram is only closed under finite intersections, but it is fine because each simplex lives in a unique smallest finite intersection of closed hemispheres.
%\end{rmk}
	
\begin{ex}\label{subcomplex_hocolim}
	Suppose $K$ is a cell complex, and $\bI$ is a poset of subcomplexes $K(i) \subseteq K$, closed under intersections, and terminating at $K$ itself. In this case the canonical map
	\[ \xymatrix{
		\underset{i \in \bI \setminus \{*\}}\uhocolim\, K(i) \ar[r] & \displaystyle\bigcup_{i \in \bI \setminus \{*\}} K(i)
	} \]
	is an equivalence, and therefore the total homotopy cofiber is equivalent to the quotient of $K$ by the union of the proper subcomplexes in the diagram:
	\[ K \Big/ \left( \bigcup_{i \in \bI \setminus \{*\}} K(i) \right). \]
\end{ex}

In other words, the quotient $\Tpl(X)/\Tpl(X)^{n-1}$ is a total homotopy cofiber.

\begin{df}\label{PT}
Let $X$ be any geometry and let $\bI$ be the poset of nonempty geometric subspaces $\emptyset \subsetneq U \subseteq X$, including $X$ itself. We define two diagrams on $\bI$:
\begin{itemize}
	\item A tautological diagram whose value at $U$ is the space $U$ itself. We define the {\bf polytopal Tits complex} %\footnote{This is not the same as the polytope complex as defined in \cite{zak_polytope}.}
	$\PT(X)$ to be its total homotopy cofiber:
	\begin{align*}
		\PT(X) &= \left( \underset{\emptyset \subsetneq U \subseteq X}\uhocolim\, U \right) \Big/ \left( \underset{\emptyset \subsetneq U \subsetneq X}\uhocolim\, U \right) \\
		&\simeq \left( \underset{\emptyset \subsetneq U \subseteq X}\uhocolim\, \Tpl(U) \right) \Big/ \left( \underset{\emptyset \subsetneq U \subsetneq X}\uhocolim\, \Tpl(U) \right) \\[0.4em]
		&\simeq \Tpl(X)/\Tpl(X)^{n-1}.
	\end{align*}
	\item A constant diagram whose value at $U$ is the one-point space $*$. We define the {\bf Steinberg complex} $\ST(X)$ to be its total homotopy cofiber:
	\begin{align*}
		\ST(X) &= \left( \underset{\emptyset \subsetneq U \subseteq X}\uhocolim\, {*} \right) \Big/ \left( \underset{\emptyset \subsetneq U \subsetneq X}\uhocolim\, {*} \right) \\
		&\simeq \left( \underset{\emptyset \subsetneq U \subseteq X}\uhocolim\, \ti\Tpl(U) \right) \Big/ \left( \underset{\emptyset \subsetneq U \subsetneq X}\uhocolim\, \ti\Tpl(U) \right) \\[0.4em]
		&\simeq \ti\Tpl(X)/\ti\Tpl(X)^{n-1}.
	\end{align*}
\end{itemize}
\end{df}
As mentioned in the introduction, the Steinberg complex $\ST(X)$ coincides with the RT-building $F^X_{\sbt}$ of \cite[Def 1.8]{cz}. See \cite[App A]{cz} for an alternative argument for the above equivalences.

Note that the second term in this definition $\left( \underset{\emptyset \subsetneq U \subsetneq X}\uhocolim\, {*} \right)$ is the Tits complex $\T(X)$, the realization of the poset of proper nonempty subspaces of $X$.
%\[ \T(X) = B(\emptyset \subsetneq U \subsetneq X) = \left( \underset{\emptyset \subsetneq U \subsetneq X}\uhocolim\, {*} \right), \]
Similarly the first term $\left( \underset{\emptyset \subsetneq U \subseteq X}\uhocolim\, {*} \right)$ is the cone on the Tits complex $\CT(X)$. Therefore
\begin{lem}
	The Steinberg complex $\ST(X)$ is homeomorphic to the unreduced suspension of the Tits complex $\T(X)$.
\end{lem}
This explains the convenient notation ``$\ST$'' for the Steinberg complex.

The collapse of each subspace $U$ to a point, or the inclusion of $\Tpl$ into $\ti\Tpl$, induces a canonical map
\[ \PT(X) \to \ST(X). \]
When $X = E^n$ or $H^n$, this is an equivalence, because every subspace $U$ is contractible, and homotopy colimits preserve all equivalences of diagrams.

The following is a variant of the Solomon-Tits theorem \cite{solomon_tits}.
\begin{thm}\label{our_st}
	$\PT(X)$ and $\ST(X)$ are each equivalent to a wedge of $n$-spheres, where $n$ is the dimension of the geometry $X$.
\end{thm}

This is a variant of known results, for instance the chain complex version that can be found in \cite[Thm 3.5 and Thm 3.10]{dupont_book}.

\begin{proof}
	From the first model of $\PT(X)$ it is easy to check that it is a cell complex with all cells of dimension $\leq n$. This makes the $n$th homology group free, and if $n \leq 1$ the proof is done. If not, we use the last model $\Tpl(X) / \Tpl(X)^{n-1}$, to see this space has the homotopy type of a complex with all cells of dimension $\geq n$. Therefore the space is simply-connected, and its reduced homology is concentrated in degree $n$ and is free. It follows using the Hurewicz theorem and Whitehead theorem that the space is equivalent to a wedge of $n$-spheres. The proof for $\ST(X)$ is the same.
\end{proof}

\begin{df}\label{pt}
The {\bf polytope module} and {\bf Steinberg module} are defined as
\begin{align*}
	\Pt(X) &:= \ti H_n(\PT(X)), \\
	\St(X) &:= \ti H_n(\ST(X)) = \ti H_{n-1}(\T(X)),
\end{align*}
as in \cite[(2.12) and (2.15)]{dupont_book}. These are free abelian groups, indexed by the spheres that show up in each wedge sum. Of course $\Pt(E^n) \cong \St(E^n)$ and $\Pt(H^n) \cong \St(H^n)$, but the canonical map $\Pt(S^n) \to \St(S^n)$ is not an isomorphism.
\end{df}

\begin{prop}\label{pt_presentation}
	$\Pt(X)$ is presented as the free abelian group on the polytopes $P \subseteq X$ modulo the relation $[P] = \sum_i [P_i]$ when the $P_i$ are a weak subdivision of $P$. Alternatively, it can be presented the same way but with $P$ and all $P_i$ required to be simplices.
\end{prop}

This can be deduced from the model $\Tpl(X) / \Tpl(X)^{n-1}$, see \cite[Thm 2.10]{dupont_book}. In particular, the generators for simplices $[\Delta^n]$ correspond to the $n$-simplices in $\Tpl(X) / \Tpl(X)^{n-1}$.

\begin{rmk}\label{twist}
	The isomorphism of \autoref{pt_presentation} is $I(X)$-equivariant provided we use the ``twisted'' action for each group. As in \autoref{intro_twist}, this means that in the presentation of \autoref{pt_presentation}, the action of $g$ sends $[P]$ to $[gP]$. On the homology of $\PT(X) \simeq \Tpl(X) / \Tpl(X)^{n-1}$, we take the obvious action that applies $g$ to the points in each tuple, but then twist by $\pm 1$, according to whether $g$ preserves or reflects orientation. With these conventions, the isomorphism is equivariant. (See also \autoref{twist2}.)
\end{rmk}

For $X = E^n$ or $H^n$, \autoref{pt_presentation} also gives a presentation of $\St(X)$. In the spherical case, $\St(S^n)$ can be presented as $\Pt(S^n)$ modulo an image of $\Pt(S(V))$ for every linear subspace $V \subseteq W$, see \autoref{lem_pi0_exact_seq} or \cite[Thm 3.13]{dupont_book}.

\begin{notn}\label{drop_s}
	For the geometry $S(W)$ that is the unit sphere of the inner product space $W$, we usually write $\PT(W)$ instead of $\PT(S(W))$ to improve readability. Similarly for $\ST(W)$, $\Pt(W)$, and $\St(W)$.
\end{notn}

\subsection{Spectra and Borel $G$-spectra}

For most of the paper we use the simplest model of spectra, sometimes called prespectra, consisting of spaces $X_n$ and bonding maps $\Sigma X_n \to X_{n+1}$. There are many standard references for these, among them \cite{mmss,schwede_ss}.

We say that a Borel $G$-spectrum is such a spectrum $X$ together with a continuous basepoint-preserving left action of $G$ on each level $X_n$, commuting with the bonding maps. An equivalence of such spectra is a map of spectra $X \to Y$, commuting with the $G$-action, that gives an isomorphism on the stable homotopy groups. Among other things, this gives an equivalence on the homotopy orbit spectra $X_{hG} \to Y_{hG}$.

\begin{lem}\label{suspend_invertible}
	If $S^V$ is any sphere with a basepoint-preserving action by $G$, then the operations $S^V \sma -$ and $\Map_*(S^V,-)$ preserve all equivalences of Borel $G$-spectra, and are inverses up to equivalence of Borel $G$-spectra.
\end{lem}

\begin{proof}
	The fact that they preserve equivalences follows from \cite{mmss} because $S^V$ is a cell complex. The unit and counit of the adjunction are equivalences on the underlying spectra, hence are equivalences of Borel $G$-spectra.
\end{proof}
	
At one point we will need the symmetric monoidal category of symmetric or orthogonal spectra with $G$-action. As before, we consider these up to equivariant maps that are stable equivalences on the underlying symmetric or orthogonal spectrum. There is a model category of such, obtained by taking the projective model structure on $G$-diagrams of diagram spectra, and it is Quillen equivalent to the model category of ordinary (pre-)spectra with $G$-action. This can be deduced easily from the standard theory in e.g.~\cite{mmss}.

The terminology ``Borel'' serves as a reminder that we are not using a more sophisticated equivariant theory that uses fixed points or representations of $G$, as in \cite{mm02}.

\subsection{Scissors congruence spectra}

Fix a geometry $X = H^n$, $E^n$, or $S^n$. Let
\[ \mc C = \Pol{X}{1} \]
be the category with one object for every polytope $P \subseteq X$ (possibly empty), and one morphism for every inclusion $P \subseteq Q$ in $X$. In other words, it is the poset of polytopes in $X$ under inclusion.

This category is an assembler in the sense of \cite{zak_assemblers}, see also \cite{bgmmz}, and therefore has an associated $K$-theory spectrum. We recall the construction.

Following \cite{zak_assemblers}, for each based set $A$, let
\[ A \sma \mc C = \bigvee_{A \setminus \{*\}} \mc C \]
denote the $(A \setminus \{*\})$-fold wedge of the category $\mc C$ along the object $\emptyset$. Explicitly, an object in this category is a pair $(a,P)$ in which $a \in A \setminus \{*\}$ and $P$ is a nonempty polytope, and additionally we have an initial object $\emptyset$. We have one morphism $(a,P) \to (b,Q)$ if $a = b$ and $P \subseteq Q$, and no morphisms otherwise.

Let $\mathcal W(\mc C)$ be the category in which an object is a finite set $I$, possibly empty, and a tuple of nonempty polytopes $\{P_i\}$ indexed by $I$. A morphism in $\mc W(\mc C)$
\[ \xymatrix{ \{ P_i \}_{i \in I} \ar[r] & \{ Q_j \}_{j \in J} } \]
is a map of finite sets $\alpha\colon I \to J$, and maps $P_i \to Q_{\alpha(i)}$ in $\mc C$, such that for each $j \in J$, the collection of $P_i$ mapping to $Q_j$ form a weak subdivision (almost-disjoint cover) of $Q_j$.

We similarly define $\mathcal W(A \sma \mc C)$ for any set $A$ as tuples $\{ (a_i,P_i) \}_{i \in I}$, and morphisms the indexed maps in the category $A \sma \mc C$. Again, for each $(a,Q_j)$, if we take every $(a,P_i)$ mapping to it (necessarily indexed by the same $a \in A$), the polytopes $P_i$ form an almost-disjoint cover of $Q_j$.

Let $S^1_\bullet$ denote the standard simplicial circle $\Delta[1]/\partial\Delta[1]$. Let $N_\bullet$ refer to the nerve of a category.

\begin{df}
The {\bf scissors congruence $K$-theory spectrum} is the symmetric spectrum arising from the Segal $\Gamma$-space
\[ {\bf n}_+ \mapsto \left| N_\bullet \mc W\left( {\bf n}_+ \sma \mc C \right) \right|. \]
Concretely, at spectrum level $k$ it is the realization of the multisimplicial set
\begin{equation}\label{sck_defn}
	[p,q_1,\ldots,q_k] \mapsto N_p \mc W\left( S^1_{q_1} \sma \ldots \sma S^1_{q_k} \sma \mc C \right),
\end{equation}
and the bonding maps of the spectrum arise from the identifications
\[  S^1_{q_1} \sma \ldots \sma S^1_{q_k} \cong S^1_{q_1} \sma \ldots \sma S^1_{q_k} \sma S^1_1. \]
This is a special $\Gamma$-space, so the scissors congruence $K$-theory spectrum is a positive $\Omega$-spectrum and is therefore semistable (see \cite{schwede_ss}). We are therefore free to ignore the symmetric spectrum structure and treat it as an ordinary spectrum. We denote this spectrum by $K(\mc C)$ or $K(\Pol{X}{1})$.
\end{df}

The spectrum $K(\Pol{X}{1})$ has an action by the isometry group $I(X)$, acting on each of the polytopes in $\mc C$ that appear at each multisimplicial level. We consider this as a Borel equivariant spectrum, that is, up to equivariant maps of spectra that are equivalences on the underlying spectrum. %This allows us to retain control over the homotopy type of the homotopy orbit spectrum $K(\Pol{X}{1})_{hG}$.

We similarly define $K(\Pol{X}{G})$ by replacing $\mc C$ by the category $\Pol{X}{G}$, whose objects are polytopes and whose morphisms are elements $g \in G$ and inclusions $g(P) \subseteq Q$. The main result of \cite{bgmmz} says that this is equivalent to the homotopy orbit spectrum
\[ K(\Pol{X}{G}) \simeq K(\Pol{X}{1})_{hG}. \]
As a result, this paper focuses exclusively on the spectrum $K(\Pol{X}{1})$ with its $G$-action.

\subsection{Finite sets and Barratt-Priddy-Quillen}

If $X^0 = E^0$ is the 0-dimensional, one-point geometry, then $\Pol{X^0}{1}$ is the poset $\{ \emptyset \to \bullet \}$. The associated category of indexed polytopes $\mathcal W(\{ \emptyset \to \bullet \})$ may be identified with the category of finite sets and isomorphisms $\Fin^{\cong}$. We denote the associated spectrum \eqref{sck_defn} as $K(\Fin^{\cong})$. At level $k$, it is the realization of the bisimplicial set
\begin{equation}\label{fin_defn}
	[p,q] \mapsto N_p \prod_{S^k_q \setminus \{*\}} \Fin^{\cong},
\end{equation}
restricted to those $S^k_q$-tuples of sets in $\Fin^{\cong}$ that are pairwise disjoint.

\begin{rmk}This is equivalent to the usual models of the $K$-theory of finite sets. For instance, the model described by Waldhausen in \cite[1.8]{1126} agrees with this one, except that instead of requiring the sets to be disjoint, we include a choice of coproduct for various collections of the finite sets, so that we can use these coproducts when we take face maps. Of course, if the sets happen to be disjoint, then the set-theoretic union works just fine as a model for the coproduct. This observation gives a map from \eqref{fin_defn} into Waldhausen's model. For each fixed $q$, this gives an equivalence of categories, hence an equivalence after taking realization with respect to $p$. Varying $q$, we then have a levelwise equivalence of Reedy cofibrant simplicial spaces, and hence we get an equivalence on the realizations.
\end{rmk}

Let $\Sph = \Sigma^\infty S^0$ refer to the standard model of the sphere spectrum whose $n$th level is $S^n$.

\begin{thm}[Barratt--Priddy--Quillen]\label{bpq}
	$K(\Fin^{\cong})$ is equivalent to the sphere spectrum $\Sph$.
\end{thm}

More specifically, the equivalence can be represented by any map $\Sph \to K(\Fin^{\cong})$ that picks out a point at spectrum level 0
\[ \left|N_\bullet \mc W\left( S^0_\bullet \sma \mc C \right)\right|
\cong \left|N_\bullet \mc W\left( \mc C \right)\right|
\cong \left|N_\bullet \Fin^{\cong} \right| \]
represented by a singleton set $I = \{i\}$.

\subsection{Reduced Thom spectra and homotopy orbits}

\begin{df}
	Suppose $Y$ is a cell complex, $E \to Y$ a vector bundle, and $A \subseteq Y$ a subcomplex. We define the reduced Thom space $\Sigma^E (Y/A)$ to be the quotient of the Thom space $\Th(E)$ by the subspace $\Th(E|_A)$. If $\xi = E - n$ is a virtual bundle then we define the {\bf reduced Thom spectrum} $\Sigma^\xi (Y/A)$ to be $\Sigma^{-n}\Sigma^E (Y/A)$.
\end{df}

\begin{ex}\label{bundle_on_pt}
The space
\[ \left( \underset{\emptyset \subsetneq U \subseteq X}\uhocolim\, U \right) \simeq \Tpl(X) \]
has a canonical map to $X$. We define a virtual vector bundle on this space by pulling back the negative of the tangent bundle $-TX$. This passes to a reduced Thom spectrum on the quotient space $\PT(X)$ from \autoref{PT}:
\[ \PT(X)^{-TX} = \Sigma^{-TX}\PT(X). \]
\end{ex}

\begin{rmk}\label{twist2}
	The polytope module with twisted action $\Pt(X)^t$ from \autoref{intro_twist} and \autoref{twist} arises naturally as the homology of this Thom spectrum:
	\[ \Pt(X)^t = \ti H_0(\PT(X)^{-TX}). \]
	By the Thom isomorphism, this agrees with $\Pt(X) \cong \ti H_n(\PT(X))$. But now any orientation-reversing map flips the orientation of the bundle $-TX$, and therefore flips the sign on homology of the Thom spectrum, giving the required twist.
\end{rmk}

\begin{lem}
	If a group $G$ acts on $Y$, $A$, and $E$ commuting with the maps $A \to Y \leftarrow E$, the reduced Thom space $\Sigma^E Y/A$ inherits a $G$-action, and its based homotopy orbit space is also a reduced Thom space
	\[ (\Sigma^E Y/A)_{hG} \cong \Sigma^{E_{hG}}((Y/A)_{hG}). \]
\end{lem}

\begin{proof}
	This follows from the description of the reduced Thom space as a quotient of two Thom spaces. Note that $E_{hG}$ is the unbased homotopy orbits of $E$, which forms a bundle over the unbased homotopy orbits $Y_{hG}$, while $(Y/A)_{hG}$ is the based homotopy orbits of $Y/A$, which is a quotient of the unbased homotopy orbit spaces $(Y_{hG})/(A_{hG})$.
\end{proof}

\begin{ex}
The space from \autoref{bundle_on_pt} is easily seen to have an action by the isometry group of $X$. Therefore for any subgroup $G \leq I(X)$ we get a reduced Thom spectrum
\[ \PT(X)_{hG}^{-TX} = (\Sigma^{-TX} \PT(X))_{hG} \cong \Sigma^{-(TX)_{hG}}(\PT(X)_{hG}). \]
\end{ex}

Finally, we discuss a reduction that will help us identify $K$-theory with this Thom space.

Let $X = E^n$ or $H^n$. Let $S^{X}$ be the one-point compactification of $X$, and let $\Sigma^{X}$ be the smash product with this one-point compactification. The exponential map of $X$ identifies the tangent bundle $TX \to X$ with the trivial fiber bundle $X \times X \to X$, whose fiber is the topological space $X$. This identification respects the action of the isometry group, essentially because isometries preserve geodesics.

This allows us to simplify the suspension of the Thom space as
\[ \Sigma^{X}\Sigma^{-TX}\PT(X) \simeq \Sigma^{TX}\Sigma^{-TX}\PT(X) \simeq \Sigma^\infty \PT(X). \]
By \autoref{suspend_invertible}, $\Sigma^{X}$ is an equivalence on spectra with an action by $I(X)$. Therefore, to prove our main theorem, it therefore suffices to suspend $K$-theory by $\Sigma^{X}$, and identify the result with the stable homotopy type of $\PT(X)$.

Similarly, when $X = S(W)$ for an inner product space $W$, we use the standard isomorphism between the once-stabilized tangent bundle and the trivial bundle with fiber $W$,
\[ TS(W) \times \R \cong S(W) \times W. \]
This gives us an equivalence of Borel $O(W)$-spectra
\[ \Sigma^{W}\Sigma^{-TS(W)}\PT(W) \simeq \Sigma^\infty (\Sigma\PT(W)). \]
Since homotopy orbits $(-)_{hG}$ preserve equivalences of Borel $G$-spectra, we conclude:
\begin{lem}\label{tangent_reduction}
	To prove that $K(\Pol{X}{G})$ is equivalent to the Thom spectrum $\PT(X)^{-TX}_{hG}$, it suffices to give an equivalence of Borel $G$-spectra
\[ \Sigma^{X} K(\Pol{X}{1}) \simeq \Sigma^\infty \PT(X) \]
and in the spherical case
\[ \Sigma^{W} K(\Pol{S(W)}{1}) \simeq \Sigma^\infty (\Sigma\PT(W)). \]
\end{lem}

\section{Homotopy colimits and total homotopy cofibers}\label{sec:hocolim}

In this section we fix some conventions for homotopy colimits that will be used in the proof of the main theorem, and some facts about total homotopy cofibers that are useful for the subsequent computations.

Recall that the simplicial category $\mathbf\Delta$ has objects the sets
\[ [n] = \{0,1,\ldots,n\} \]
and maps are the nondecreasing maps of sets. Let $\mathscr I \subseteq \mathbf\Delta$ denote the category with the same objects but only the injective maps. A simplicial space is a functor $X_\bullet\colon \mathbf\Delta^\op \to \cat{Top}$. For more terminology about simplicial objects, we refer the reader to e.g.~\cite{goerss_jardine}.

%Recall that a {\bf simplicial space} is a functor $X_\bullet\colon \mathbf\Delta^\op \to \cat{Top}$, where we usually write $X_n$ for the value on $[n]$, and a simplicial set is a simplicial space in which every $X_n$ is discrete. The action of $\mathbf\Delta$ is generated by face maps $X_n \to X_{n-1}$ and degeneracy maps $X_{n-1} \to X_n$. The standard $k$-simplex $\Delta[k]_\bullet$ is the one represented by $[k]$, in other words
%\[ \Delta[k]_n = \mathbf\Delta([n],[k]). \]
%As $k$ varies these form a $\mathbf\Delta$-diagram of simplicial sets. Using this action, we form for each simplicial space $X$ a coequalizer diagram of simplicial spaces as follows.
%\begin{equation}\label{simp_presentation}
%		\coprod_{k,\ell \geq 0} \Delta[k]_\bullet \times \mathbf\Delta([k],[\ell]) \times X_\ell \rightrightarrows 
%		\coprod_{k \geq 0} \Delta[k]_\bullet \times X_k \rightarrow X_\bullet
%\end{equation}
%The realization of 

Let $\bI$ be any poset (partially ordered set) and $F\colon \bI \to \cat{Top}$ any diagram of unbased spaces indexed by $\bI$. As mentioned in \autoref{sec:complexes}, we define the unbased homotopy colimit
\[ \underset{\bI}\uhocolim F \]
using the Bousfield-Kan formula. To be concrete, we take the simplicial space that at level $k$ is the coproduct over chains in $\bI$ of $F$ evaluated at the chain's first object:
\[ \coprod_{i_0 \leq \cdots \leq i_k} F(i_0). \]
As usual, the face maps delete elements from the chain $i_0 \leq \cdots \leq i_k$, and apply the map $F(i_0) \to F(i_1)$ if $i_0$ is deleted. The degeneracy maps duplicate one of the elements $i_j$.

This simplicial space has a special property. The nondegenerate points at each level are the union over the strict chains
\[ \coprod_{i_0 < \cdots < i_k} F(i_0), \]
and these are preserved by the face maps. It follows that the realization can be defined using only the nondegenerate points and the face maps, as the following coequalizer.
\begin{equation}\label{realization_inj}
	\coprod_{k,\ell \geq 0} \left(\Delta^k \times \mathscr I([k],[\ell]) \times \coprod_{i_0 < \cdots < i_\ell} F(i_0) \right) \rightrightarrows 
	\coprod_{k \geq 0} \left( \Delta^k \times \coprod_{i_0 < \cdots < i_k} F(i_0) \right) % \rightarrow \underset{\bI}\uhocolim F.
\end{equation}
%Here $\mathscr I$ is the subcategory of the simplicial category $\mathbf\Delta$ consisting of the injective maps of totally-ordered sets, in other words, the subcategory generated by the coface maps.

This is a standard fact, but it will be helpful to spell it out explicitly.
\begin{lem}\label{reduce_to_inj}
	When $\bI$ is a poset and $F\colon \bI \to \cat{Top}$ is a diagram of spaces, the inclusion of \eqref{realization_inj} into the usual formula for the realization
	\begin{equation}\label{realization}
	\coprod_{k,\ell \geq 0} \left(\Delta^k \times \mathbf\Delta([k],[\ell]) \times \coprod_{i_0 \leq \cdots \leq i_\ell} F(i_0) \right) \rightrightarrows 
	\coprod_{k \geq 0} \left( \Delta^k \times \coprod_{i_0 \leq \cdots \leq i_k} F(i_0) \right)
\end{equation}
	is a homeomorphism.
\end{lem}

\begin{proof}
	We define the inverse map from \eqref{realization} to \eqref{realization_inj} as follows. For each non-strict chain $c = [i_0 \leq \cdots \leq i_k]$, there is a unique strict chain $c' = [i_0 < \cdots < i_k]$ with the same objects but only $k' < k$ terms, and a surjective map $\beta_c\colon [k] \to [k']$ encoding which of the terms in the original chain were duplicated. We map the corresponding copy of $\Delta^k \times F(i_0)$ in \eqref{realization} to $\Delta^{k'} \times F(j_0)$ in \eqref{realization_inj} by applying $\beta_c$ to the simplex and the identity to $F(i_0)$.
	
	We check this respects the coequalizer relation along each map $\alpha\colon [k] \to [\ell]$ in $\mathbf\Delta$. Since each map factors into a surjective map followed by an injective map, we consider these two cases separately. When $\alpha$ is surjective, the check is straightforward since the two non-strict chains we start with have the same associated strict chain. If $\alpha$ is injective, then for each non-strict chain $c_1 = [j_0 \leq \cdots \leq j_\ell]$, the action of $\alpha$ takes it to a smaller non-strict chain
	\[ c_0 = [i_0 \leq \cdots \leq i_k] \subseteq c_1 = [j_0 \leq \cdots \leq j_\ell]. \]
	This induces an inclusion of the associated strict chains $c_0' \subseteq c_1'$. The associated map of finite sets $\alpha'$ fits into a commuting square
	\[ \xymatrix{
		[k] \ar@{->>}[d]_-{\beta_{c_0}} \ar[r]^\alpha & [\ell] \ar@{->>}[d]^-{\beta_{c_1}} \\
		[k'] \ar[r]^-{\alpha'} & [\ell'].
	} \]
	Using this we check that the coequalizer relation in \eqref{realization} along $\alpha$ goes to the coequalizer relation in \eqref{realization_inj} along $\alpha'$, and so after applying both the map is well defined.
	
	If we go from \eqref{realization_inj} to \eqref{realization} to \eqref{realization_inj} we clearly get the identity. If we go from \eqref{realization} to \eqref{realization_inj} to \eqref{realization}, the image of each term $\Delta^k \times F(i_0)$ is identified back to $\Delta^k \times F(i_0)$ by the identity map, using the coequalizer relation along $\beta_c$.
\end{proof}

\begin{rmk}\label{skeletal}
	We will frequently use the following standard ``inductive'' or ``skeletal'' interpretation of \eqref{realization_inj}. The homotopy colimit may be defined inductively, where at the $k$th stage we attach copies of $\Delta^k \times F(i_0)$ along the boundary $(\partial \Delta^k) \times F(i_0)$, one for each strict chain $i_0 < \cdots < i_k$. The $j$th face of $\partial\Delta^k$ is attached to the copy of $\Delta^{k-1} \times F(i_0)$ whose chain is $i_0 < \cdots < i_k$ but with the $j$th term missing. In the case of $j = 0$, this attaching map also applies the map $F(i_0) \to F(i_1)$ arising from $i_0 < i_1$.
\end{rmk}

Let $C(\bI)$ be the poset whose objects are nonempty strict chains in $\bI$
\[ c_0 = [ i_0 < \cdots < i_k ] \]
and whose morphisms $c_0 \to c_1$ are reverse inclusions $c_0 \supseteq c_1$. In other words, the morphisms delete elements from the chains. Taking the first element of each chain defines a functor $C(\bI) \to \bI$, and composing with this functor makes $F$ into a diagram over $C(\bI)$ as well.
\begin{lem}\label{sd_homeo}
	Barycentric subdivision gives a natural homeomorphism
	\[ \underset{\bI}\uhocolim\, F \cong \underset{C(\bI)}\uhocolim\, F. \]
\end{lem}	

\begin{proof}
	Let $\Delta[k]_\bullet$ denote the standard $k$-simplex as a simplicial set. Recall that its barycentric subdivision $(\sd\Delta[k])_\bullet$ is the nerve of the poset of the nonempty faces $\Delta^a \subseteq \Delta^k$, $0 < a \leq k$, in other words the poset of strict chains contained in $[0 < \cdots < k]$.
	
	This is natural in $k$, and therefore extends in a canonical way to an operation on simplicial spaces $X_\bullet$, by taking the canonical presentation of any such simplicial space
	\[
		\coprod_{k,\ell \geq 0} \Delta[k]_\bullet \times \mathbf\Delta([k],[\ell]) \times X_\ell \rightrightarrows 
		\coprod_{k \geq 0} \Delta[k]_\bullet \times X_k \rightarrow X_\bullet
	\]
	and applying the subdivision to the terms $\Delta[k]_\bullet$ and $\Delta[\ell]_\bullet$:
	\[
		\coprod_{k,\ell \geq 0} (\sd\Delta[k])_\bullet \times \mathbf\Delta([k],[\ell]) \times X_\ell \rightrightarrows 
		\coprod_{k \geq 0} (\sd\Delta[k])_\bullet \times X_k \rightarrow (\sd X)_\bullet
	\]
	There is a homeomorphism of realizations
	\[ |\sd\Delta[k]_\bullet| \cong |\Delta[k]_\bullet| \]
	that is natural in $k$, and therefore for any simplicial space $X$ a homeomorphism
	\[ |\sd X_\bullet| \cong |X_\bullet|. \]
	
	Now let $X_\bullet$ be the simplicial space defining the homotopy colimit of $F$ over $\bI$. Take the subdivision $\sd X_\bullet$ and apply the argument of \autoref{reduce_to_inj}, but with the simplicial set $(\sd\Delta[k])_\bullet$ in the place of the topological space $\Delta^k$. The conclusion is that the subdivision is expressed as the coequalizer along just the face maps,
	\begin{equation}\label{sd_inj}
		\coprod_{k,\ell \geq 0} \left((\sd\Delta[k])_\bullet \times \mathscr I([k],[\ell]) \times \coprod_{i_0 < \cdots < i_\ell} F(i_0) \right) \rightrightarrows 
		\coprod_{k \geq 0} \left( (\sd\Delta[k])_\bullet \times \coprod_{i_0 < \cdots < i_k} F(i_0) \right).
	\end{equation}
	Writing this out, the space $(\sd X)_n$ has one copy of $F(i_0)$ for each choice of an integer $k \geq 0$, a strict chain $i_0 < \cdots < i_k$, and a point in $(\sd\Delta[k])_n$, corresponding to an $n$-tuple of strict chains
	\[ c_0 \supseteq \cdots \supseteq c_n \]
	in $[0 < \cdots < k]$, equivalently in $[i_0 < \cdots < i_k]$. This is up to the relation that if all the chains $c_0, \ldots, c_n$ miss the point $i_j$, we may delete that point from the data (and apply $F(i_0) \to F(i_1)$ if $j = 0$).
	
	Equivalently, the space $(\sd X)_n$ has one copy of $F(i_0)$ for each $n$-tuple of strict chains
	\[ c_0 \supseteq \cdots \supseteq c_n \]
	in the entire poset $\bI$. The degeneracy and face maps match exactly those in the definition of the homotopy colimit over $C(\bI)$. Therefore $\sd X_\bullet$ is identified with the simplicial space defining the homotopy colimit of $F$ over $C(\bI)$.
	
	In summary, we get homeomorphisms
	\[ \underset{\bI}\uhocolim\, F = |X_\bullet| \cong |\sd X_\bullet| \cong \underset{C(\bI)}\uhocolim\, F. \]
\end{proof}

\begin{rmk}
	Using non-strict chains in the definition of $C(\bI)$ would give an equivalent answer, but not a homeomorphic one. Also, the result does not appear to generalize from posets to arbitrary categories. % The inclusion of strict into non-strict is homotopy terminal: every non-strict chain has an initial strict chain under it, so the category of all strict chains under it is contractible.
\end{rmk}

Recall from \autoref{tcofib} the definition of total homotopy cofiber: if $\bI$ is a poset with a terminal object $*$, and $F$ is a diagram on $\bI$, then
\[ \tcofib(F) := \left( \underset{\bI}\uhocolim\, F \right) / \left( \underset{\bI \setminus \{*\}}\uhocolim\, F \right). \]
\begin{lem}\label{tcofib_mapping_cone}
	In this case, the total homotopy cofiber is homeomorphic to the mapping cone of the canonical map
	\[ \xymatrix{ \underset{\bI \setminus \{*\}}\uhocolim\, F \ar[r] & F(*). } \]
\end{lem}
\begin{proof}
	The skeletal description of the homotopy colimit in \autoref{skeletal} leads to a skeletal description of the mapping cone, built from pieces of the form $\Delta^k \times F(i_0) \times I$ for each $i_0 < \cdots < i_k$. The identifications are along the faces of $\Delta^k$. All the faces associated to the front of the interval $I$ are quotiented to a single point, and all the faces associated to the back of the interval $I$ are glued to $F(*)$.
	
	To identify each of these pieces to the corresponding piece $\Delta^{k+1} \times F(i_0)$ in the homotopy colimit over $\bI$, we add the terminal object to the chain to get $i_0 < \cdots < i_k < *$. We apply the map
	\[ \Delta^k \times I \to \Delta^{k+1} \]
	that sends $\Delta^k \times \{1\}$ to the final vertex, $\Delta^k \times \{0\}$ to the opposite face by a homeomorphism, and interpolates linearly for the rest. Inductively, these identifications give homeomorphisms on each of the skeleta, and therefore give a homeomorphism on the entire realization.
\end{proof}

We finish the section by showing that our definition of total homotopy cofiber agrees with the total homotopy cofiber of a cube in a wide array of examples, not just the obvious one when $\bI$ is a cube-shaped poset.

Suppose $\bI$ is equipped with a ``dimension'' map, a map of posets from $\bI$ to
\[ [n] = \{ 0 < \cdots < n \}. \]
We ask that the preimage of $n$ is only the terminal object $*$, and the preimage of any other point is discrete (contains no non-identity morphisms). In other words, every nontrivial morphism in $\bI$ goes to a nontrivial morphism in $[n]$. For example, $\bI$ may be the geometric subspaces of a fixed geometry $X$, and the map to $[n]$ takes every subspace to its dimension.

Let $\bJ$ be the cube-shaped poset of all subsets of $[n-1]$, ordered by reverse inclusion, so $S \to T$ means that $S \supseteq T$. Note that every strict chain $[ i_0 < \cdots < i_k ]$ in $\bI \setminus \{*\}$ is assigned to a subset $S \subseteq [n-1]$ of size $(k+1)$, encoding the images of the elements $i_0$ through $i_k$ in the set $[n-1]$. This defines a functor that we denote
\[ \dim\colon C(\bI \setminus \{*\}) \to (\bJ \setminus \{\emptyset\}). \]
We define a diagram on $\bJ$ by sending each subset $S \subseteq [n-1]$ to the coproduct of $F$ over all chains whose dimensions match $S$:
\[ \coprod_{\underset{\dim = S}{i_0 < \cdots < i_k}} F(i_0). \]
We also send the empty set to $F(*)$. We call this diagram $\amalg_{\dim} F$.
\begin{prop}\label{punctured_homeo}
	Under these hypotheses, there is a homeomorphism
	\[ \underset{\bI\setminus\{*\}}\uhocolim\, F \cong \underset{\bJ \setminus \{\emptyset\}}\uhocolim\, \amalg_{\dim} F. \]
\end{prop}

\begin{proof}
	By \autoref{sd_homeo} the left-hand side is identified with the homotopy colimit of $F$ over $C(\bI \setminus \{*\})$. The associated simplicial space has as $n$-simplices the disjoint union
	\[ \coprod_{c_0 \supseteq \cdots \supseteq c_n} F(i_0) \]
	where the $c_j$ are strict chains in $\bI \setminus \{*\}$. On the other hand, the right-hand side comes from a simplicial space whose $n$-simplices are the disjoint union
	\[ \coprod_{S_0 \supseteq \cdots \supseteq S_n} \left(\coprod_{\underset{\dim c_0 = S_0}{c_0}} F(i_0) \right) \]
	where the $S_j$ are subsets of $[n-1]$. Given a strict chain $c_0$ in $\bI \setminus \{*\}$ with dimensions given by the set $S_0$, each subset $S_j \subseteq S_0$ corresponds to a unique strict chain $c_j \subseteq c_0$. This allows us to identify the two simplicial spaces and conclude that their realizations are homeomorphic.
\end{proof}

Together with \autoref{tcofib_mapping_cone} this implies
\begin{cor}\label{tcofib_homeo}
	The total homotopy cofibers of $F$ and $\amalg_{\dim} F$ are homeomorphic:
	\[ \left( \underset{\bI}\uhocolim\, F \right) / \left( \underset{\bI \setminus \{*\}}\uhocolim\, F \right)
	\cong \left( \underset{\bJ}\uhocolim\, \amalg_{\dim} F \right) / \left( \underset{\bJ \setminus \{\emptyset\}}\uhocolim\, \amalg_{\dim} F \right). \]
\end{cor}

Therefore for any poset $\bI$ with such a map to $[n]$, the total homotopy cofiber over $\bI$ is also a total homotopy cofiber of over a cube $\bJ$. In particular, the polytope and Steinberg complexes of \autoref{PT} are total homotopy cofibers of cubes.
\begin{cor}\label{PT_cube}
\[ \PT(X) \cong \tcofib\left(
	S \mapsto \coprod_{\underset{\dim = S}{U_0 \subseteq \cdots \subseteq U_k}} U_0, \quad 
	\emptyset \mapsto X
\right), \]
\[ \ST(X) \cong \tcofib\left(
	S \mapsto \coprod_{\underset{\dim = S}{U_0 \subseteq \cdots \subseteq U_k}} {*}, \quad 
	\emptyset \mapsto {*}
\right). \]
\end{cor}
This description is of fundamental importance in calculations, since taking the based homotopy orbits and reduced Thom spectra has the effect of taking unbased homotopy orbits and ordinary Thom spectrum inside the cube. See \autoref{sec:homology}.

\section{Proof of the main theorem}

In this section we prove \autoref{intro_main}, focusing first on the case where $X$ is either $H^n$ or $E^n$. Set $G = I(X)$. By \autoref{tangent_reduction}, it suffices to construct an equivalence of Borel $G$-spectra
\[ \Sigma^{X} K(\Pol{X}{1}) \simeq \PT(X) \]
where $\Sigma^{X}$ refers to the smash product with $S^{X}$, the one-point compactification of $X$.

\subsection{The main diagram}\label{sec:main_diagram}
Our proof proceeds by constructing the following zig-zag of equivalences of Borel $G$-spectra.
\begin{equation}\label{main_diagram}
\xymatrix @!C=7em{
	\Sigma^{X} K(\mc C) && \Sigma^\infty \PT(X) \\
	\Sigma^{X} \underset{\{ P_i \} \in \mathcal D}\hocolim\, K(\mc C_{\{P_i\}}) \ar[u]_-\sim \ar[d]^-\sim && \Sigma^\infty EG_+ \sma \underset{\{ P_i \} \in \mathcal D}\hocolim\, \bigvee_i P_i/\partial P_i \ar[u]_-\sim \ar[d]^-\sim \\
	\Sigma^{X} \underset{\{ P_i \} \in \mathcal D}\hocolim\, \prod_i K(\Fin^{\cong}) && \underset{\{ P_i \} \in \mathcal D}\hocolim\, \prod_i \Sigma^\infty P_i/\partial P_i \\
	& \Sigma^{X} \underset{\{ P_i \} \in \mathcal D}\hocolim\, \prod_i \Sph \ar[lu]_-\sim \ar[ru]^-\sim & 
}
\end{equation}
As before, $\mc C = \Pol{X}{1}$ denotes the category of polytopes in $X$ and inclusions. The homotopy colimits are all based homotopy colimits, hence the lack of the ``$u$'' decoration.

Recall from \autoref{sec:polytopes} that a set of geometric simplices is \emph{almost-disjoint} if their interiors are disjoint. Let $\mathcal D$ denote the set of finite collections of almost-disjoint geometric simplices in $X$. We give $\mc D$ a partial ordering by
\[ \{ P_i \} \to \{ Q_j \} \]
if each $P_i$ is almost-disjointly covered by a subset of the $Q_j$. Equivalently, if we can go from the $P_i$ to the $Q_j$ by weakly subdividing \emph{and} by adding extra simplices that are almost-disjoint from the existing ones.

Each collection $\{ P_i \} \in \mathcal D$ specifies a subcategory $\mc C_{\{P_i\}} \subseteq \mc C$, consisting of those polytopes that are unions of subsets of the collection $\{P_i\}$.

\begin{lem}
	The category $\mc C$ is the filtered colimit of the subcategories $\mc C_{\{P_i\}} \subseteq \mc C$.
\end{lem}

\begin{proof}
	This follows from \autoref{triangulation_exists}. Since triangulations exist, each object or morphism of $\mc C$ is in one of the categories $\mc C_{\{P_i\}}$. Since any two triangulations have a common refinement (\autoref{triangulation_common_refinement}), any two such subcategories $\mc C_{\{P_i\}}$ and $\mc C_{\{Q_j\}}$ are contained in a third subcategory $\mc C_{\{R_k\}}$.
\end{proof}
This means that the multisimplicial set in \eqref{sck_defn} is a filtered colimit of the same sets for $\mc C_{\{P_i\}}$. We recall the standard fact:
\begin{lem}
	Filtered colimits of simplicial sets are homotopy colimits.
\end{lem}

Therefore the induced map from the homotopy colimit to $K(\mc C)$ is an equivalence of spectra:
\[ \xymatrix{ \underset{\{ P_i \} \in \mathcal D}\hocolim\, K(\mc C_{\{P_i\}}) \ar[r]^-\sim & K(\mc C). } \]

This is the first equivalence in \eqref{main_diagram}. The second equivalence is
\begin{equation}\label{equiv_2}
\xymatrix{ \underset{\{ P_i \} \in \mathcal D}\hocolim\, K(\mc C_{\{P_i\}}) \ar[r]^-\sim & \underset{\{ P_i \} \in \mathcal D}\hocolim\, \prod_i K(\Fin^{\cong}) }
\end{equation}
where $\Fin^{\cong}$ denotes the category of finite sets and isomorphisms. The category $\mathcal D$ acts on the product $\prod_i K(\Fin^{\cong})$ by sending every weak subdivision to a diagonal map
\[ \xymatrix{ K(\Fin^{\cong}) \ar[r]^-{\Delta} & \prod K(\Fin^{\cong}) } \]
and every inclusion of a new simplex to the zero map $* \to K(\Fin^{\cong})$.

For each $\{P_i\}_{i \in I}$ and each $i \in I$, consider the following functor
\[ \mathcal W(\mc C_{\{P_i\}}) \to \mc W(\Pol{X^0}{1}) \cong \Fin^{\cong}. \]
It takes each formal disjoint union of polytopes $\coprod_{j \in J} Q_j$, where each polytope $Q_j$ is a union of some of the simplices $P_i$, to the subset $J_i \subseteq J$ consisting of the values of $j$ for which $Q_j$ contains the simplex $P_i$. Put another way, this operation intersects each $Q_j$ with some fixed point in the interior of $P_i$ to give either a point or the empty set, and takes the union of the results over all $j \in J$ to give $J_i \subseteq J$. This is a functor because any almost-disjoint cover
\[ \xymatrix{ \coprod_{j \in J} Q_j \ar[r] & \coprod_{k \in K} R_k }. \]
induces for each $i \in I$ an isomorphism of sets $J_i \cong K_i$.

This functor induces a map of bisimplicial sets
\[ N_p \mc W\left( S^k_q \sma \mc C_{\{P_i\}} \right) \to N_p \mc W\left( S^k_q \sma \Pol{X^0}{1} \right) \]
for each $k$, giving a map of spectra
\begin{equation*}\label{atomic_factor}
	K(\mc C_{\{P_i\}}) \to K(\Fin^{\cong}).
\end{equation*}
Taking the product over $i$ gives
\begin{equation}\label{atomic}
	K(\mc C_{\{P_i\}}) \to \prod_{i \in I} K(\Fin^{\cong}).
\end{equation}
It is straightforward to check that this commutes strictly with the action of $\mc D$, giving the desired map of homotopy colimits \eqref{equiv_2}. The author learned the following fact from Zakharevich.
\begin{lem}\label{zak_lemma}
	The map \eqref{atomic} is an equivalence of spectra.
\end{lem}
\begin{proof}
It suffices to show for each based set $A$ that the product map
\[ \mathcal W(A \sma \mc C_{\{P_i\}}) \to \prod_i \mathcal W(A \sma \Pol{X^0}{1}) \]
is a right adjoint. We define a left adjoint by sending the tuple of finite sets $\{S_i\}_{i \in I}$ to the formal disjoint union $\coprod_{i \in I} \coprod_{S_i} P_i$ (and repeating this procedure for each $a \in A$).

To check this is an adjunction, it suffices to restrict attention to a single point of $A$. Suppose we have any almost-disjoint cover of the form
\begin{equation}\label{atomic_adj_2}
	\xymatrix{ \coprod_{i \in I} \coprod_{S_i} P_i \ar[r] & \coprod_{j \in J} Q_j }.
\end{equation}
The associated map of sets sends each $S_i$ to $J$. This is injective because $P_i$ cannot occur twice in an almost-disjoint cover. The image in $J$ is exactly those $Q_j$ containing $P_i$, since the map is a cover. It therefore gives an isomorphism of sets
\[ S_i \cong J_i \]
for each $i \in I$. Conversely, any isomorphism $S_i \cong J_i$ gives an almost-disjoint cover \eqref{atomic_adj_2}, because each $Q_j$ is being covered by $P_i$ exactly once for each $i$ such that $j \in J_i$. It is straightforward to check this identification is natural in the finite sets $S_i$ and in the polytopes $\{Q_j\}$.
%Suppose $\{Q_j\}_{j \in J}$ is a tuple of polytopes in $\mc C_{\{P_i\}}$, so that each $Q_j$ is a union of a subset of the $P_i$. Define a relation between the sets $I$ and $J$ by $iRj$ if $P_i \subseteq Q_j$. Note that
%\[ \xymatrix{ \coprod_{i \in I \colon iRj} P_i \ar[r] & Q_j } \]
%is an almost-disjoint cover, and therefore so is
%\[ \xymatrix{ \coprod_{j \in J} \coprod_{i \in I \colon iRj} P_i \ar[r] & \coprod_{j \in J} Q_j }. \]
%Rearranging by pulling the $i$ out front, we get an almost-disjoint cover
%\begin{equation}\label{atomic_adj}
%	\xymatrix{ \coprod_{i \in I} \coprod_{j \in J \colon iRj} P_i \ar[r] & \coprod_{j \in J} Q_j }.
%\end{equation}
%This is the counit of the adjunction. It is natural because any almost-disjoint cover to a new set of polytopes $\{R_k\}_{k \in K}$ induces the following commuting square.
%\[ \xymatrix{
%	\coprod_{i \in I} \coprod_{j \in J \colon iRj} P_i \ar[d]^-\cong \ar[r] & \coprod_{j \in J} Q_j \ar[d] \\
%	\coprod_{i \in I} \coprod_{k \in K \colon iRk} P_i \ar[r] & \coprod_{k \in K} R_k
%} \]
%
%The unit of the adjunction is the obvious isomorphism of sets $J_i \cong J_i$ for each $i \in I$, while the counit is the map that takes, for each $Q_j$ that is a union of some of the $P_i$, the almost-disjoint covering map from those $P_i$ back to $Q_j$.
\end{proof}

The third equivalence in \eqref{main_diagram} is:
\[ \xymatrix{ \underset{\{ P_i \} \in \mathcal D}\hocolim\, \prod_i \Sph \ar[r]^-\sim & \underset{\{ P_i \} \in \mathcal D}\hocolim\, \prod_i K(\Fin^{\cong}). } \]
In the diagram on the left, we have $\mc D$ act by diagonal and zero maps, just as it does on the right. Any fixed choice of equivalence $\Sph \to K(\Fin^{\cong})$ from the Barratt-Priddy-Quillen theorem (\autoref{bpq}) induces an equivalence of diagrams, and therefore an equivalence on homotopy colimits.

We are now ready to prove the Solomon-Tits theorem for the spectrum $K(\mc C) = K(\Pol{X}{1})$, in other words \autoref{intro_st} from the introduction.
\begin{thm}\label{wedge_of_spheres}
	$K(\mc C)$ is non-equivariantly a wedge of sphere spectra $\Sph$.
\end{thm}

We know from the introduction that $K_0(\mc C) = K_0(\Pol{X}{1}) \cong \Pt(X)^t$ is the free abelian polytope module, so the spheres are indexed by any basis for $\Pt(X)^t$ as an abelian group.

\begin{proof}
We use the above equivalences to simplify the homotopy groups as follows.
\[ 
	\pi_k(K(\mc C)) \cong \pi_k\left( \underset{\{ P_i \} \in \mathcal D}\hocolim\, \prod_i \Sph \right)
	\cong \underset{\{ P_i \} \in \mathcal D}\colim\, \left(\bigoplus_i \pi_k(\Sph) \right)
	\cong \pi_k(\Sph) \otimes_\Z \left( \underset{\{ P_i \} \in \mathcal D}\colim\, \bigoplus_i \Z \right).
\]
The last isomorphism arises because the maps in the colimit system of spectra are all sums of identity maps of the sphere spectrum, hence the map induced on $\pi_k$ agrees with the induced map on $\pi_0$, tensored with $\pi_k(\Sph)$.

The colimit inside the right-hand expression is isomorphic to $\Pt(X)^t$, by comparing the colimit over $\mc D$ to the presentation in \autoref{pt_presentation}. Since this group is free abelian, tensoring with $\pi_*(\Sph)$ gives a free module over $\pi_*(\Sph)$. Since the homotopy groups of $K(\mc C)$ are a free module over $\pi_*(\Sph)$, the spectrum $K(\mc C)$ receives a $\pi_*$-isomorphism from a wedge of sphere spectra, and is therefore equivalent to a wedge of sphere spectra.
\end{proof}

The fourth equivalence in \eqref{main_diagram} begins to use the suspension by $X$. We pass it into the colimit and past the finite product to get
\[ \xymatrix{ \Sigma^{X} \underset{\{ P_i \} \in \mathcal D}\hocolim\, \prod_i \Sph \ar[r]^-\sim & \underset{\{ P_i \} \in \mathcal D}\hocolim\, \prod_i \Sigma^\infty S^{X}. } \]
We further apply a Pontryagin-Thom collapse $S^{X} \to P_i/\partial P_i$. Recall this is defined to be the identity in the interior of $P_i$ (which makes sense since $P_i \subseteq X$), while everything else in $X$ is sent to the basepoint of $P_i/\partial P_i$.

The collapse map is easily seen to be a degree-one map of spheres, and therefore an equivalence. It also commutes with the action of $\mathcal D$, if we define the action of the weak subdivision $\{P_i\} \to \{Q_j\}$ on $P_i/\partial P_i$ to be the product of collapse maps
\begin{equation}\label{one_collapse}
	P_i / \partial P_i \to Q_j / \partial Q_j
\end{equation}
for each $Q_j \subseteq P_i$. We get
\[ \xymatrix{ \underset{\{ P_i \} \in \mathcal D}\hocolim\, \prod_i \Sigma^\infty S^{X} \ar[r]^-\sim & \underset{\{ P_i \} \in \mathcal D}\hocolim\, \prod_i \Sigma^\infty P_i/\partial P_i, } \]
which finishes the fourth equivalence.

The fifth equivalence is the observation that the collapse maps \eqref{one_collapse} descend to maps of wedge sums
\[ \bigvee_i P_i / \partial P_i \to \bigvee_j Q_j / \partial Q_j, \]
because each point in $P_i$ goes to the interior of at most one $Q_j$ in the almost-disjoint cover of $P_i$. So we replace the finite products with finite wedges and pull the suspension spectrum out of the colimit, giving the fifth equivalence
\[ \xymatrix{ \Sigma^\infty \underset{\{ P_i \} \in \mathcal D}\hocolim\,  \bigvee_i P_i/\partial P_i \ar[r]^-\sim & \underset{\{ P_i \} \in \mathcal D}\hocolim\, \prod_i \Sigma^\infty P_i/\partial P_i. } \]

It remains to show that the space $\left(\underset{\{ P_i \} \in \mathcal D}\hocolim \bigvee_i P_i/\partial P_i\right)$ is stably equivalent to $\PT(X)$. This will require several more lemmas.

\subsection{Apartment-like maps}

For each simplex $P \subseteq X$, not necessarily of full dimension, let $\Tpl(P) \subseteq \Tpl(X)$ denote the subcomplex of all tuples of points in $X$ that all lie inside $P$.

\begin{df}\label{apartment-like}
	\begin{itemize}
		\item A map from a simplex $f\colon P \to \Tpl(X)$ is {\bf apartment-like} if for each face $D \subseteq P$, including $P$ itself, $f(D) \subseteq \Tpl(D)$.
		\item A map from a finite union of simplices $\cup_i P_i \to \Tpl(X)$ is apartment-like if this condition holds for every face of every simplex $P_i$.
		\item A map from the unbased homotopy colimit
		\[ \xymatrix{ \underset{\{ P_i \} \in \mathcal D}\uhocolim\,  \bigcup_i P_i \ar[r] & \Tpl(X) } \]
		is apartment-like if for each $k$-tuple of composable morphisms in $\mathcal D$ starting at $\{P_i\}$, the above condition holds on the corresponding piece of the homotopy colimit
		\[ \xymatrix{ \Delta^k \times \bigcup_i P_i \ar[r] & \Tpl(X) } \]
		for each point in $\Delta^k$.
	\end{itemize}
\end{df}

\begin{rmk}
	An earlier version of this paper uses the word ``introspective,'' but ``apartment-like'' is better terminology, because every apartment in the Tits complex $T(X)$ in the classical sense (see e.g.~\cite{solomon_tits,church_farb_putman_integrality}) is apartment-like in our sense.
\end{rmk}

\begin{lem}\label{apartment-like_contractible}
	For any finite union of simplices $\cup_i P_i$ in $X$ (not necessarily almost-disjoint), the space of apartment-like maps $\cup_i P_i \to \Tpl(X)$ is weakly contractible.
\end{lem}

\begin{proof}
	For any simplex $D$, $\Tpl(D)$ is contractible because the simplices can be coned off to any single point in $D$. The same argument shows that for any finite list of simplices $\{D_i\}$, the intersection
\[ \cap_i \Tpl(D_i) = \Tpl\left(\cap_i D_i\right) \]
is contractible.

Given a finite union of simplices $\cup_i P_i$, by \autoref{triangulation_exists} 
%to subdivide each $P_i$ one at a time, extending the subdivision of $P_i \cap P_j$ for those $P_j$ already subdivided. If that proposition is only rel boundary, take pairwise intersections first to make the $P_i$ almost-disjoint. At any rate, this gives a triangulation of the union with each $P_i$ a subcomplex.
we may form a triangulation of the union with each $P_i$ a subcomplex. Let $\{Q_j\}$ be the simplices in this triangulation. For each face $D$ of one of the $Q_j$, if we consider those faces $D_{i,k} \subseteq \partial P_i$ containing $D$, then the apartment-like maps on $\cup_i P_i$ require $D$ to go into the intersection $\cap_{i,k} D_{i,k}$. This intersection contains $D$ but may be larger.
	
	For each $D$ separately, the space of such maps is contractible. Also, if $D \subseteq D'$, then $\cap_{i,k} D_{i,k} \subseteq \cap_{i,k} D'_{i,k}$. So when we pass to a larger face, the subspace where it must be sent is also enlarged, never shrunk.
	
	We prove weak contractibility of the whole space by showing that any sphere $S^{m-1}$ of apartment-like maps can be extended to a disc $D^m$ of apartment-like maps. We do this by induction on the cells of the triangulation $\cup_j Q_j$. When adding a new cell $D$, we aim to specify a map
	\begin{equation}\label{apartment-like_step}
		\xymatrix{ D \times D^m \ar[r] & \Tpl(\cap_{i,k} D_{i,k}) }
	\end{equation}
	that is already given both on $D \times S^{m-1}$ and $\partial D \times D^m$. The map from $D \times S^{m-1}$ lands in $\Tpl(\cap_{i,k} D_{i,k})$ by assumption. The map from $\partial D \times D^m$ is composed of maps that are apartment-like on each face of $D$, so they land in a smaller space inside $\Tpl(\cap_{i,k} D_{i,k})$, so they land in $\Tpl(\cap_{i,k} D_{i,k})$. Therefore \eqref{apartment-like_step} is specified on the entire boundary, and we just need to extend it to the rest of $D \times D^m$. Since $\Tpl(\cap_{i,k} D_{i,k})$ is contractible, this extension exists, completing the induction.
\end{proof}

It is also straightforward to check:
\begin{lem}\label{apartment-like_weak_subdivision}
	If $\{Q_j\}$ is a weak subdivision of $\{P_i\}$ then any apartment-like map $\bigcup_j Q_j \to \Tpl(X)$ is also apartment-like as a map $\bigcup_i P_i \to \Tpl(X)$.
\end{lem}

\begin{lem}\label{apartment-like_hocolim}
	The space of apartment-like maps on the homotopy colimit
	\[ \underset{\{ P_i \} \in \mathcal D}\uhocolim\,  \bigcup_i P_i \to \Tpl(X) \]
	is weakly contractible. Any such map that is only specified on the 0-skeleton of the homotopy colimit (\autoref{skeletal}) extends to an apartment-like map on the entire homotopy colimit.
\end{lem}

\begin{proof}
	This is essentially the same induction as in \autoref{apartment-like_contractible}. To show that an $S^{m-1}$ of apartment-like maps cones off to a $D^m$ of such maps, we induct up the pieces of the form $\Delta^k \times \bigcup_i P_i$ in the homotopy colimit.
	
	At the inductive step, we are given an apartment-like map on $S^{m-1} \times \Delta^k \times \bigcup_i P_i$, and a map on $D^m \times \partial \Delta^k \times \bigcup_i P_i$ that came from the apartment-like maps at the previous stages of the induction. This latter map is apartment-like at each point in $D^m \times \partial \Delta^k$, using \autoref{apartment-like_weak_subdivision}, because on some of the faces the map comes from a weak subdivision of the $P_i$. All together, then, we have an apartment-like map on $\partial(D^m \times \Delta^k) \times \bigcup_i P_i$, and this extends to $D^m \times \Delta^k \times \bigcup_i P_i$ because apartment-like maps out of $\bigcup_i P_i$ are weakly contractible (\autoref{apartment-like_contractible}).
	
	When $m = 0$, this proves the existence of an apartment-like map, since $D^0 = *$ and $S^{-1} = \emptyset$. The final claim of the lemma now follows from the same induction, but fixing $m$ at 0 and only varying $k$.
\end{proof}

The following is a straightforward comparison of cell complex structures. Notice that the first two homotopy colimits are unbased, while the final one is based.
\begin{lem}\label{unbased_to_based}
	The canonical maps
	\[ \xymatrix{
		\underset{\{ P_i \} \in \mathcal D}\uhocolim\,  \bigcup_i \partial P_i \ar[r] &
		\underset{\{ P_i \} \in \mathcal D}\uhocolim\,  \bigcup_i P_i \ar[r] &
		\underset{\{ P_i \} \in \mathcal D}\hocolim\,  \bigvee_i P_i/\partial P_i
	} \]
	form both a cofiber sequence and a homotopy cofiber sequence.
\end{lem}

Each apartment-like map $\underset{\{ P_i \} \in \mathcal D}\uhocolim\,  \bigcup_i P_i \to \Tpl(X)$ sends the boundaries of each simplex $P_i$ into $\Tpl(X)^{n-1}$, and therefore induces a map of quotients
\begin{equation}\label{apartment-like_quotient}
	\xymatrix{ \underset{\{ P_i \} \in \mathcal D}\hocolim\,  \bigvee_i P_i/\partial P_i \ar[r] & \Tpl(X)/\Tpl(X)^{n-1}. }
\end{equation}

\begin{lem}
	Any map \eqref{apartment-like_quotient} arising from an apartment-like map, induces an equivalence on suspension spectra.
\end{lem}

This, finally, finishes the sixth equivalence from \eqref{main_diagram}.

\begin{proof}
	By the second part of \autoref{apartment-like_hocolim}, we can build an apartment-like map that sends each $P_i$ to the simplex $[P_i]$ in $\Tpl(X)$. Since all apartment-like maps are homotopic, it suffices to prove that this one is an equivalence. We know that the suspension spectra of both sides of \eqref{apartment-like_quotient} are wedges of $n$-spheres $\Sph^n$, so it suffices to prove the map is an isomorphism on $n$th homology.

Since homology sends filtered homotopy colimits to colimits, we get a presentation of $H_n$ on the left-hand side with one generator for every $P_i$, modulo the relation that the generator for $P_i$ is the sum of the generators in any weak subdivision. This matches the presentation of $H_n$ of the right-hand side from \autoref{pt_presentation}. Furthermore the map that sends $P_i$ to $[P_i]$ respects this presentation. The map is therefore an isomorphism on homology. Since both sides are wedges of spheres of the same dimension by \autoref{wedge_of_spheres} and \autoref{our_st}, the map is therefore an equivalence of spectra.
\end{proof}

\begin{rmk}\label{apartment-like_change}
	The definition of apartment-like is easily adapted from $\Tpl(X)$ to the equivalent space
	\[ \underset{\emptyset \subsetneq U^k \subseteq X}\uhocolim\, \Tpl(U^k) \]
	by replacing each instance of $\Tpl(D)$ with the subcomplex of this homotopy colimit
	\[ \underset{\emptyset \subsetneq U^k \subseteq \spa(D)}\uhocolim\, \Tpl(U^k \cap D). \]
	The key property is that the finite intersections of such complexes are contractible:
	\[ \underset{\emptyset \subsetneq U^k \subseteq \cap_i \spa(D)}\uhocolim\, \Tpl(U^k \cap (\cap_i D_i)). \]
	The indexing category has $\cap_i \spa(D)$ as a terminal object, and at that object we get the contractible space $\Tpl(\cap_i D_i)$, hence this homotopy colimit is contractible.
	
	We will need this modification for the reduced spherical case, because then collapsing the copies of $\Tpl(U^k)$ to a point brings us to the Steinberg complex.
\end{rmk}

\subsection{Equivariance}
The final phase of the proof is to check that each of the equivalences we defined can be made equivariant with respect to the action of the isometry group $G = I(X)$. The group $G$ acts on the category $\mathcal D$ in the obvious way. It also acts on each diagram that we have defined over $\mathcal D$, bearing in mind that this action can be nontrivial on the sums or products over $i$, rearranging the terms to match where each simplex $P_i$ goes under the action of the isometry group.

\begin{lem}\label{make_equivariant}
	Each of the maps in \eqref{main_diagram} can be made into a $G$-equivariant equivalence.
\end{lem}

\begin{proof}
For almost all of the maps, the map as we defined it is equivariant, and the check of equivariance is straightforward.

The last map, however, is not usually equivariant. Instead, the space of apartment-like maps from \autoref{apartment-like_hocolim} has an action of $G$ by conjugation. In other words, if $f$ is apartment-like and $g \in G$ then the conjugate $g \circ f \circ g^{-1}$ is also apartment-like.

It follows that there is a $G$-equivariant map from $EG$ into the space of apartment-like maps. The adjoint of this is a $G$-equivariant map
\[ \xymatrix{ EG \times \underset{\{ P_i \} \in \mathcal D}\uhocolim\,  \bigcup_i P_i \ar[r] & \Tpl(X). } \]
On quotients, this induces an equivariant map
\[ \xymatrix{ EG_+ \sma \underset{\{ P_i \} \in \mathcal D}\hocolim\,  \bigvee_i P_i/\partial P_i \ar[r]^-\sim & \Tpl(X)/\Tpl(X)^{n-1}. } \]
This equivariant map is also an equivalence, because it was an equivalence on each point of $EG$, and $EG$ is contractible.
\end{proof}

In summary, \eqref{main_diagram} is a zig-zag of equivalences of Borel $G$-spectra. This finishes the proof of \autoref{intro_main} in the Euclidean and hyperbolic cases.

\subsection{The spherical case}

Let $W$ be a fixed finite-dimensional inner product space of dimension $n+1$. Let $S(W)$ be its unit sphere and $S^W$ its one-point compactification. The proof of \autoref{intro_main} detailed in this section applies to this case as well, with a small modification:
\begin{itemize}
	\item We suspend by $S^W$ instead of $S^{X}$. 
	\item In the Pontryagin-Thom collapse step, for each simplex $P_i \subseteq S(W)$ we use the collapse map
	\[ S^W \to \Sigma(P_i / \partial P_i). \]
	We model $\Sigma(P_i / \partial P_i)$ as the region in $S^W$ consisting of 0, $\infty$, and all rays from the origin through $P_i$, modulo the boundary of this region. Note that this collapse is still an equivalence because it is a degree-one map of spheres.
\end{itemize}
The resulting diagram of suspensions $\Sigma(P_i / \partial P_i)$ is the suspension of the diagram of quotients $P_i / \partial P_i$ we considered before, so we can pull the suspension out of the homotopy colimit:
\[ \underset{\{ P_i \} \in \mathcal D}\hocolim\, \bigvee_i \Sigma(P_i/\partial P_i) \simeq \Sigma \left( \underset{\{ P_i \} \in \mathcal D}\hocolim\, \bigvee_i P_i/\partial P_i \right). \]
The remaining discussion of apartment-like maps is unchanged.

All together, this gives a Borel equivalence to the suspension
\[ \Sigma^W K(\Pol{S(W)}{1}) \simeq \Sigma^\infty \Sigma \PT(W), \]
recalling the convention in \autoref{drop_s}. By \autoref{tangent_reduction}, this completes the proof of \autoref{intro_main} in the spherical case.

\section{Comparison to group homology}\label{sec:homology}

In this section, we use the main theorem to establish the relationship between the higher scissors congruence groups and group homology.

We will need to work in a model of spectra with smash product, such as symmetric or orthogonal spectra. See e.g.~\cite{schwede_ss,mmss} for details.

Let $HA$ denote the Eilenberg-Maclane spectrum for the abelian group $A$. This is characterized by the property that its homotopy groups consist only of $A$ in degree 0.

If $A$ has a left action of $G$ through homomorphisms of abelian groups, then $HA$ becomes a Borel $G$-spectrum. Again, it is characterized up to equivalence of Borel spectra by its lone homotopy group $A$, with its $G$-action. We recall that group homology can be defined by taking homotopy orbits of Eilenberg-Maclane spectra:
\begin{lem}\label{hg_is_group_homology}
	For any $\Z[G]$-module $A$, there is a canonical isomorphism
\[ \pi_n( (HA)_{hG} ) \cong H_n(G;A). \]
\end{lem}
This is well known, and follows for instance from the equivalence of $\infty$-categories between chain complexes and $H\Z$-module spectra from \cite{ss_modules,shipley_algebras}, or the homotopy orbit spectral sequence for stable homotopy groups.

%Recall that for a homotopy orbit spectrum $X_{hG}$ there is an Atiyah-Hirzebruch spectral sequence
%\[ H_p(G;E_q(X)) \Rightarrow E_{p+q}(X_{hG}) \]
%for any extraordinary homology theory $E_*$. Taking $E_*$ to be ordinary homology and considering the spectrum $K(\Pol{X}{G}) \simeq K(\Pol{X}{1})_{hG}$, we get
%\[ H_p(G;H_q(K(\Pol{X}{1})) ) \Rightarrow H_{p+q}(K(\Pol{X}{G})). \]
%Since $K(\Pol{X}{1})$ is a wedge of spheres, this sequence is concentrated on the line $q = 0$, and therefore it collapses, giving an isomorphism
%\[ H_p(G;\Pt(X)^t) \cong H_p(K(\Pol{X}{G})). \]
%Furthermore the $G$-action on $\Pt(X)^t$ is exactly the action on $K_0(\Pol{X}{1})$ that we described in \autoref{twist}. Alternatively, if we use the main theorem and describe $K_0(\Pol{X}{1})$ as a Thom spectrum, it is the action-with-a-twist described in \autoref{twist2}. Either way, this finishes the proof of \autoref{intro_homology}.

Recall that the homology of a spectrum is computed by taking the smash product with $H\Z$ and then taking the homotopy groups.
By \autoref{wedge_of_spheres} for all three of the geometries, $K(\Pol{X}{1})$ is equivalent to a wedge of sphere spectra $\Sph$, indexed by any basis for the polytope module $\Pt(X)^t$. Therefore, after smashing with $H\Z$ we get a wedge of Eilenberg-Maclane spectra. Since homotopy groups turn wedge sums into direct sums, this is also an Eilenberg-Maclane spectrum, whose sole homotopy group is the abelian group $\Pt(X)^t$:
\begin{align*}
	K(\Pol{X}{1}) \sma H\Z &\simeq H(K_0(\Pol{X}{1})) \simeq H(\Pt(X)^t).
\end{align*}
The operation $- \sma (H\Z)$ commutes with homotopy orbits, and therefore
\begin{align*}
	K(\Pol{X}{G}) \sma H\Z &\simeq K(\Pol{X}{1})_{hG} \sma H\Z \\
	&\simeq (K(\Pol{X}{1}) \sma H\Z)_{hG} \\
	&\simeq H(\Pt(X)^t)_{hG}.
\end{align*}
Taking homotopy groups and using \autoref{hg_is_group_homology}, we conclude that there is an isomorphism
\begin{align*}
	H_m(K(\Pol{X}{G})) &\cong H_i(G; \Pt(X)^t).
\end{align*}
Furthermore the $G$-action on $\Pt(X)^t$ is exactly the action on $K_0(\Pol{X}{1})$ that we described in \autoref{twist}, finishing the proof of \autoref{intro_homology}. %Alternatively, if we use the main theorem and describe $K(\Pol{X}{1})$ as a Thom spectrum, it is the action-with-a-twist described in \autoref{twist2}.

\begin{rmk}
	The above proof uses only the Solomon-Tits result, \autoref{wedge_of_spheres}. We can also give a second proof based on the stronger result that $K(\Pol{X}{G})$ is a Thom spectrum. By the Thom isomorphism, the homology of $\PT_{hG}^{-TX}$ is a shift of the homology of $\PT(X)_{hG}$, with a twist. Since $\PT(X)$ is a wedge of spheres, the homotopy orbit spectral sequence for $\PT(X)_{hG}$ is also concentrated on a single line, so it collapses, giving the result.
\end{rmk}
	
We next turn to the proof of \autoref{intro_translation_homology}. Let $X = E^n$. We first handle the case where $G = T(n)$ is the translation group.

Recall that the rationalization of a spectrum is its smash product with $H\Q$. This has the effect of tensoring all of the homotopy groups with $\Q$. A spectrum $E$ is rational if the map
\[ E \cong E \sma \Sph \to E \sma H\Q \]
is an equivalence. This holds iff the homotopy groups $\pi_*(E)$ are rational vector spaces.

\begin{prop}\label{rationality_proof}
	The spectrum $K(\Pol{E^n}{T(n)})$ is rational.
\end{prop}

\begin{proof}
	The bundle $TE^n$ is trivial as a $T(n)$-bundle so we get
	\[ K(\Pol{E^n}{T(n)}) \simeq \Sigma^{-n} \PT(E^n)_{hT(n)}. \]
	Next we use the total homotopy cofiber description of $\PT(E^n) \simeq \ST(E^n)$ from \autoref{PT_cube}:
	\[ \PT(E^n) \simeq \tcofib\left(
		S \mapsto \coprod_{\underset{\dim = S}{U_0 \subseteq \cdots \subseteq U_k}} {*}, \quad 
	\emptyset \mapsto {*}
	\right). \]
	The homotopy orbits can be passed inside the cube, so that each vertex is a homotopy orbit space of a single group orbit. Recall that the Shapiro Lemma allows us to simplify such spaces by the formula
	\[ (G/H)_{hG} \simeq BH \]
	where $BH$ is the classifying space of $H \leq G$. All together this gives
	\begin{equation}\label{translation_presentation}
	\PT(E^n)_{hT(n)} = \tcofib\left(
		S \mapsto \coprod_{\underset{\dim = S}{V_0 \subseteq \cdots \subseteq V_k}} BT(V_0), \quad 
	\emptyset \mapsto BT(n)
	\right),
	\end{equation}
	where now each of the spaces $V_i$ is a linear subspace passing through the origin, with translation group $T(V_i)$.
	
	Recall that the total cofiber of a cube is obtained by iteratively taking homotopy cofibers in each direction. If the spaces are unbased, we add disjoint basepoints to them first before performing this operation. Equivalently, when we iteratively take homotopy cofiber, we make the first homotopy cofiber unbased (i.e.~the mapping cone), and then the subsequent ones are all based (the reduced mapping cone).
	
	In this case, let us take the first homotopy cofiber in the direction of the element $0 \in [n-1]$. Each of the maps is of the form
	\[ \xymatrix{
		\displaystyle \coprod_{\underset{\dim = S}{V_0 \subseteq \cdots \subseteq V_k}} {*} \ar[r] & 
		\displaystyle \coprod_{\underset{\dim = S}{V_0 \subseteq \cdots \subseteq V_k}} BT(V_0)
	} \]
	for some subset $S \subseteq \{1,\ldots,n-1\}$. The cofiber of this map is the wedge sum
	\[ \bigvee_{\underset{\dim = S}{V_0 \subseteq \cdots \subseteq V_k}} BT(V_0). \]
	This cofiber is a based space whose reduced integral homology groups are all rational \cite[4.7]{dupont_book}. The remaining cofibers will preserve this property, hence $\PT(E^n)_{hT(n)}$ has all reduced homology groups rational. It follows by an easy application of the Atiyah--Hirzebruch spectral sequence that its suspension spectrum is a rational spectrum.
\end{proof}

\begin{prop}\cite[4.10a]{dupont_book}
	The homology groups $H_i(T(n);\Pt(X)^t)$ are rational.
\end{prop}

%\begin{proof}
%	This is done in \cite[4.10a]{dupont_book}, although it is not stated explicitly there. The argument actually parallels the one just above, but with the cube replaced by the double complex \cite[]{dupont_book}.
%\end{proof}

\begin{prop}
	The above two rationality results also hold for any subgroup $G \leq E(n)$ containing $T(n)$.
\end{prop}

\begin{proof}
	Any such group contains $T(n)$ as a normal subgroup, because $T(n)$ is normal in $E(n)$. Therefore the homotopy $G$-orbits can be written as iterated homotopy orbits
	\[ K(\Pol{X}{1})_{hG} \simeq \left( K(\Pol{X}{1})_{hT(n)}\right)_{h(G/T(n))}. \]
	Since the homotopy orbits on the inside are rational, this is equivalent to
	\[ \left( K(\Pol{X}{1})_{hT(n)} \sma H\Q \right)_{h(G/T(n))}  \simeq \left( K(\Pol{X}{1})_{hT(n)} \right)_{h(G/T(n))} \sma H\Q \]
	and is therefore rational.
	
	The same reasoning applies to the homology groups $H_i(G;\Pt(X)^t)$, since they are the homotopy groups of the homotopy orbit spectrum $H(\Pt(X)^t)_{hG}$.
\end{proof}

In summary, for $G \geq T(n)$ we have
\[ K_i(\Pol{X}{G}) \cong H_i(G;\Pt(X)^t) \]
because both sides are rational and we know they are equivalent rationally (\autoref{intro_homology}). This finishes the proof of  \autoref{intro_translation_homology}.

\section{The reduced spherical case}\label{sec:reduced}

In this section we introduce the reduced spherical $K$-theory spectrum and prove \autoref{intro_reduced}.

Let $V \subseteq W$ be an inclusion of finite-dimensional inner product spaces with orthogonal complement $V^\perp$. We define a suspension map
\[ \xymatrix{ K(\Pol{S(V)}{1}) \ar[r]^-\Sigma & K(\Pol{S(W)}{1}) } \]
by taking every polytope $P \subseteq S(V)$ to the join $P * S(V^\perp)$, regarded as a subset of $S(W)$ by drawing the straight lines in $W$ and then projecting to $S(W)$. (The lines never pass through the origin because each one joins a nonzero point in $V$ to a nonzero point in the complement $V^\perp$.) This operation induces a functor on the category of polytopes
\begin{equation}\label{suspension_on_category}
	\xymatrix{ \Pol{S(V)}{1} \ar[r]^-\Sigma & \Pol{S(W)}{1}. }
\end{equation}
This, in turn, induces a map on $K$-theory spectra.

The suspension maps respect composition, in the sense that if $V_0 \subseteq V_1 \subseteq W$, then suspending from $V_0$ to $V_1$ and then to $W$ gives the same map (identically, not just up to homotopy) as suspending from $V_0$ to $W$.

If $V = 0$, then $S(V)$ is the empty geometry, with a single polytope, the empty polytope. We define the suspension map in this case to take the empty polytope to the entire sphere $S(W)$, inducing a map
\[ \xymatrix{ \Sph \simeq K(\Pol{S(0)}{1}) \ar[r]^-\Sigma & K(\Pol{S(W)}{1}). } \]

All together, the suspension maps make the spectra $K(\Pol{S(V)}{1})$ into a diagram indexed by the poset of subspaces $0 \subseteq V \subseteq W$.

\begin{df}\label{reduced_spherical_df}
Reduced spherical $K$-theory $\ti K(\Pol{S(W)}{1})$ is defined to be the total homotopy cofiber of this diagram in the sense of \autoref{tcofib}. So, there is a cofiber sequence
\begin{equation}\label{reduced_spherical}
\xymatrix{
	\underset{0 \subseteq V \subsetneq W}\hocolim\, K(\Pol{S(V)}{1}) \ar[r] & K(\Pol{S(W)}{1}) \ar[r] & \ti K(\Pol{S(W)}{1}).
}
\end{equation}
We similarly define $\ti K(\Pol{S(W)}{G})$ for any subgroup $G \leq O(W)$ as the homotopy orbit spectrum of $\ti K(\Pol{S(W)}{1})$. So, there is a cofiber sequence
\begin{equation*}
\xymatrix{
	\left(\underset{0 \subseteq V \subsetneq W}\hocolim\, K(\Pol{S(V)}{1})\right)_{hG} \ar[r] & K(\Pol{S(W)}{G}) \ar[r] & \ti K(\Pol{S(W)}{G}).
}
\end{equation*}
\end{df}

\begin{rmk}
	When $S(W) = S^n$ and $G = O(n+1)$, this definition may be compared to the more standard definition as the cofiber of the suspension map
	\[ \xymatrix{ K(\Pol{S^{n-1}}{O(n)}) \ar[r]^-\Sigma & K(\Pol{S^{n}}{O(n+1)}). } \]
	The two definitions agree on $K_0$, where most previous work has been done, but differ on the higher $K$-groups. We consider \autoref{reduced_spherical_df} to be the more natural of the two definitions, in light of \autoref{intro_reduced}.
\end{rmk}

\begin{rmk}
	It is essential that the colimit in \autoref{reduced_spherical_df} contains the $\Sph$ term at $V = 0$. Though it does not change the definition of $\ti K(\Pol{S(W)}{1})$ at $\pi_0$, it does affect the higher homotopy groups, and without this change \autoref{intro_reduced} would not hold.
\end{rmk}

\subsection{Suspension on the polytope module}

On $\pi_0$, the suspension map $P \mapsto P * S(V^\perp)$ induces a suspension homomorphism
\begin{equation}\label{suspension_pt}
	\xymatrix{ \Pt(V) \ar[r]^-\Sigma & \Pt(W). }
\end{equation}
\begin{lem}
	The map \eqref{suspension_pt} is the inclusion of a summand.
\end{lem}

\begin{proof}
We can characterize $\Pt(W)$ as functions $S(W) \to \Z$ that are constant on the interior of each simplex in some triangulation of $S(W)$, up to changing the function on measure-zero sets (or a finite union of simplices of dimension less than that of $S(W)$). Such functions can be characterized by the closures of the preimages $\overline{(\textup{int} f^{-1}(n))}$ for each $n \in \Z$.

The image of $\Pt(V)$ under the suspension map is those functions for which the preimages described above are all joins of polytopes in $S(V)$ with $S(V^\perp)$. This is clearly isomorphic to $\Pt(V)$, so the suspension map $\Pt(V) \to \Pt(W)$ is injective. In fact, it is split injective; we see this when $\dim W - \dim V = 1$ by picking a normal direction to $S(V)$ and taking for each polytope in $S(W)$ its intersection with an infinitesimal neighborhood of $S(V)$ in the chosen normal direction. We conclude that the suspension map is the inclusion of a summand.
\end{proof}

We say the polytope $Q \subseteq S(W)$ is {\bf a suspension from $V$} if it is in the image of the suspension map, and similarly a function $f \in \Pt(W)$ is a suspension from $V$ if it is in the image of the suspension map \eqref{suspension_pt}.

If we suspend from $U$ to $V$ and then from $V$ to $W$, we get the same result as if we had suspended from $U$ to $W$. As a result, if $U \subseteq V$ then any suspension from $U$ is also a suspension from $V$. The following lemma lets us go backwards.

\begin{lem}\label{intersect_suspensions}
	If the polytope $Q \subseteq S(W)$ is a suspension from $U$ and from $V$ then it is a suspension from $U \cap V$.
\end{lem}

\begin{proof}
	First note that $w \in W$ is perpendicular to $U$ and to $V$ iff $w$ is perpendicular to the internal sum $U + V \subseteq W$. Therefore
	\[ (U + V)^\perp = U^\perp \cap V^\perp \]
	and dually $(U \cap V)^\perp = U^\perp + V^\perp$.
	
	Assume $Q$ is a suspension from $U$ and from $V$. Then it contains all of $S(U^\perp)$ and $S(V^\perp)$. Each point in $S(U^\perp + V^\perp)$ lies along a path from a point $y \in S(U^\perp \cap V)$ to a point $z \in S(V^\perp)$. Since $Q$ contains $y$ and $z$, and $Q$ is a suspension from $V$, it contains the whole path from $y$ to $z$. Therefore $Q$ contains all of $S(U^\perp + V^\perp) = S((U \cap V)^\perp)$.
	
	Now suppose $Q$ contains some $x \in S(U \cap V)$. Then since $Q$ is a suspension from $U$, it contains all paths joining $x$ to any $y \in S(U^\perp \cap V)$. Each intermediate point $w$ in such a path lies in $S(V)$. Since $Q$ is a suspension from $V$, it contains the path linking $w$ to any $z \in S(V^\perp)$. These paths together cover the join of $x$ to all of $S((U \cap V)^\perp)$, so $Q$ contains this entire join.
	
	Conversely, if $Q$ does not contain $x$, then since it is a suspension from $U$, it does not contain any of the points $w$ on the path from $x$ to $y$, save for $y$ itself. Since $Q$ is a suspension from $V$, it therefore does not contain any of the points on the path from any such $w$ to $z$, save for $z$ itself. Therefore $Q$ does not contain any of the join of $x$ to $S((U \cap V)^\perp)$, except for $S((U \cap V)^\perp)$ itself. All together, this shows that $Q$ satisfies the conditions of being a suspension from $U \cap V$.
\end{proof}

\begin{cor}\label{minimal_suspension}
	Each polytope $Q \subseteq S(W)$ is a suspension from some unique minimal space $U \subseteq W$. It is also a suspension from every subspace containing $U$.
\end{cor}

\begin{lem}\label{weak subdivision_monotonicity}
	If $\{Q_i\}$ is a weak subdivision of $P$, and every $Q_i$ is a suspension from $V$, then $P$ is a suspension from $V$.
\end{lem}

\begin{proof}
	The union of $Q_i = P_i * S(V^\perp)$ is equal to $P = (\cup_i P_i) * S(V^\perp)$.
\end{proof}

\begin{lem}\label{lem_pi0_exact_seq}
	The canonical maps form a split exact sequence of free abelian groups
\begin{equation}\label{pi0_exact_seq}
	\xymatrix{ 0 \ar[r] & \underset{0 \subseteq V \subsetneq W}\colim \Pt(V) \ar[r] & \Pt(W) \ar[r] & \St(W) \ar[r] & 0. }
\end{equation}
\end{lem}
If we chose splittings for these exact sequences for every $V \subseteq W$ we would get an isomorphism $\Pt(W) \cong \bigoplus_{0 \subseteq V \subseteq W} \St(V)$, though we will not use this.

\begin{proof}
	The colimit over all subspaces $0 \subseteq V \subseteq W$ simplifies to the colimit over those subspaces of codimension 1 and 2, since this smaller set is cofinal in the larger set. (In other words, given a $V$ of any other codimension, the poset of all $V'$ of codimension 1 or 2 that contain $V$ is a connected poset.)
	
	The conclusion now follows from the exact sequence of \cite[Thm 3.13]{dupont_book}:
\begin{equation*}\label{pi0_exact_seq_2}
	\xymatrix{ \bigoplus_{\dim U = n-2} \Pt(U) \otimes \St(U^\perp) \ar[r] & \bigoplus_{\dim V = n-1} \Pt(V) \ar[r] & \Pt(W) \ar[r] & \St(W) \ar[r] & 0. }
\end{equation*}
	The first map sends each polytope $P \subseteq S(U)$ and pair of vectors $v,w \in S(U^\perp)$ to the difference of the suspensions $[P * \{v,-v\}] - [P * \{w,-w\}]$, while the second map sends $Q \subseteq S(V)$ to its suspension in $S(W)$. It follows that the cokernel of the first map is the colimit $\underset{0 \subseteq V \subsetneq W}\colim \Pt(V)$, giving the desired exact sequence.
\end{proof}

\subsection{Suspension on $K$-theory}

The suspension map \eqref{suspension_on_category} sends polytopes to polytopes and clearly preserves inclusions. It does not send simplices to simplices, so to induce a functor $\mathcal D(V) \to \mathcal D(W)$, we have to expand the definition of the categories to allow suspensions of simplices.

\begin{df}
	The category $\mc A(W)$ has an object for every finite tuple of almost-disjoint polytopes in $S(W)$ of the form $\{ Q_i = P_i * S(V^\perp) \}_{i \in I}$, for some subspace $V \subseteq W$, and simplices $P_i \subseteq S(V)$. As before, each map subdivides the polytopes $Q_i$ and adds polytopes that are almost-disjoint from the $Q_i$.
	
	Note that the choice of $V$ is necessarily the minimal subspace for $Q_i$ from \autoref{minimal_suspension}, because the geometric simplex $P_i$ cannot be a suspension from a lower-dimensional space. By \autoref{weak subdivision_monotonicity}, morphisms in $\mc A(W)$ can enlarge the subspace $V$, but never shrink it.
	
	The subcategory $\mc A_\Sigma(W) \subseteq \mc A(W)$ consists of those tuples for which $V$ is a proper subspace of $W$, and $\mc D(W) \subseteq \mc A(W)$ consists of those tuples with $V = W$. The category $\mc D(W)$ is therefore just the tuples of simplices in $S(W)$, as in \autoref{sec:main_diagram}.
\end{df}

Note that the objects of $\mc A$ are the disjoint union of the objects in $\mc A_\Sigma$ and $\mc D$. As a consequence of \autoref{weak subdivision_monotonicity}, there are maps from objects in $\mc A_\Sigma$ to objects in $\mc D$, but not the other way around. Furthermore if $W \neq 0$ then $\mc D$ is cofinal in $\mc A$, so this enlargement will not affect the homotopy type of the homotopy colimits we defined previously.

As in the $\pi_0$ analysis, the suspension functor \eqref{suspension_on_category} is injective and induces an injective map
\[ \xymatrix{ \mc A(V) \ar[r]^-\Sigma & \mc A(W). } \]
We let $\mc A_{V}(W) \subseteq \mc A_\Sigma(W)$ refer to its image in $\mc A(W)$. Note that this image consists of polytopes that are suspensions of simplices from subspaces $U \subseteq V$.
\begin{lem}\label{union_hocolim}
	Any homotopy colimit over $\mc A_{\Sigma}(W)$ is the union, and also the homotopy colimit, of the corresponding homotopy colimits over $\mc A_{V}(W)$:
	\begin{align*}
		\underset{\mc A_{\Sigma}(W)}\hocolim\, F &= \bigcup_{0 \subseteq V \subsetneq W} \underset{\mc A_V(W)}\hocolim\, F \\[0.5em]
		&\cong \underset{0 \subseteq V \subsetneq W}\colim\, \underset{\mc A_V(W)}\hocolim\, F \\[0.5em]
		&\simeq \underset{0 \subseteq V \subsetneq W}\hocolim\, \underset{\mc A_V(W)}\hocolim\, F
	\end{align*}
\end{lem}

\begin{proof}
	For the equality in the first line, it suffices to show that any finite string of composable arrows in $\mc A_{\Sigma}(W)$ lies in one of the subcategories $\mc A_{V}(W)$. By \autoref{weak subdivision_monotonicity}, the minimal subspace $V$ increases along any such string. So if the final object of the string is in $\mc A_{\Sigma}(W)$, its minimal subspace $V$ is properly contained in $W$, and then every object of the string is a suspension from a subspace of $V$. Therefore the string is contained in the subcategory $\mc A_{V}(W)$.
	
	For the isomorphism and equivalence in the remaining two lines, it suffices to show that the subcomplexes are closed under intersection. By \autoref{intersect_suspensions}, any finite string of composable arrows lying in both $\mc A_{U}(W)$ and $\mc A_{V}(W)$ must also lie in $\mc A_{U \cap V}(W)$. This is enough to show that the intersection of the two homotopy colimits over $\mc A_{U}(W)$ and $\mc A_{V}(W)$ gives the homotopy colimit over $\mc A_{U \cap V}(W)$.
\end{proof}

\begin{lem}\label{suspension_main_diagram_1}
	The following diagram commutes strictly. The vertical maps are defined as in \eqref{main_diagram}, and the horizontal maps arise from the functor $\Sigma$.
\[ \xymatrix @C=4em{
	K(\Pol{S(V)}{1}) \ar[r]^-\Sigma & K(\Pol{S(W)}{1})
	\\
	\underset{\{ P_i \} \in \mc A(V)}\hocolim\, K(\mc C_{\{P_i\}}) \ar[u]_-\sim \ar[d]^-\sim \ar[r]^-\Sigma &
	\underset{\{ Q_i \} \in \mc A(W)}\hocolim\, K(\mc C_{\{Q_i\}}) \ar[u]_-\sim \ar[d]^-\sim
	\\
	\underset{\{ P_i \} \in \mc A(V)}\hocolim\, \prod_i K(\Fin^{\cong}) \ar[r]^-\Sigma &
	\underset{\{ Q_i \} \in \mc A(W)}\hocolim\, \prod_i K(\Fin^{\cong})
	\\
	\underset{\{ P_i \} \in \mc A(V)}\hocolim\, \prod_i \Sph \ar[u]_-\sim \ar[r]^-\Sigma &
	\underset{\{ Q_i \} \in \mc A(W)}\hocolim\, \prod_i \Sph \ar[u]_-\sim
} \]
\end{lem}

\begin{thm}
	$\ti K(\Pol{S(W)}{1})$ is non-equivariantly equivalent to a wedge of sphere spectra.
\end{thm}

\begin{proof}
	The last row of the diagram from \autoref{suspension_main_diagram_1} is at each spectrum level the inclusion of a subcomplex. By \autoref{union_hocolim}, the union of these subcomplexes over $0 \subseteq V \subsetneq W$ gives the homotopy colimit over the category $\mc A_{\Sigma}(W)$. We conclude that in the homotopy cofiber sequence of spectra \eqref{reduced_spherical} that defines $\ti K(\Pol{S(W)}{1})$, the first two terms can be written as
	\[ \xymatrix @C=4em{
		\underset{0 \subseteq V \subsetneq W}\colim\, \underset{\{ Q_i \} \in \mc A_{V}(W)}\hocolim\, \prod_i \Sph
		= \underset{\{ Q_i \} \in \mc A_{\Sigma}(W)}\hocolim\, \prod_i \Sph \ar[r] &
		\underset{\{ Q_i \} \in \mc A(W)}\hocolim\, \prod_i \Sph.
	} \]
	On $\pi_0$, this matches the first two terms of \eqref{pi0_exact_seq}, so it is an inclusion of free abelian groups. Choosing bases that make it the inclusion of a summand, we apply the argument from \autoref{wedge_of_spheres} and conclude that the map is equivalent to an inclusion of wedges of sphere spectra. Therefore its cofiber $\ti K(\Pol{S(W)}{1})$ is also a wedge of sphere spectra.
\end{proof}

This is the Solomon-Tits theorem for reduced spherical $K$-theory. The argument in \autoref{sec:homology} therefore applies to this variant of $K$-theory as well:
\begin{cor}\label{reduced_spherical_homology}
	Reduced spherical $K$-theory is rationally isomorphic to group homology,
\begin{align*}
	\ti K_i(\Pol{S(W)}{G}) \otimes \Q &\cong H_i(G; \St(W)) \otimes \Q.
\end{align*}
\end{cor}

\subsection{Suspension on diagrams of quotients}

The next step is to extend \autoref{suspension_main_diagram_1} all the way to the homotopy colimit of the terms $\Sigma(Q_i/\partial Q_i)$ over $\mc A(W)$. We need to modify this diagram to account for the object of $\mc A(W)$ that corresponds to the entire sphere $S(W)$.

\begin{df}\label{initial_term}
	We define the diagram $\bigvee_i \Sigma(Q_i / \partial Q_i)$ over $\mc A(W)$ in the obvious way, with one exception. At the 1-tuple on
	\[ Q = \emptyset * S(0^\perp) = S(W), \]
	we define the corresponding space in the diagram to be $S^W$, instead of
	\[ \Sigma (Q/\partial Q) = \Sigma (S(W)_+) \cong S^W/\{0,\infty\}. \]
	Since $S^W$ maps to this quotient, and there are no morphisms from other objects of $\mc A(W)$ into this one, we still get a well-defined diagram.
	
	This change does not affect the homotopy colimit, unless $W = 0$, in which case it makes the homotopy colimit into $S^0$ instead of $*$. This corresponds to the desired outcome that $K(\Pol{S(0)}{1}) = K(\Pol{\emptyset}{1})$ is the sphere spectrum $\Sph$, not the zero spectrum $*$.
\end{df}

%\begin{rmk}
%	We refer to the 1-tuple on $\emptyset * S(0^\perp) = S(W)$ as the ``initial object'' of $\mc A(W)$, although that's not quite true. It has either one morphism or no morphisms to each tuple $\{ Q_i \}_{i \in I}$, according to whether $\{ Q_i \}_{i \in I}$ covers the sphere $S(W)$ or not. So it is the initial object of the cofinal subcategory of $\mc A(W)$, consisting of those tuples $\{ Q_i \}_{i \in I}$ that cover the sphere $S(W)$.
%\end{rmk}

For any $V \subseteq W$, we have the corresponding diagram on $\mc A(V)$, whose value at a tuple of polytopes $\{P_i\}$ (that are suspensions from some $U \subseteq V$) is the wedge sum $\bigvee_i \Sigma(P_i / \partial P_i)$. We may apply suspension by the orthogonal complement of $V$, giving the diagram $\Sigma^{V^\perp} \bigvee_i \Sigma(P_i / \partial P_i)$.

\begin{df}\label{suspension_on_diagram}
	We define a map of diagrams
	\[ \xymatrix{
		\left\{ \Sigma^{V^\perp} \bigvee_i \Sigma(P_i / \partial P_i) \right\} \ar[r]^-{\Sigma} & \left\{ \bigvee_i \Sigma(Q_i / \partial Q_i) \right\}
	} \]
	over the functor $\Sigma\colon \mc A(V) \to \mc A(W)$ as follows. For each tuple of polytopes $\{ P_i \}$ in $S(V)$, with $Q_i = P_i * S(V^\perp)$, we take the homeomorphism
	\[ \Sigma^{V^\perp} \Sigma (P_i/\partial P_i) \cong \Sigma (Q_i/\partial Q_i) \]
	arising from the fact that both are quotients of the one-point compactification $S^W$ by the same subspace. At the term where $P = S^V$ and $Q = S^W$, we take the canonical isomorphism $\Sigma^{V^\perp}S^V \cong S^W$.
\end{df}

It is straightforward to check this is a well-defined map of diagrams. (Viewed in the right way, every map is a collapse map of subspaces of $S^W$, so they all agree with each other.)

\begin{lem}\label{suspension_main_diagram_2}
	The following diagram commutes strictly. The vertical maps are defined as in \eqref{main_diagram}, and the horizontal maps arise from the map of diagrams in \autoref{suspension_on_diagram}.
\[ \xymatrix @C=4em{
	\Sigma^W \underset{\{ P_i \} \in \mc A(V)}\hocolim\, \prod_i \Sph \ar[d]^-\sim \ar[r]^-\Sigma &
	\Sigma^W \underset{\{ Q_i \} \in \mc A(W)}\hocolim\, \prod_i \Sph \ar[d]^-\sim
	\\
	\underset{\{ P_i \} \in \mc A(V)}\hocolim\, \prod_i \Sigma^\infty \Sigma^{V^\perp}  \Sigma (P_i/\partial P_i) \ar[r]^-\Sigma &
	\underset{\{ Q_i \} \in \mc A(W)}\hocolim\, \prod_i \Sigma^\infty \Sigma (Q_i/\partial Q_i)
	\\
	\Sigma^\infty \underset{\{ Q_i \} \in \mc A_{V}(W)}\hocolim\, \bigvee_i \Sigma (Q_i/\partial Q_i) \ar[u]_-\sim \ar[r]^-\Sigma &
	\Sigma^\infty \underset{\{ Q_i \} \in \mc A(W)}\hocolim\, \bigvee_i \Sigma (Q_i/\partial Q_i) \ar[u]_-\sim
} \]
\end{lem}
This is a straightforward comparison of Pontryagin-Thom collapse maps. Since we have expanded the definition of $\mc A(V)$, we also check that on the new terms, the Pontryagin-Thom collapse is an equivalence. For most of these new terms, $P_i$ is contractible, so we still get a degree-one map of spheres, which is an equivalence. At the singleton term $\{Q\} = \{S(W)\}$, the collapse is the identity of $S^V$ or $S^W$, so it is clearly an equivalence.

The last row of \autoref{suspension_main_diagram_2} is again the inclusion of a subcomplex. Taking the union of these subcomplexes over $0 \subseteq V \subsetneq W$ gives the homotopy colimit over $\mc A_\Sigma(W)$, by \autoref{union_hocolim}.

\subsection{Defining a map from the quotient}

By \autoref{suspension_main_diagram_1} and \autoref{suspension_main_diagram_2}, we have an equivariant equivalence from $\Sigma^W \ti K(\Pol{S(W)}{1})$ to the suspension spectrum of the quotient
\begin{equation}\label{big_quotient}
	\left( \underset{\{ Q_i \} \in \mc A(W)}\hocolim\, \bigvee_i \Sigma (Q_i/\partial Q_i) \right) \Big/ \left( \underset{\{ Q_i \} \in \mc A_{\Sigma}(W)}\hocolim\, \bigvee_i \Sigma (Q_i/\partial Q_i) \right).
\end{equation}
To finish the proof of \autoref{intro_reduced}, it remains to define an equivariant map from this quotient into $\Sigma\ST(W)$, giving an isomorphism on the lone homology group $\St(W)$. For this purpose, we model $\ST(W)$ as the quotient
\[ \left( \underset{\emptyset \subsetneq S(U) \subseteq S(W)}\uhocolim\, {*} \right) \Big/ \left( \underset{0 \subsetneq S(U) \subsetneq S(W)}\uhocolim\, {*} \right)
= B(0 \subsetneq \bullet \subseteq W) \big/ B(0 \subsetneq \bullet \subsetneq W),
\]
where $B$ denotes classifying space of the poset of subspaces of $W$ satisfying the given conditions.

\begin{df}\label{obedient}
	A map
	\[ \xymatrix{ \underset{\{ Q_i \} \in \mc A(W)}\uhocolim\, \bigcup_i Q_i \ar[r] & B(0 \subsetneq \bullet \subseteq W) } \]
	is {\bf obedient} if for each $k$-tuple of composable morphisms in $\mc A(W)$ starting at $\{ P_i * S(U^\perp) \}$ and ending at $\{ Q_j * S(V^\perp) \}$, on the corresponding piece of the homotopy colimit
	\[ \xymatrix{ \Delta^k \times \bigcup_i (P_i * S(U^\perp)), %\ar[r] & B(0 \subsetneq \bullet \subseteq W),
	} \]
	we have the following condition at each point in $\Delta^k$:
	\begin{itemize}
		\item If $U = 0$, then $P_i * S(U^\perp) = S(W)$, and we allow it to go anywhere in $B(0 \subsetneq \bullet \subseteq W)$. \\[0.5em]
		\item If $U \neq 0$ and $V = W$, each face $(D * S(U^\perp)) \subseteq \partial (P_i * S(U^\perp))$ is sent into 
		\[ B( 0 \subsetneq \bullet \subseteq \spa(D) + U^\perp ) \subseteq B(0 \subsetneq \bullet \subsetneq W) \]
		and the entirety of $P_i * S(U^\perp)$ is sent into $B(0 \subsetneq \bullet \subseteq W)$. \\[0.5em]
		\item If $U \neq 0$ and $V \neq W$, each face $(D * S(U^\perp)) \subseteq \partial (P_i * S(U^\perp))$ is sent into 
		\[ B( V^\perp \subseteq \bullet \subseteq \spa(D) + U^\perp )  \subseteq B(0 \subsetneq \bullet \subsetneq W) \]
		and the entirety of $P_i * S(U^\perp)$ is sent into $B( V^\perp \subseteq \bullet \subsetneq W )$.
	\end{itemize}
\end{df}

Note that each of the subcomplexes of $B(0 \subsetneq \bullet \subseteq W)$ named in this definition is contractible, on account of having either an initial object or a terminal object. To show the space of obedient maps is contractible, the main thing to check is the following lemma.

\begin{lem}\label{di_compatible}
	These conditions are compatible, in the sense that along any face of $\Delta^k$, the condition we get from the corresponding smaller string of maps is stricter than the original condition.
\end{lem}

\begin{proof}
	For the faces corresponding to a composition of two maps in the string, the conditions do not change. If we delete a morphism off the end of the string, the space $V$ may shrink. The only way this can change the above conditions is if we move from $V = W$ to $V \neq W$. Our maps gain the condition of having to land in subspaces that contain $V^\perp \neq 0$, which is stricter than landing in some nonzero subspace.
	
	If we delete a morphism off the beginning, then the polytopes $P_i * S(U^\perp)$ get subdivided, and as a result the minimal space $U$ can get larger. If this moves from $U = 0$ to $U \neq 0$, then the condition of course gets stricter, because there was no condition before. Otherwise, $U \neq 0$, and $P_i * S(U^\perp)$ gets subdivided into pieces, giving many more faces. It is clear that the new faces in the interior have a stricter condition than before. On the exterior, each face $D * S(U^\perp)$ gets cut into smaller faces of the form
	\[ D_i * S(U_i^\perp) \subseteq D * S(U^\perp). \]
	But then
	\[ \spa(D_i) + U_i^\perp = \spa(D_i * S(U_i^\perp)) \subseteq \spa(D * S(U^\perp)) = \spa(D) + U^\perp, \]
	so the conditions on each of these new faces also get stricter.
\end{proof}

\begin{lem}
	The space of obedient maps is weakly contractible.
\end{lem}
\begin{proof}
	The proof is identical to that of \autoref{apartment-like_hocolim}, using \autoref{di_compatible} for compatibility.
\end{proof}

\begin{rmk}
As before, obedient maps are not equivariant, but the conjugation $G$-action preserves obedient maps, because $G$ preserves polytopes, joins, spans, and orthogonal complements. As a result, as in \autoref{make_equivariant}, whenever we multiply the source by $EG$, we get a contractible space of equivariant maps.
\end{rmk}

The definition of ``obedient'' agrees with ``apartment-like'' in the following sense. Consider the square of maps
\begin{equation}\label{eq_i_di_compatibility}
\xymatrix{
		\underset{\{ Q_i \} \in \mc D(W)}\uhocolim\, \bigcup_i Q_i \ar[d]^-\sim \ar[r] & \underset{0 \subsetneq \bullet \subseteq W}\uhocolim\, \Tpl(\bullet) \ar[d] \\
		\underset{\{ Q_i \} \in \mc A(W)}\uhocolim\, \bigcup_i Q_i \ar[r] & \underset{0 \subsetneq \bullet \subseteq W}\uhocolim\, {*} \ar@{=}[r] & B(0 \subsetneq \bullet \subseteq W)
}
\end{equation}
where the top horizontal is any apartment-like map the sense of \autoref{apartment-like_change}, and the bottom horizontal is any obedient map in the sense of \autoref{obedient}. The left horizontal is the inclusion, and is an equivalence because $\mc D$ is cofinal in $\mc A$.
\begin{lem}\label{i_di_compatibility}
	Any square of the form \eqref{eq_i_di_compatibility} commutes up to homotopy.
\end{lem}

\begin{proof}
	We only need to consider the obedient map on the terms where $U = V = W$, and therefore only second condition in \autoref{obedient} applies. This agrees the definition of apartment-like from \autoref{apartment-like_change}, except with the complexes $\Tpl(U^k \cap D)$ collapsed to a point. Therefore the obedient map can be chosen so that the square commutes strictly.
\end{proof}

Suppose we choose obedient and apartment-like maps as in \eqref{eq_i_di_compatibility}. Taking quotients and suspending gives a commuting square
\begin{equation}\label{eq_i_di_compatibility_2}
\xymatrix{
		\underset{\{ Q_i \} \in \mc D(W)}\hocolim\, \bigvee_i \Sigma(Q_i/\partial Q_i) \ar[d]^-\sim \ar[r]^-\sim & \Sigma\left(\underset{0 \subsetneq \bullet \subseteq W}\uhocolim\, \Tpl(\bullet) \big/ \underset{0 \subsetneq \bullet \subsetneq W}\uhocolim\, \Tpl(\bullet)\right) \ar[d] \\
		\underset{\{ Q_i \} \in \mc A(W)}\hocolim\, \bigvee_i \Sigma(Q_i/\partial Q_i) \ar[r] & \Sigma\left( B(0 \subsetneq \bullet \subseteq W) \big/ B(0 \subsetneq \bullet \subsetneq W)\right).
}
\end{equation}
Actually, we have to say a little more in order to define the bottom horizontal map on the term $Q = S(W)$ of $\mc A(W)$. We take as our definition that the bottom horizontal map of \eqref{eq_i_di_compatibility_2} quotients this term down to $\Sigma (Q/\partial Q)$, then applies the suspension of the obedient map. There is no such issue for the left-hand vertical map because the term $Q = S(W)$ lies outside the image of this map.

The square \eqref{eq_i_di_compatibility_2} gives on homology
\begin{equation}\label{eq_i_di_compatibility_3}
\xymatrix{
		\Pt(W) \ar@{=}[d] \ar@{=}[r] & \Pt(W) \ar[d] \\
		\Pt(W) \ar[r] & \St(W).
}
\end{equation}
The right-hand vertical map of \eqref{eq_i_di_compatibility_3} is the standard surjection $\Pt(W) \to \St(W)$, but the square commutes, so the bottom horizontal map also induces the standard surjection $\Pt(W) \to \St(W)$.

\begin{rmk}\label{dirty_trick}
To finish the proof, it suffices to take the bottom horizontal map of \eqref{eq_i_di_compatibility_2} and deform it to zero on the subcomplex that is the hocolim over $\mc A_\Sigma$. This creates the desired map from the quotient \eqref{big_quotient} into $\Sigma\ST(W)$. The induced map $\St(W) \to \St(W)$ lies underneath the identity of $\Pt(W)$, and therefore it is the identity map as well, so we get our equivalence to $\Sigma\ST(W)$ as desired.
\end{rmk}

\begin{prop}\label{big_nullhomotopy}
	Any map
	\[ \xymatrix{
		\underset{\{ Q_i \} \in \mc A(W)}\hocolim\, \bigvee_i \Sigma(Q_i/\partial Q_i) \ar[r] & \Sigma\left( B(0 \subsetneq \bullet \subseteq W) \big/ B(0 \subsetneq \bullet \subsetneq W)\right)
	} \]
	induced by a obedient map, is nullhomotopic when restricted to the homotopy colimit over $\mc A_\Sigma(W)$.
\end{prop}

%\begin{lem}
%	If $K$ is a contractible complex and $S$ is a finite complex then the space of maps $\Map(S,K)$ deformation retracts onto the subspace of maps that are constant, i.e.~$K \subseteq \Map(S,K)$.
%\end{lem}

\begin{proof}
	By \autoref{obedient}, for any $k$-tuple of composable morphisms in $\mc A_{\Sigma}$ starting at $\{ P_i * S(U^\perp) \}$ and ending at $\{ Q_j * S(V^\perp) \}$, if $V \neq W$ and $U \neq 0$ then the corresponding piece of the homotopy colimit
	\[ \xymatrix{ \Delta^k \times \bigcup_i P_i * S(U^\perp) \ar[r] & B(0 \subsetneq \bullet \subseteq W) } \]
	lands in $B(0 \subsetneq \bullet \subsetneq W)$, and therefore goes to zero once $B(0 \subsetneq \bullet \subsetneq W)$ is quotiented out. Since this is $\mc A_\Sigma$, $V \neq W$ always, but sometimes $U = 0$. So the only remaining pieces of the homotopy colimit are those strings of maps beginning at the 1-tuple $\{ S(W)  = \emptyset * S(0^\perp) \}$, with $U = 0$.
	
	Each of the remaining pieces of the homotopy colimit is of the form $\Delta^k \times S(W)$, joined along faces that all contain the initial vertex of $\Delta^k$. Let $T$ denote the union of these simplices, with $\partial T$ the union of the faces of these simplices that are opposite the initial vertex. The rules described in \autoref{obedient} give us a map of pairs
	\begin{equation}\label{last_step}  \xymatrix{ (S(W) \times T,S(W) \times \partial T) \ar[r] & \left( B(0 \subsetneq \bullet \subseteq W) , B(0 \subsetneq \bullet \subsetneq W)\right) } \end{equation}
	By the discussion after \eqref{eq_i_di_compatibility_2}, the map we want on homotopy colimits is obtained from this by adding a disjoint basepoint to $S(W)$ and taking reduced suspension. We get the same result if, instead, we take unreduced suspension of the $S(W)$ on the left and the classifying spaces on the right, then pass to the quotient:
	\[ \xymatrix @R=0.5em{ (S^W \times T)/((* \times T) \cup (S^W \times \partial T)) \ar[r] & \left( S B(0 \subsetneq \bullet \subseteq W) \big/ S B(0 \subsetneq \bullet \subsetneq W)\right) \\
	& \cong \Sigma\left(B(0 \subsetneq \bullet \subseteq W) \big/ B(0 \subsetneq \bullet \subsetneq W)\right). } \]
	The check of this amounts to confirming that both the top and the bottom of the unreduced suspension are sent to the basepoint after everything is quotiented out.
	
	So, if we modify the map of pairs \eqref{last_step} up to homotopy, this will induce a homotopy on the quotient. This quotient is homeomorphic to the original homotopy colimit, quotiented out by those terms that are already sent to the basepoint. So this will induce a homotopy on the original homotopy colimit, as desired.
	
	We deform the map of pairs \eqref{last_step} to one that is constant in $S(W)$, i.e.~factors through
	\[ \xymatrix{ (S(W) \times T,S(W) \times \partial T) \ar[r] & ({*} \times T, {*} \times \partial T). } \]
	We do this inductively on each face $F$ of $\partial T$. Each face $F \subseteq \partial T$ corresponds to some tuple of composable maps ending at $\{ Q_j * S(V^\perp) \}$. The corresponding map
	\[ S(W) \times F \to B(0 \subsetneq \bullet \subsetneq W) \]
	 lands in the contractible subcomplex $B(V^\perp \subseteq \bullet \subsetneq W)$.
	 
	 Consider the following general principle: if $K$ is a contractible complex, then the space of maps $S(W) \to K$ and the subspace of constant maps $S(W) \to {*} \to K$ are both contractible. Hence, any map from $F$ to the larger space, such that $\partial F$ lands in the smaller space, can be deformed rel $\partial F$ to land in the smaller space. Applying this to $K = B(V^\perp \subseteq \bullet \subsetneq W)$, we see that we can deform the map \eqref{last_step} on $S(W) \times F$, relative to $S(W) \times \partial F$, so that it is constant in $S(W)$.
	
	After this induction, the map of pairs \eqref{last_step} on $S(W) \times \partial T$ factors through $* \times \partial T$. We perform the inductive step one more time to $S(W) \times T$, using maps to the contractible complex $B(0 \subsetneq \bullet \subseteq W)$, and arrive at a map of pairs that is now constant in $S(W)$ across all of $T$.
	
	Once the map of pairs \eqref{last_step} factors through $* \times T$, we return to the unreduced suspension
	\[ \xymatrix @R=0.5em{ (I \times T)/((\{0\} \times T) \cup (I \times \partial T)) \ar[r] & \left( S B(0 \subsetneq \bullet \subseteq W) \big/ S B(0 \subsetneq \bullet \subsetneq W)\right) \\
	& \cong \Sigma\left(B(0 \subsetneq \bullet \subseteq W) \big/ B(0 \subsetneq \bullet \subsetneq W)\right). } \]
	Since $T$ is a cell complex, $I \times T$ deformation retracts onto $(\{0\} \times T) \cup (I \times \partial T)$, so this map is nullhomotopic.
\end{proof}

\begin{rmk}
	One might hope to make the above proof shorter, perhaps by adjusting the definition of obedient (\autoref{obedient}) to make the map vanish on all of $\mc A_\Sigma$, without the need for an additional homotopy. However, an examination of low-dimensional examples suggests that this is not possible. The term $Q = S(W)$ is to blame for this additional step in the proof. However, as mentioned earlier, the theorem would not be true without it.
\end{rmk}

\begin{lem}
	The nullhomotopy of the previous lemma can be made equivariant, if the colimit on the left is smashed with $EG_+$ first.
\end{lem}

\begin{proof}
	As before, we take any map from $EG$ into the contractible space of obedient maps. The adjoint of this is a map
	\[ \xymatrix{
		EG_+ \sma \underset{\{ Q_i \} \in \mc A(W)}\hocolim\, \bigvee_i \Sigma(Q_i/\partial Q_i) \ar[r] & \Sigma\left( B(0 \subsetneq \bullet \subseteq W) \big/ B(0 \subsetneq \bullet \subsetneq W)\right)
	} \]
	that is equivariant, and that for every point of $EG$ comes from a obedient map. The map vanishes on most of the homotopy colimit, as before, and we reduce to defining an equivariant homotopy of maps
	\[ \xymatrix{ EG_+ \sma (S^W \times T)/((* \times T) \cup (S^W \times \partial T)) \ar[r] & \Sigma\left(B(0 \subsetneq \bullet \subseteq W) \big/ B(0 \subsetneq \bullet \subsetneq W)\right) } \]
	coming from an equivariant map of pairs
	\begin{equation}\label{last_step_eq}  \xymatrix{ (EG \times S(W) \times T, EG \times S(W) \times \partial T) \ar[r] & \left( B(0 \subsetneq \bullet \subseteq W) , B(0 \subsetneq \bullet \subsetneq W)\right). } \end{equation}
	The argument proceeds as before, but inducting over the \emph{pairs} of cells $F \to T$ and free $G$-cells $G \times D^n \to EG$. We deform the map on $D^n \times F$ to be constant in $S(W)$, by a non-equivariant homotopy rel $\partial(D^n \times F)$. This automatically extends to an equivariant homotopy on $G \times D^n \times F$, allowing us to do the inductive step equivariantly.
	
	After suspending, for the final homotopy we use that $EG \times I \times T$ equivariantly deformation retracts onto $EG \times (\{0\} \times T) \cup (I \times \partial T)$.
\end{proof}

This equivariant nullhomotopy defines a map from $EG_+$ smashed with the quotient \eqref{big_quotient} to $\Sigma\ST(W)$. As discussed in \autoref{dirty_trick}, this induced map is an equivalence of spectra
\[ \xymatrix{ \Sigma^W \ti K(\Pol{S(W)}{1}) \ar[r]^-\sim & \Sigma^\infty \Sigma\ST(W). } \]
De-suspending by $W$ and taking homotopy orbits, this finishes the proof of \autoref{intro_reduced}.

\section{Examples}\label{sec:examples}

We conclude with a few low-dimensional examples and direct computations.

\begin{ex}
	The empty geometry $\emptyset$. As mentioned before, 
	\[ K(\Pol{\emptyset}{1}) \simeq K(\Fin^{\cong}) \simeq \Sph. \]
	The empty geometry $\emptyset$ is also the unit sphere in $\R^0$. There are no proper subspaces to reduce by, so the reduced spherical $K$-theory is also the sphere:
	\[ \ti K(\Pol{S(\R^0)}{1}) \simeq \Sph. \]
	There are no nontrivial isometries, so there is nothing more to calculate.
\end{ex}

\begin{ex}
	The one-point geometry $E^0 \cong H^0$. We calculate
	\[ \PT(E^0) \simeq \ST(E^0) \simeq S^0 \]
	and so by \autoref{intro_main},
	\[ K(\Pol{E^0}{1}) \simeq \Sph. \]
	This could also be deduced by identifying it with $K(\Fin^{\cong})$, as in the previous example. %The isometry group $E(0)$ is trivial, so there is nothing more to calculate.
\end{ex}

\begin{ex}
	The two-point geometry $S^0$. In this case
	\[ \PT(S^0) \simeq S^0 \vee S^0, \qquad \ST(S^0) \simeq S^0. \]
	The $O(1)$-action on $\PT(S^0)$ swaps the two spheres. After suspending by the virtual bundle
	\[ -TS^0 := 1 - T(\R^1) = 1 - \sigma, \]
	where $\sigma$ denotes the sign representation, the action swaps and negates. Similarly, on $\ST(S^0)$ the action negates the fundamental class. We conclude using \autoref{intro_main} and \autoref{intro_reduced} that
	\[ \begin{array}{rcl}
		K(\Pol{S^0}{1}) &\simeq& \Sph \vee \Sph, \\
		K(\Pol{S^0}{O(1)}) &\simeq& \Sph, \\[0.5em]
		\ti K(\Pol{S^0}{1}) &\simeq& \Sph, \\
		\ti K(\Pol{S^0}{O(1)}) &\simeq& \Sigma^{1-\sigma}_+ BO(1).
	\end{array} \]
	
	This latter spectrum has homology equal to $H_n(BO(1);\Z^t)$, where the $t$ indicates $O(1)$ acting by negation. This homology is $\Z/2$ in every even degree and zero in every odd degree. In particular, the rationalization vanishes.
\end{ex}

\begin{ex}\label{main_computation}
	The noncompact 1-dimensional geometry $E^1 \cong H^1$. We calculate
	\[ \PT(E^1) \simeq \ST(E^1) \simeq S(\R^\delta) \]
	where the last term is the unreduced suspension of the reals with the discrete topology, not the unit sphere in $\R^\delta$. Desuspending by the (trivial) tangent bundle gives
	\[ K(\Pol{E^1}{1}) \simeq \bigvee_{\R^\delta \setminus \{0\}} \Sph. \]
	The homotopy groups for $K(\Pol{E^1}{T(1)})$ and $K(\Pol{E^1}{E(1)})$ follow easily from \autoref{intro_translation_homology}, but we will also show how to derive them directly.
	
	The action of $T(1)$ on $S(\R^\delta)$ is the obvious one that makes $\R^\delta \cong T(1)$ a single free orbit. Taking based homotopy $T(1)$-orbits gives the reduced suspension spectrum $\Sigma^\infty BT(1)$. (The top of the suspension gives an unbased copy of $BT(1)$, while the middle of the suspension becomes a single interval connecting $BT(1)$ to the basepoint.) We conclude
	\[ K(\Pol{E^1}{T(1)}) \simeq \Sigma^{-1}\Sigma^\infty BT(1). \]
	The reduced homology of $T(1) \cong \R$ is $\R^{\wedge i}$ in degree $i$, see e.g.~\cite[4.7]{dupont_book}. Since these are rational, the $K$-theory spectrum is rational and these are the homotopy groups as well:
	\[ \begin{array}{rcl}
		K_n(\Pol{E^1}{T(1)}) &\cong& \R^{\wedge (n+1)}.
	\end{array} \]
	Finally, since this spectrum is rational, the homotopy $\Z/2$-orbits simply take $\Z/2$ coinvariants on the homotopy groups. We check that the action is by negation on the odd homotopy groups and the identity on the even homotopy groups, and conclude
	\[ \begin{array}{rcl}
		K_{2n}(\Pol{E^1}{E(1)}) &\cong& \R^{\wedge (2n+1)}, \\
		K_{2n+1}(\Pol{E^1}{E(1)}) &\cong& 0.
	\end{array} \]
	This finishes the proof of \autoref{intro_dim_one}.
\end{ex}

\begin{rmk}
	This example is helpful for explaining why we introduced the apartment-like maps in \autoref{apartment-like}. We can see directly that $S(\R^\delta)$ is equivalent to the homotopy colimit of the wedges $\bigvee_i P_i/\partial P_i$ over the almost-disjoint intervals $P_i \subseteq \R$. Each $P_i / \partial P_i$ traverses the two intervals in the suspension given by the endpoints of $P_i$. When we subdivide $P_i$ into two intervals, this is homotopic to the wedge of the corresponding maps for each of the smaller intervals.
	
	However, the cleanest way to describe this homotopy is to name a contractible space that both maps live in, the space of maps into $C(\R^\delta)$ (cone point at the top) in which the endpoints of each interval go to the corresponding points at the bottom of the cone, and the rest of the interval goes anywhere inside the cone $C(P_i^\delta)$. This is what the space of apartment-like maps (at least, the version described in \autoref{apartment-like_change}) achieves.
\end{rmk}

\begin{ex}
	We next consider 1-dimensional spherical geometry $S^1$. We identify $SO(2)$ with $(\R/\Z)^\delta$. Recall that $\PT(E^1)$ is the unreduced suspension of $\R^\delta$, equivalently the homotopy pushout of the following diagram.
	\[ \xymatrix @R=1.8em{
		\R^\delta \ar[d] \ar[r] & \R \\
		{*}.
	} \]
	If we take based homotopy $\Z$-orbits, we get the homotopy pushout of the following diagram.
	\[ \xymatrix @R=1.8em{
		(\R/\Z)^\delta \ar[d] \ar[r] & (\R/\Z) \\
		{*}.
	} \]
	This agrees with polytope complex $\PT(S^1)$, hence we have an equivalence of $O(2)$-spectra
	\[ \PT(S^1) \simeq \PT(E^1)_{h\Z}. \]
	Therefore without isometries we get
	\[ K(\Pol{S^1}{1}) \simeq \bigvee_{(\R/\Z)^\delta} \Sph, \]
	and with isometries we get the same answer as in the Euclidean case,
	\[ \begin{array}{rcl}
		K_n(\Pol{S^1}{SO(2)}) &\cong& \R^{\wedge (n+1)}, \\[0.5em]
		K_{2n}(\Pol{S^1}{O(2)}) &\cong& \R^{\wedge (2n+1)}, \\
		K_{2n+1}(\Pol{S^1}{O(2)}) &\cong& 0.
	\end{array} \]
\end{ex}

\begin{ex}
	We next consider 1-dimensional reduced spherical geometry. Continuing to identify $SO(2)$ with $(\R/\Z)^\delta$, we calculate
	\[ \xymatrix{ \ST(S^1) \simeq S\left(\left(\R/\frac12\Z\right)^\delta\right), } \]
	\[ \ti K(\Pol{S^1}{1}) \simeq \Sigma^{-1}\ST(S^1) \simeq \bigvee_{(\R/\frac12\Z)^\delta \setminus \{0\}} \Sph. \]
	It's easiest to describe this equivariantly by the homotopy fiber sequence
	\[ \xymatrix{ \ti K(\Pol{S^1}{1}) \ar[r] & \Sigma^\infty_+ \left(\R/\frac12\Z\right)^\delta \ar[r] & \Sph. } \]
	Taking homotopy orbits gives the fiber sequence
	\[ \xymatrix @R=0.5em{
		\ti K(\Pol{S^1}{SO(2)}) \ar[r] & \Sigma^\infty_+ B(\Z/2) \ar[r] & \Sigma^\infty_+ B\left(\R/\Z\right)^\delta. \\
%		\ti K(\Pol{S^1}{O(2)}) \ar[r] & \Sigma^\infty_+ B(\Z/2 \times \Z/2) \ar[r] & \Sigma^\infty_+ BO(2)^\delta. [[ twist is wrong ]]
	} \]
	This uses the Shapiro lemma, $(G/H)_{hG} = BH$. In more detail, $(\R/\frac12\Z)^\delta$ is all lines through the origin in $\R^2$. This is of the form $G/H$ where $G = SO(2)$, acting transitively on $(\R/\frac12\Z)^\delta$, and $H = \Z/2$, the stabilizer of a single line. Therefore the homotopy $G$-orbits of the set $(\R/\frac12\Z)^\delta$ is the space $B(\Z/2)$.
	
	Using \cite[4.7]{dupont_book} and the short exact sequence $0 \to \Z/2 \to \Q/\Z \to \Q/\frac12\Z \to 0$, we get
	\[ \begin{array}{c|cccccc}
		& H_0 & H_1 & H_2 & H_3 & \\\hline
		H_*((\R/\Z)^\delta) & \Z & \R/\Q \oplus \Q/\Z & (\R/\Q)^{\wedge 2} & (\R/\Q)^{\wedge 3} \oplus \Q/\Z & \cdots \\
		H_*(\ti K(\Pol{S^1}{SO(2)})) & \R/\Q \oplus \Q/\frac12\Z & (\R/\Q)^{\wedge 2} & (\R/\Q)^{\wedge 3} \oplus \Q/\frac12\Z & (\R/\Q)^{\wedge 4} & \cdots \\
%		H_*(\ti K(\Pol{S^1}{O(2)})) & \R/\Q \oplus \Q/\frac12\Z & \Z/2 & (\R/\Q)^{\wedge 3} \oplus \Q/\frac12\Z & (\Z/2)^2? & \cdots \\
%		\ti K_*(\Pol{S^1}{SO(2)}) \otimes \Q & \R/\Q & (\R/\Q)^{\wedge 2} & (\R/\Q)^{\wedge 3} & (\R/\Q)^{\wedge 4} & \cdots \\
%		\ti K_*(\Pol{S^1}{O(2)}) \otimes \Q & \R/\Q & 0 & (\R/\Q)^{\wedge 3} & 0 & \cdots \\
	\end{array} \]
	
	Note that $\ti K(\Pol{S^1}{SO(2)})$ and $\ti K(\Pol{S^1}{O(2)})$ are not rational---the $K_0$ groups are both $\R/\frac12\Z$, which is not a rational vector space. The higher homotopy groups are therefore more difficult to calculate. We will be content to give the rational homotopy groups, which follow from the homology computation above:
	\[ \begin{array}{rcl}
		\ti K_n(\Pol{S^1}{SO(2)}) \otimes \Q &\cong& (\R/\Q)^{\wedge (n+1)}, \\[0.5em]
		\ti K_{2n}(\Pol{S^1}{O(2)}) \otimes \Q &\cong& (\R/\Q)^{\wedge (2n+1)}, \\
		\ti K_{2n+1}(\Pol{S^1}{O(2)}) \otimes \Q &\cong& 0.
	\end{array} \]
\end{ex}

\begin{ex}\label{translational_case}
	Finally, we consider $n$-dimensional Euclidean space $E^n$, but with $G$ equal to the translation group $T(n)$. By \autoref{intro_translation_homology}, we have
	\[ K_m(\Pol{E^n}{T(n)}) \cong H_m(T(n); \St(E^n)). \]
	We recall a splitting of these groups due to Dupont, see \cite[Thm 1.1]{dupont_82} and \cite[Thm 4.10]{dupont_book}. %Although it is stated in \cite{dupont_book} just for $H_0$, the splitting applies to homology in every degree.
%	 The idea is to filter the Steinberg complex of $E^n$ by the dimensions of the subspaces, giving a spectral sequence in which the $E^2$ page is expressed in terms of the Tits complex of \emph{linear} subspaces of $\R^n$. But the spectral sequence collapses at the $E^2$ page, since all remaining differentials go between different eigenspaces for the operator that dilates all the polytopes by a factor of $\lambda$. As an aside, it appears we get the same spectral sequence if we apply the cubical spectral sequence from \cite{munson_volic} to the cube \eqref{translation_presentation}.
	Let $\T(\R^n)$ denote the Tits complex of linear subspaces of $\R^n$. Consider the ``local coefficient system'' $\Lambda^q_\Q(\mathfrak g)$ defined on the suspended Tits or Steinberg complex
	\[ \ST(\R^n) = B(0 \subsetneq \bullet \subseteq \R^n) \big/ B(0 \subsetneq \bullet \subsetneq \R^n) =  \CT(\R^n) / \T(\R^n) \]
	by assigning the simplex $V_0 \subseteq \cdots \subseteq V_k$ to the $q$th rational exterior power $\Lambda^q_\Q(V_0)$.
	
	This is not a local coefficient system in the usual sense, since these groups do not form a bundle over the complex $\CT(\R^n)$. However, every face map in this complex gives an evident homomorphism. For instance, deleting the 0th vertex has the effect of applying the inclusion $V_0 \subseteq V_1$, giving a map $\Lambda^q_\Q(V_0) \to \Lambda^q_\Q(V_1)$. This is enough to define a chain complex with coefficients in these groups.
	
	We reduce the chain complex by modding out by those simplices in the subspace $\T(\R^n)$ and take homology, calling the result
	\[ \ti H_*(\ST(\R^n); \Lambda^q_\Q(\mathfrak g)). \]
	Dupont's chain complex is defined instead using the simplices in $\T(\R^n)$, and then augmented by the group $\Lambda^q_\Q(\R^n)$. Our homology groups of $\ST(\R^n)$ are therefore the same as his homology groups of $\T(\R^n)$, but one degree higher:
	\[ \ti H_m\left(\ST(\R^n); \Lambda^q_\Q(\mathfrak g)\right) \cong H_{m-1}\left(\T(\R^n); \Lambda^q_\Q(\mathfrak g)\right). \]
	With this adjustment to the notation, Dupont's splitting is written as
	\begin{equation}\label{translation_splitting}
		H_m(T(n); \St(E^n)) \cong \bigoplus_{q=1}^{n} \ti H_{n-q}\left(\ST(\R^n);\Lambda^{m+q}_\Q(\mathfrak g)\right).
	\end{equation}
	The $q$th term of this splitting is the eigenspace of the ``dilation by $a$'' operator, with eigenvalue $a^{m+q}$, for any positive rational number $a$.
	
	Concretely, for $n = 2$ this gives
	\begin{align*}
		K_m(\Pol{E^2}{T(2)}) \cong & \ker\left(\bigoplus_{0 \subsetneq V \subsetneq \R^2} \Lambda^{m+1}_\Q (V) \to \Lambda^{m+1}_\Q (\R^2) \right) \\
		&\oplus \coker\left(\bigoplus_{0 \subsetneq V \subsetneq \R^2} \Lambda^{m+2}_\Q (V) \to \Lambda^{m+2}_\Q (\R^2) \right),  \\
	\end{align*}
	where the sum is taken over all one-dimensional vector subspaces, i.e.~lines through the origin. The reader is invited to check this result against the proof of \autoref{rationality_proof}, which gives a cofiber sequence of rational spectra
	\[ \xymatrix{
		\displaystyle \bigvee_{0 \subsetneq V \subsetneq \R^2} \Sigma^\infty BT(V) \ar[r] & \Sigma^\infty BT(2) \ar[r] & \Sigma^2 K(\Pol{E^2}{T(2)}).
	} \]
	
	Returning to \eqref{translation_splitting} for general $n$, when $m < 0$, these groups are trivial, and when $m = 0$ the homology is unchanged if we take exterior powers over $\R$, see \cite[Thm 4.14]{dupont_book}. This allows for a considerable simplification of the terms in the splitting, see \cite[Cor 4.15]{dupont_book}. However, the argument does not apply in the range $m > 0$, where the exterior powers must be taken over $\Q$. It remains to be seen if these higher groups can be simplified in some other way.

\end{ex}

\bibliographystyle{amsalpha}
\bibliography{references}{}

@misc{kkmmw1,
	AUTHOR = {Klang, Inbar and Kuijper, Josefien  and Malkiewich, Cary and Mehrle, David and Wittich, Thor},
	TITLE = {Higher spherical scissors congruence {I}: {H}opf algebra},
	NOTE = {\href{https://arxiv.org/abs/2509.18009}{arXiv:2509.18009}},
	YEAR = {2025},
}

@misc{kkmmw2,
	AUTHOR = {Klang, Inbar and Kuijper, Josefien  and Malkiewich, Cary and Mehrle, David and Wittich, Thor},
	TITLE = {Higher spherical scissors congruence {II}: {C}omodules},
	NOTE = {to appear},
	YEAR = {2026},
}

@incollection {neumann_notes,
    AUTHOR = {Neumann, W. D.},
     TITLE = {Notes on geometry and 3-manifolds},
 BOOKTITLE = {Low dimensional topology ({E}ger, 1996/{B}udapest, 1998)},
    SERIES = {Bolyai Soc. Math. Stud.},
    VOLUME = {8},
     PAGES = {191--267},
      NOTE = {With appendices by Paul Norbury},
 PUBLISHER = {J\'{a}nos Bolyai Math. Soc., Budapest},
      YEAR = {1999},
   MRCLASS = {57N10 (57M50)},
  MRNUMBER = {1747270},
MRREVIEWER = {James W. Anderson},
}

@article {thurston_problems,
    AUTHOR = {Thurston, William P.},
     TITLE = {Three-dimensional manifolds, {K}leinian groups and hyperbolic
              geometry},
   JOURNAL = {Bull. Amer. Math. Soc. (N.S.)},
  FJOURNAL = {American Mathematical Society. Bulletin. New Series},
    VOLUME = {6},
      YEAR = {1982},
    NUMBER = {3},
     PAGES = {357--381},
      ISSN = {0273-0979},
   MRCLASS = {57N10 (20H15 30F40 57M35 57M40 57S17)},
  MRNUMBER = {648524},
MRREVIEWER = {Klaus Johannson},
       DOI = {10.1090/S0273-0979-1982-15003-0},
       URL = {https://doi.org/10.1090/S0273-0979-1982-15003-0},
}

@article {wallace_bolyai_gerwien,
    AUTHOR = {Gardner, R. J.},
     TITLE = {A problem of {S}allee on equidecomposable convex bodies},
   JOURNAL = {Proc. Amer. Math. Soc.},
  FJOURNAL = {Proceedings of the American Mathematical Society},
    VOLUME = {94},
      YEAR = {1985},
    NUMBER = {2},
     PAGES = {329--332},
      ISSN = {0002-9939,1088-6826},
   MRCLASS = {52A10 (52A15)},
  MRNUMBER = {784187},
MRREVIEWER = {Stan\ Wagon},
       DOI = {10.2307/2045399},
       URL = {https://doi.org/10.2307/2045399},
}

@incollection {solomon_tits,
	AUTHOR = {Solomon, L.},
	TITLE = {The {S}teinberg character of a finite group with {$BN$}-pair},
	BOOKTITLE = {Theory of {F}inite {G}roups ({S}ymposium, {H}arvard {U}niv.,
	{C}ambridge, {M}ass., 1968)},
	PAGES = {213--221},
	PUBLISHER = {W. A. Benjamin, Inc., New York-Amsterdam},
	YEAR = {1969},
	MRCLASS = {20.25},
	MRNUMBER = {246951},
	MRREVIEWER = {T. Ono},
}

@article {church_farb_putman_integrality,
    AUTHOR = {Church, Thomas and Farb, Benson and Putman, Andrew},
     TITLE = {Integrality in the {S}teinberg module and the top-dimensional
              cohomology of {${\rm SL}_n\mathcal O_K$}},
   JOURNAL = {Amer. J. Math.},
  FJOURNAL = {American Journal of Mathematics},
    VOLUME = {141},
      YEAR = {2019},
    NUMBER = {5},
     PAGES = {1375--1419},
      ISSN = {0002-9327,1080-6377},
   MRCLASS = {20E42 (20G10 51E24)},
  MRNUMBER = {4011804},
MRREVIEWER = {Matthias\ Wendt},
       DOI = {10.1353/ajm.2019.0036},
       URL = {https://doi.org/10.1353/ajm.2019.0036},
}

@article {church_putman_one,
    AUTHOR = {Church, Thomas and Putman, Andrew},
     TITLE = {The codimension-one cohomology of {${\rm SL}_n\Bbb Z$}},
   JOURNAL = {Geom. Topol.},
  FJOURNAL = {Geometry \& Topology},
    VOLUME = {21},
      YEAR = {2017},
    NUMBER = {2},
     PAGES = {999--1032},
      ISSN = {1465-3060,1364-0380},
   MRCLASS = {11F75 (20E42 20G30 20J06 57Q05)},
  MRNUMBER = {3626596},
MRREVIEWER = {Steffen\ Kionke},
       DOI = {10.2140/gt.2017.21.999},
       URL = {https://doi.org/10.2140/gt.2017.21.999},
}

@article {lee_szczarba,
    AUTHOR = {Lee, Ronnie and Szczarba, R. H.},
     TITLE = {On the homology and cohomology of congruence subgroups},
   JOURNAL = {Invent. Math.},
  FJOURNAL = {Inventiones Mathematicae},
    VOLUME = {33},
      YEAR = {1976},
    NUMBER = {1},
     PAGES = {15--53},
      ISSN = {0020-9910,1432-1297},
   MRCLASS = {22E40 (20H05)},
  MRNUMBER = {422498},
MRREVIEWER = {Kenneth\ S.\ Brown},
       DOI = {10.1007/BF01425503},
       URL = {https://doi.org/10.1007/BF01425503},
}

@article {szymik_wahl,
	AUTHOR = {Szymik, M. and Wahl, N.},
	TITLE = {The homology of the {H}igman-{T}hompson groups},
	JOURNAL = {Invent. Math.},
	FJOURNAL = {Inventiones Mathematicae},
	VOLUME = {216},
	YEAR = {2019},
	NUMBER = {2},
	PAGES = {445--518},
	ISSN = {0020-9910},
	MRCLASS = {19D23 (20J05)},
	MRNUMBER = {3953508},
	MRREVIEWER = {Fernando Muro},
	DOI = {10.1007/s00222-018-00848-z},
	URL = {https://doi.org/10.1007/s00222-018-00848-z},
}

@unpublished {saf1,
    AUTHOR = {Sah, Chin-Han},
     TITLE = {Scissors congruences of the interval},
     NOTE = {Preprint},
     YEAR = 1981,
}

@incollection {saf2,
    AUTHOR = {Arnoux, Pierre},
     TITLE = {\'{E}changes d'intervalles et flots sur les surfaces},
 BOOKTITLE = {Ergodic theory ({S}em., {L}es {P}lans-sur-{B}ex, 1980)
              ({F}rench)},
    SERIES = {Monogr. Enseign. Math.},
    VOLUME = {29},
     PAGES = {5--38},
 PUBLISHER = {Univ. Gen\`eve, Geneva},
      YEAR = {1981},
   MRCLASS = {28D99 (58F11)},
  MRNUMBER = {609891},
MRREVIEWER = {Michael Keane},
}

@article {saf3,
    AUTHOR = {Veech, William A.},
     TITLE = {The metric theory of interval exchange transformations. {III}.
              {T}he {S}ah-{A}rnoux-{F}athi invariant},
   JOURNAL = {Amer. J. Math.},
  FJOURNAL = {American Journal of Mathematics},
    VOLUME = {106},
      YEAR = {1984},
    NUMBER = {6},
     PAGES = {1389--1422},
      ISSN = {0002-9327},
   MRCLASS = {28D05 (58D20 58F11)},
  MRNUMBER = {765584},
MRREVIEWER = {Andr\'{e}s del Junco},
       DOI = {10.2307/2374398},
       URL = {https://doi.org/10.2307/2374398},
}

@unpublished {KLMMS-1,
	AUTHOR = {Kupers, A. and Lemannm, E. and Malkiewich, C and Miller, J. and Sroka, R. J.},
	TITLE = {Scissors automorphism groups and their homology},
	NOTE = {https://arxiv.org/abs/2408.08081},
	YEAR = {2024},
}

@article {li,
    AUTHOR = {Li, Xin},
     TITLE = {Ample groupoids, topological full groups, algebraic {K}-theory
              spectra and infinite loop spaces},
   JOURNAL = {Forum Math. Pi},
  FJOURNAL = {Forum of Mathematics. Pi},
    VOLUME = {13},
      YEAR = {2025},
     PAGES = {Paper No. e9, 56},
      ISSN = {2050-5086},
   MRCLASS = {20J05 (19D23 22A22 46L05 57M07)},
  MRNUMBER = {4866416},
MRREVIEWER = {Bj\o rn\ Ian\ Dundas},
       DOI = {10.1017/fmp.2024.31},
       URL = {https://doi.org/10.1017/fmp.2024.31},
}

@article{tanner,
  title={Rigidity, {G}enerators and {H}omology of {I}nterval {E}xchange {G}roups},
  author={Tanner, Owen},
  journal={arXiv preprint arXiv:2304.13691},
  year={2023}
}

@article{cz,
    AUTHOR = {Campbell, Jonathan A. and Zakharevich, Inna},
     TITLE = {Hilbert's third problem and a conjecture of {G}oncharov},
   JOURNAL = {Adv. Math.},
  FJOURNAL = {Advances in Mathematics},
    VOLUME = {451},
      YEAR = {2024},
     PAGES = {Paper No. 109757, 57},
      ISSN = {0001-8708,1090-2082},
   MRCLASS = {18G40 (19D55 19E99 52B45 55U10)},
  MRNUMBER = {4759410},
MRREVIEWER = {Primo\v z\ Moravec},
       DOI = {10.1016/j.aim.2024.109757},
       URL = {https://doi.org/10.1016/j.aim.2024.109757},
}

@article{miller_patzt_wilson,
  title={On rank filtrations of algebraic {K}-theory and {S}teinberg modules},
  author={Miller, Jeremy and Patzt, Peter and Wilson, Jennifer CH},
  journal={J. Europ. Math. Soc., Online First},
  year={2025}
}

@article{bgmmz,
    author = {Bohmann, Anna Marie and Gerhardt, Teena and Malkiewich, Cary and Merling, Mona and Zakharevich, Inna},
    title = "{A Trace Map on Higher Scissors Congruence Groups}",
    journal = {International Mathematics Research Notices},
    year = {2024},
    month = {08},
    issn = {1073-7928},
    doi = {10.1093/imrn/rnae153},
    url = {https://doi.org/10.1093/imrn/rnae153},
    eprint = {https://academic.oup.com/imrn/advance-article-pdf/doi/10.1093/imrn/rnae153/58939140/rnae153.pdf},
}

@article{sevkm,
    AUTHOR = {Dutour Sikiri\'c, Mathieu and Elbaz-Vincent, Philippe and
              Kupers, Alexander and Martinet, Jacques},
     TITLE = {Voronoi complexes in higher dimensions, cohomology of
              {$GL_N(\Bbb Z)$} for {$N\geqslant8$} and the triviality of
              {$K_8(\Bbb Z)$}},
   JOURNAL = {J. Inst. Math. Jussieu},
  FJOURNAL = {Journal of the Institute of Mathematics of Jussieu. JIMJ.
              Journal de l'Institut de Math\'ematiques de Jussieu},
    VOLUME = {25},
      YEAR = {2026},
    NUMBER = {1},
     PAGES = {565--590},
      ISSN = {1474-7480,1475-3030},
   MRCLASS = {20G10 (11F06 11F75 11H55 19D50 20J06)},
  MRNUMBER = {5018892},
       DOI = {10.1017/S1474748025101394},
       URL = {https://doi.org/10.1017/S1474748025101394},
}

@inproceedings {quillen_fg,
    AUTHOR = {Quillen, Daniel},
     TITLE = {Finite generation of the groups {$K_{i}$} of rings of
              algebraic integers},
 BOOKTITLE = {Algebraic {$K$}-theory, {I}: {H}igher {$K$}-theories ({P}roc.
              {C}onf., {B}attelle {M}emorial {I}nst., {S}eattle, {W}ash.,
              1972)},
    SERIES = {Lecture Notes in Math., Vol. 341},
     PAGES = {179--198},
 PUBLISHER = {Springer, Berlin},
      YEAR = {1973},
   MRCLASS = {18F25},
  MRNUMBER = {0349812},
MRREVIEWER = {T. Y. Lam},
}

@book {weibel,
    AUTHOR = {Weibel, Charles A.},
     TITLE = {The {$K$}-book},
    SERIES = {Graduate Studies in Mathematics},
    VOLUME = {145},
      NOTE = {An introduction to algebraic $K$-theory},
 PUBLISHER = {American Mathematical Society, Providence, RI},
      YEAR = {2013},
     PAGES = {xii+618},
      ISBN = {978-0-8218-9132-2},
   MRCLASS = {19-01},
  MRNUMBER = {3076731},
MRREVIEWER = {L. N. Vaserstein},
       DOI = {10.1090/gsm/145},
       URL = {https://doi.org/10.1090/gsm/145},
}

@article {cath1,
    AUTHOR = {Cathelineau, Jean-Louis},
     TITLE = {Scissors congruences and the bar and cobar constructions},
   JOURNAL = {J. Pure Appl. Algebra},
  FJOURNAL = {Journal of Pure and Applied Algebra},
    VOLUME = {181},
      YEAR = {2003},
    NUMBER = {2-3},
     PAGES = {141--179},
      ISSN = {0022-4049},
   MRCLASS = {52B45 (20J05)},
  MRNUMBER = {1975297},
MRREVIEWER = {Ioannis Emmanouil},
       DOI = {10.1016/S0022-4049(02)00333-X},
       URL = {https://doi.org/10.1016/S0022-4049(02)00333-X},
}

@article {cath3,
    AUTHOR = {Cathelineau, Jean-Louis},
     TITLE = {Homology stability for orthogonal groups over algebraically
              closed fields},
   JOURNAL = {Ann. Sci. \'{E}cole Norm. Sup. (4)},
  FJOURNAL = {Annales Scientifiques de l'\'{E}cole Normale Sup\'{e}rieure. Quatri\`eme
              S\'{e}rie},
    VOLUME = {40},
      YEAR = {2007},
    NUMBER = {3},
     PAGES = {487--517},
      ISSN = {0012-9593},
   MRCLASS = {20G10 (17B45 17B55)},
  MRNUMBER = {2493389},
MRREVIEWER = {Christopher P. Bendel},
       DOI = {10.1016/j.ansens.2007.03.001},
       URL = {https://doi.org/10.1016/j.ansens.2007.03.001},
}

@article {rognes_rank,
    AUTHOR = {Rognes, John},
     TITLE = {A spectrum level rank filtration in algebraic {$K$}-theory},
   JOURNAL = {Topology},
  FJOURNAL = {Topology. An International Journal of Mathematics},
    VOLUME = {31},
      YEAR = {1992},
    NUMBER = {4},
     PAGES = {813--845},
      ISSN = {0040-9383},
   MRCLASS = {19D99 (18F25 55P42)},
  MRNUMBER = {1191383},
MRREVIEWER = {Roland Schw\"{a}nzl},
       DOI = {10.1016/0040-9383(92)90012-7},
       URL = {https://doi.org/10.1016/0040-9383(92)90012-7},
}

@misc {rudenko,
    TITLE = {Notes on scissors congruence, based on a course by {D}aniil {R}udenko},
    AUTHOR = {Danny Calegari},
    NOTE = {http:/\!\!/math.uchicago.edu/{\texttildelow}dannyc/notes/scissors.pdf},
    EPRINT = {http://math.uchicago.edu/~dannyc/notes/scissors.pdf},
    URL = {http://math.uchicago.edu/~dannyc/notes/scissors.pdf},
	YEAR = {2022}
}

@article {ds2,
    AUTHOR = {Dupont, Johan L. and Sah, Chih-Han},
     TITLE = {Homology of {E}uclidean groups of motions made discrete and
              {E}uclidean scissors congruences},
   JOURNAL = {Acta Math.},
  FJOURNAL = {Acta Mathematica},
    VOLUME = {164},
      YEAR = {1990},
    NUMBER = {1-2},
     PAGES = {1--27},
      ISSN = {0001-5962},
   MRCLASS = {22E41 (19D55 22E40 52B45 57Q99)},
  MRNUMBER = {1037596},
MRREVIEWER = {Jean-Louis Cathelineau},
       DOI = {10.1007/BF02392750},
       URL = {https://doi.org/10.1007/BF02392750},
}

@article {ds1,
    AUTHOR = {Dupont, Johan L. and Sah, Chih Han},
     TITLE = {Scissors congruences. {II}},
   JOURNAL = {J. Pure Appl. Algebra},
  FJOURNAL = {Journal of Pure and Applied Algebra},
    VOLUME = {25},
      YEAR = {1982},
    NUMBER = {2},
     PAGES = {159--195},
      ISSN = {0022-4049},
   MRCLASS = {53C40 (51M10 51M20 52A40 53A35)},
  MRNUMBER = {662760},
MRREVIEWER = {O. V. Shvartsman},
       DOI = {10.1016/0022-4049(82)90035-4},
       URL = {https://doi.org/10.1016/0022-4049(82)90035-4},
}

@article {dupont_parry_sah,
    AUTHOR = {Dupont, Johan L. and Parry, Walter and Sah, Chih-Han},
     TITLE = {Homology of classical {L}ie groups made discrete. {II}.
              {$H_2,H_3,$} and relations with scissors congruences},
   JOURNAL = {J. Algebra},
  FJOURNAL = {Journal of Algebra},
    VOLUME = {113},
      YEAR = {1988},
    NUMBER = {1},
     PAGES = {215--260},
      ISSN = {0021-8693},
   MRCLASS = {22E40 (18F25 19C09 20G10 20G20)},
  MRNUMBER = {928063},
MRREVIEWER = {K. Vogtmann},
       DOI = {10.1016/0021-8693(88)90191-3},
       URL = {https://doi.org/10.1016/0021-8693(88)90191-3},
}

@article {jessen72,
    AUTHOR = {Jessen, B{\o}rge},
     TITLE = {Zur {A}lgebra der {P}olytope},
   JOURNAL = {Nachr. Akad. Wiss. G\"ottingen Math.-Phys. Kl. II},
  FJOURNAL = {Nachrichten der Akademie der Wissenschaften zu G\"ottingen.
              II. Mathematisch-Physikalische Klasse},
      YEAR = {1972},
     PAGES = {47--53},
      ISSN = {0065-5295},
   MRCLASS = {52A25 (50B30)},
  MRNUMBER = {353150},
MRREVIEWER = {H.\ S. M. Coxeter},
}

@article {jessen_karpf_thorup,
    AUTHOR = {Jessen, B{\o}rge and Karpf, J{\o}rgen and Thorup, Anders},
     TITLE = {Some functional equations in groups and rings},
   JOURNAL = {Math. Scand.},
  FJOURNAL = {Mathematica Scandinavica},
    VOLUME = {22},
      YEAR = {1968},
     PAGES = {257--265 (1969)},
      ISSN = {0025-5521},
   MRCLASS = {39.38 (13.00)},
  MRNUMBER = {284739},
MRREVIEWER = {Chr. U. Jensen},
       DOI = {10.7146/math.scand.a-10889},
       URL = {https://doi.org/10.7146/math.scand.a-10889},
}

@article {jessen_thorup,
    AUTHOR = {Jessen, B{\o}rge and Thorup, Anders},
     TITLE = {The algebra of polytopes in affine spaces},
   JOURNAL = {Math. Scand.},
  FJOURNAL = {Mathematica Scandinavica},
    VOLUME = {43},
      YEAR = {1978},
    NUMBER = {2},
     PAGES = {211--240 (1979)},
      ISSN = {0025-5521},
   MRCLASS = {52A25},
  MRNUMBER = {531303},
MRREVIEWER = {G. T. Sallee},
       DOI = {10.7146/math.scand.a-11777},
       URL = {https://doi.org/10.7146/math.scand.a-11777},
}

@article {goncharov,
    AUTHOR = {Goncharov, Alexander},
     TITLE = {Volumes of hyperbolic manifolds and mixed {T}ate motives},
   JOURNAL = {J. Amer. Math. Soc.},
  FJOURNAL = {Journal of the American Mathematical Society},
    VOLUME = {12},
      YEAR = {1999},
    NUMBER = {2},
     PAGES = {569--618},
      ISSN = {0894-0347},
   MRCLASS = {19F27 (11G99 11R70 57M50)},
  MRNUMBER = {1649192},
MRREVIEWER = {Burt Totaro},
       DOI = {10.1090/S0894-0347-99-00293-3},
       URL = {https://doi.org/10.1090/S0894-0347-99-00293-3},
}

@article {zylev_65,
    AUTHOR = {Zylev, V. B.},
     TITLE = {On the equi-dissectability of two equi-completable polytopes},
   JOURNAL = {Dokl. Akad. Nauk SSSR},
  FJOURNAL = {Doklady Akademii Nauk SSSR},
    VOLUME = {161},
      YEAR = {1965},
     PAGES = {515--516},
      ISSN = {0002-3264},
   MRCLASS = {50.90 (52.25)},
  MRNUMBER = {0177347},
MRREVIEWER = {E. M. Bruins},
}

@article {zak_k1,
    AUTHOR = {Zakharevich, Inna},
     TITLE = {On {$K_1$} of an assembler},
   JOURNAL = {J. Pure Appl. Algebra},
  FJOURNAL = {Journal of Pure and Applied Algebra},
    VOLUME = {221},
      YEAR = {2017},
    NUMBER = {7},
     PAGES = {1867--1898},
      ISSN = {0022-4049},
   MRCLASS = {19B99 (14C35 18F30)},
  MRNUMBER = {3614978},
MRREVIEWER = {L. N. Vaserstein},
       DOI = {10.1016/j.jpaa.2016.12.028},
       URL = {https://doi.org/10.1016/j.jpaa.2016.12.028},
}

@article {zak_assemblers,
    AUTHOR = {Zakharevich, Inna},
     TITLE = {The {$K$}-theory of assemblers},
   JOURNAL = {Adv. Math.},
  FJOURNAL = {Advances in Mathematics},
    VOLUME = {304},
      YEAR = {2017},
     PAGES = {1176--1218},
      ISSN = {0001-8708},
   MRCLASS = {19M05 (14C35 18F25 18F30 55N15)},
  MRNUMBER = {3558230},
MRREVIEWER = {Kevin Hutchinson},
       DOI = {10.1016/j.aim.2016.08.045},
       URL = {https://doi.org/10.1016/j.aim.2016.08.045},
}

@article {dehn,
    AUTHOR = {Dehn, M.},
     TITLE = {Ueber den {R}auminhalt},
   JOURNAL = {Math. Ann.},
  FJOURNAL = {Mathematische Annalen},
    VOLUME = {55},
      YEAR = {1901},
    NUMBER = {3},
     PAGES = {465--478},
      ISSN = {0025-5831},
   MRCLASS = {DML},
  MRNUMBER = {1511157},
       DOI = {10.1007/BF01448001},
       URL = {https://doi.org/10.1007/BF01448001},
}

@article {sydler,
    AUTHOR = {Sydler, J.-P.},
     TITLE = {Conditions n\'{e}cessaires et suffisantes pour l'\'{e}quivalence des
              poly\`edres de l'espace euclidien \`a trois dimensions},
   JOURNAL = {Comment. Math. Helv.},
  FJOURNAL = {Commentarii Mathematici Helvetici},
    VOLUME = {40},
      YEAR = {1965},
     PAGES = {43--80},
      ISSN = {0010-2571},
   MRCLASS = {50.90 (52.10)},
  MRNUMBER = {192407},
MRREVIEWER = {H. Freudenthal},
       DOI = {10.1007/BF02564364},
       URL = {https://doi.org/10.1007/BF02564364},
}

@book {dupont_book,
    AUTHOR = {Dupont, Johan L.},
     TITLE = {Scissors congruences, group homology and characteristic
              classes},
    SERIES = {Nankai Tracts in Mathematics},
    VOLUME = {1},
 PUBLISHER = {World Scientific Publishing Co., Inc., River Edge, NJ},
      YEAR = {2001},
     PAGES = {viii+168},
      ISBN = {981-02-4507-6; 981-02-4508-4},
   MRCLASS = {52B45 (57M27 57M50 58J28)},
  MRNUMBER = {1832859},
MRREVIEWER = {H. R. Gluck},
       DOI = {10.1142/9789812810335},
       URL = {https://doi.org/10.1142/9789812810335},
}

@article {dupont_82,
    AUTHOR = {Dupont, Johan L.},
     TITLE = {Algebra of polytopes and homology of flag complexes},
   JOURNAL = {Osaka J. Math.},
  FJOURNAL = {Osaka Journal of Mathematics},
    VOLUME = {19},
      YEAR = {1982},
    NUMBER = {3},
     PAGES = {599--641},
      ISSN = {0030-6126},
   MRCLASS = {52A25 (20J05 57Q99)},
  MRNUMBER = {676240},
MRREVIEWER = {Chih-Han Sah},
       URL = {http://projecteuclid.org/euclid.ojm/1200775324},
}

@book {sah_79,
    AUTHOR = {Sah, C. H.},
     TITLE = {Hilbert's {T}hird {P}roblem: {S}cissors {C}ongruence},
    SERIES = {Research Notes in Mathematics},
    VOLUME = {33},
 PUBLISHER = {Pitman (Advanced Publishing Program), Boston, Mass.-London},
      YEAR = {1979},
     PAGES = {vi+188},
      ISBN = {0-273-08426-7},
   MRCLASS = {51M20 (20H15 51M25 57S30)},
  MRNUMBER = {554756},
MRREVIEWER = {E. Vinberg},
}

@article {calc3,
	AUTHOR = {Goodwillie, Thomas G.},
	TITLE = {Calculus. {III}. {T}aylor series},
	JOURNAL = {Geom. Topol.},
	FJOURNAL = {Geometry and Topology},
	VOLUME = {7},
	YEAR = {2003},
	PAGES = {645--711},
	ISSN = {1465-3060},
	MRCLASS = {55P65 (19D10 55P42 55U35)},
	MRNUMBER = {2026544},
	MRREVIEWER = {Daniel C. Isaksen},
	DOI = {10.2140/gt.2003.7.645},
	URL = {https://doi.org/10.2140/gt.2003.7.645},
}

@article {calc2,
	AUTHOR = {Goodwillie, Thomas G.},
	TITLE = {Calculus. {II}. {A}nalytic functors},
	JOURNAL = {$K$-Theory},
	FJOURNAL = {$K$-Theory. An Interdisciplinary Journal for the Development,
	Application, and Influence of $K$-Theory in the Mathematical
	Sciences},
	VOLUME = {5},
	YEAR = {1991/92},
	NUMBER = {4},
	PAGES = {295--332},
	ISSN = {0920-3036},
	MRCLASS = {55P65 (18G55 19D10 55Q30)},
	MRNUMBER = {1162445},
	MRREVIEWER = {Roland Schw\"{a}nzl},
	DOI = {10.1007/BF00535644},
	URL = {https://doi.org/10.1007/BF00535644},
}

@book {munson_volic,
    AUTHOR = {Munson, Brian A. and Voli\'{c}, Ismar},
     TITLE = {Cubical homotopy theory},
    SERIES = {New Mathematical Monographs},
    VOLUME = {25},
 PUBLISHER = {Cambridge University Press, Cambridge},
      YEAR = {2015},
     PAGES = {xv+631},
      ISBN = {978-1-107-03025-1},
   MRCLASS = {55P99 (16B50 18G30 55U10 57Q45 57R19)},
  MRNUMBER = {3559153},
MRREVIEWER = {Rui Miguel Saramago},
       DOI = {10.1017/CBO9781139343329},
       URL = {https://doi.org/10.1017/CBO9781139343329},
}

@article{mm02,
    AUTHOR = {Mandell, M. A. and May, J. P.},
    TITLE = {Equivariant orthogonal spectra and {$S$}-modules},
    JOURNAL = {Mem. Amer. Math. Soc.},
    FJOURNAL = {Memoirs of the American Mathematical Society},
    VOLUME = {159},
    YEAR = {2002},
    NUMBER = {755},
    PAGES = {x+108},
    ISSN = {0065-9266},
    MRCLASS = {55P91 (18E30 55P42 55P43 55P48)},
    MRNUMBER = {1922205},
    MRREVIEWER = {J. P. C. Greenlees},
    DOI = {10.1090/memo/0755},
    URL = {https://doi.org/10.1090/memo/0755},
}

@article{mmss,
    AUTHOR = {Mandell, M. A. and May, J. P. and Schwede, S. and Shipley, B.},
    TITLE = {Model categories of diagram spectra},
    JOURNAL = {Proc. London Math. Soc. (3)},
    FJOURNAL = {Proceedings of the London Mathematical Society. Third Series},
    VOLUME = {82},
    YEAR = {2001},
    NUMBER = {2},
    PAGES = {441--512},
    ISSN = {0024-6115},
    MRCLASS = {55P42 (18A25 18E30 55U35)},
    MRNUMBER = {1806878},
    MRREVIEWER = {Mark Hovey},
    DOI = {10.1112/S0024611501012692},
    URL = {https://doi.org/10.1112/S0024611501012692},
}

@incollection{1126,
	title={Algebraic {$K$}-theory of spaces},
	author={Waldhausen, F.},
	booktitle={Algebraic and geometric topology},
	pages={318--419},
	year={1985},
	publisher={Springer}
}

@article{schwede_ss,
	title={An untitled book project about symmetric spectra (version 3)},
	author={Schwede, Stefan},
	journal={preprint},
	year={2012},
	URL = {http://www.math.uni-bonn.de/people/schwede/SymSpec-v3.pdf},
	howpublished = "\url{http://www.math.uni-bonn.de/people/schwede/SymSpec-v3.pdf}"
}

@book {goerss_jardine,
    AUTHOR = {Goerss, Paul G. and Jardine, John F.},
     TITLE = {Simplicial homotopy theory},
    SERIES = {Progress in Mathematics},
    VOLUME = {174},
 PUBLISHER = {Birkh\"{a}user Verlag, Basel},
      YEAR = {1999},
     PAGES = {xvi+510},
      ISBN = {3-7643-6064-X},
   MRCLASS = {55U10 (18G55 55-01 55Pxx)},
  MRNUMBER = {1711612},
MRREVIEWER = {R. M. Vogt},
       DOI = {10.1007/978-3-0348-8707-6},
       URL = {https://doi.org/10.1007/978-3-0348-8707-6},
}

@article {ss_modules,
    AUTHOR = {Schwede, Stefan and Shipley, Brooke},
     TITLE = {Stable model categories are categories of modules},
   JOURNAL = {Topology},
  FJOURNAL = {Topology. An International Journal of Mathematics},
    VOLUME = {42},
      YEAR = {2003},
    NUMBER = {1},
     PAGES = {103--153},
      ISSN = {0040-9383},
   MRCLASS = {55U35 (18G55 55P42 55P43)},
  MRNUMBER = {1928647},
MRREVIEWER = {Mark Hovey},
       DOI = {10.1016/S0040-9383(02)00006-X},
       URL = {https://doi.org/10.1016/S0040-9383(02)00006-X},
}

@article {shipley_algebras,
    AUTHOR = {Shipley, Brooke},
     TITLE = {{$H\Bbb Z$}-algebra spectra are differential graded algebras},
   JOURNAL = {Amer. J. Math.},
  FJOURNAL = {American Journal of Mathematics},
    VOLUME = {129},
      YEAR = {2007},
    NUMBER = {2},
     PAGES = {351--379},
      ISSN = {0002-9327},
   MRCLASS = {55P42 (18G55 55P20 55P43 55U35)},
  MRNUMBER = {2306038},
MRREVIEWER = {Tyler D. Lawson},
       DOI = {10.1353/ajm.2007.0014},
       URL = {https://doi.org/10.1353/ajm.2007.0014},
}

@incollection {haupt_survey,
    AUTHOR = {Ranicki, A. A.},
     TITLE = {On the {H}auptvermutung},
 BOOKTITLE = {The {H}auptvermutung book},
    SERIES = {$K$-Monogr. Math.},
    VOLUME = {1},
     PAGES = {3--31},
 PUBLISHER = {Kluwer Acad. Publ., Dordrecht},
      YEAR = {1996},
   MRCLASS = {57Q25},
  MRNUMBER = {1434101},
MRREVIEWER = {Oliver Attie},
       DOI = {10.1007/978-94-017-3343-4\_1},
       URL = {https://doi.org/10.1007/978-94-017-3343-4_1},
}

\end{document}